\documentclass[a4paper,8pt]{article}
\usepackage{framed,enumitem} 
\usepackage[arrow, matrix, curve]{xy}
\usepackage{pdfpages}
\usepackage{floatrow}
\usepackage[utf8]{inputenc}
\usepackage{a4wide}
\usepackage[english,ngerman]{babel}
\usepackage[abs]{overpic}
\PassOptionsToPackage{english}{babel}
\usepackage{mathtools}
\usepackage{amsmath, amssymb, amsthm, amsfonts}
\usepackage{physics}
\usepackage{graphicx}
\usepackage{layout}
\usepackage{fancyhdr}
\usepackage{setspace}
\usepackage{hyperref}

\usepackage{ dsfont }
\usepackage{ wasysym }
\usepackage{siunitx}
\usepackage{tikz}
\usepackage{tkz-euclide}
\usepackage{cite}
\usetikzlibrary{quotes,angles}
\usepackage{verbatim}
\usepackage{color}
\usepackage{cite}
\usepackage[toc,page]{appendix}
\usepackage{multirow} 
\usepackage{url}
\usepackage{bm}
\usepackage{mwe}
\usepackage{aligned-overset}

\usepackage{algorithm}
\usepackage{algpseudocode}
\usepackage{aliascnt}
\usepackage{placeins}

\newcommand{\Arg}{\operatorname {Arg}}
\newcommand{\RE}{\operatorname {Re}}

\newcommand{\RR}{{\mathbb R}}

\newcommand{\CC}{\mathbb{C}}

\algnewcommand\Load{\textbf{Load} }
\algnewcommand\Init{\textbf{Init}}
\algnewcommand\Break{\textbf{break}}

\usepackage{authblk}

\theoremstyle{plain}
\newtheorem{theorem}{Theorem}[section]

\newtheorem{dexa}[theorem]{Data Example}
\newtheorem{rem}[theorem]{Remark}

\newaliascnt{lemma}{theorem}
\newtheorem{lemma}[lemma]{Lemma}
\aliascntresetthe{lemma}
\newtheorem{corollary}[theorem]{Corollary}
\newtheorem{definition}[theorem]{Definition}
\newtheorem{assumption}[theorem]{Assumption}

\title{Drift Models on Complex Projective Space for Electron-Nuclear Double Resonance}
\author{Henrik Wiechers$^1$,  Markus Zobel$^1$, Marina Bennati$^{2,3}$, Igor Tkach$^2$,\\ Benjamin Eltzner$^2$,  Stephan Huckemann$^1$, Yvo Pokern$^{4}$}
\affil{{ 
\normalsize{$^1$Felix-Bernstein-Institute for Mathematical Statistics, Georg-August-University,}\\
\normalsize{G\"ottingen, 37077 G\"ottingen, Germany.}\\
\normalsize{$^2$Max Planck Institute for Multidisciplinary Sciences,}\\
\normalsize{37077 G\"ottingen, Germany.}\\ 
\normalsize{$^3$Department of Chemistry, Georg-August University of Göttingen,}\\
\normalsize{Tammannstr. 2, Göttingen, Germany}\\
\normalsize{$^4$Department of Statistical Science, University College London,}\\
\normalsize{London WC1E 6BT, United Kingdom.}\\
}}
\date{}

\newcommand*\diff{\mathop{}\!\mathrm{d}}

\renewcommand{\appendixname}{SI}
\renewcommand{\appendixtocname}{Supplementary Information }
\newcommand\numberthis{\addtocounter{equation}{1}\tag{\theequation}}
\newcommand{\Ffun}{\mathcal{F}} 
\newcommand{\Scomp}{\mathcal{S}^{2N-1}} 
\newcommand{\rhofun}{\rho} 
\newcommand{\bul}{\bullet} 
\newcommand{\dia}{\diamond} 

\newcommand{\Mf}{{M}} 
\newcommand{\vf}{\mathrm{vec}} 
\newcommand{\vft}[1]{\vf \left( #1 \right)}

\newcommand{\vect}[1]{\mathrm{vec} \left( #1 \right)}
\newcommand{\vectn}[1]{\widetilde{\mathrm{vec}} \left( #1 \right)}
\newcommand{\comp}[1]{\mathfrak{c} \left( #1 \right)}
\newcommand{\compn}[1]{\tilde{\mathfrak{c}} \left( #1 \right)}
\newcommand\normp[1]{\left\lVert#1\right\rVert_P}
\newcommand{\spr}[2]{\left\langle #1, #2 \right\rangle}

\newcommand{\Covmat}[2]{\text{Cov}\qty[#1,#2]} 

\DeclareMathOperator*{\argmax}{arg\,max}

\newcommand*{\defeq}{\mathrel{\vcenter{\baselineskip0.5ex \lineskiplimit0pt \hbox{\scriptsize.}\hbox{\scriptsize.}}}=}
\newcommand*{\eqdef}{=\mathrel{\vcenter{\baselineskip0.5ex \lineskiplimit0pt \hbox{\scriptsize.}\hbox{\scriptsize.}}}}
\newcommand{\Var}[1]{\text{Var}\qty[#1]}  
\makeatletter
\newcommand\thefontsize{{ The current font size is: \f@size pt\par}}
\makeatother

\addto\extrasenglish{%
  \renewcommand{\appendixname}{SI}
}

\begin{document}

\selectlanguage{english}
\setlength{\headheight}{28pt}

\lhead{\fancyplain{}{\thepage}}
\chead{}
\rhead{\fancyplain{}{\textit{\leftmark}}}
\lfoot{}
\cfoot{}
\rfoot{}
\onehalfspacing

\maketitle

\begin{abstract}
ENDOR spectroscopy is an important tool to determine the complicated three-dimensional structure of biomolecules and in particular enables measurements of intramolecular distances.
Usually, spectra are determined by averaging the data matrix, which does not take into account the significant thermal drifts that occur in the measurement process. 
In contrast, we present an asymptotic analysis for the homoscedastic drift model, a pioneering parametric model that achieves striking model fits in practice and allows both hypothesis testing and confidence intervals for spectra. 
The ENDOR spectrum and an orthogonal component are modeled as an element of complex projective space, and formulated in the framework of generalized Fr\'echet means. 
To this end, two general formulations of strong consistency for set-valued Fr\'echet means are extended and subsequently applied to the homoscedastic drift model to prove strong consistency. Building on this, central limit theorems for the ENDOR spectrum are shown.
Furthermore, we extend applicability by taking into account a phase noise contribution leading to the heteroscedastic drift model. Both drift models offer improved signal-to-noise ratio over pre-existing models. 
\end{abstract}
\newpage
\tableofcontents
\newpage
\maketitle

\section{Introduction}

One of the main objectives of structural biology is to understand the complicated three-dimensional structure of 
biomolecules, and thus provide meaningful links between structure and functionality.
In particular, this information can be used in the field of structure-based drug design, see for example  \cite{Anderson_drug_design}.
There is a wide range of different methods to determine the structure, such as \emph{X-ray crystallography} (X-ray), \emph{cryogenic electron microscopy} (cryo-EM) and \emph{spectroscopic} methods. 
\emph{Nuclear magnetic resonance} (NMR) spectroscopy is possibly the most widely used spectroscopic method: it studies the interactions between the nuclei of a molecule using radio frequency (RF) pulses. \emph{Electron paramagnetic resonance} (EPR), on the other hand, studies the local environment and different kinds of interactions of the spins of unpaired electrons using microwave (MW) pulses. It can be more selective than NMR in that it targets only the tiny minority of unpaired electrons among the large number of electrons present in a biomolecule. Additionally, the larger gyromagnetic ratio of the electron compared to any magnetic nucleus usually leads to higher detection sensitivity and thus to better signal-to-noise ratio (SNR).
Electron Nuclear Double Resonance (ENDOR) spectroscopy \cite{Feher, Gemperle, Harmer} seeks to combine the advantages of EPR and NMR by interacting with both, nuclei and radical electrons, using both MW and RF pulses in a single experiment (see \autoref{sec: homoscedastic drift} for an accessible exposition of how this works). It should be emphasized that NMR, EPR and ENDOR differ in their domain of applicability, in particular in the range of distances between interacting spins, rather than one method being generally superior to another. 
Roughly, ENDOR's double resonance approach yields information on how the unpaired electron interacts with magnetic nuclei of a chosen kind (e.g. protons, deuterium nuclei or fluorine nuclei) and explores their environment.
Artificially inserting \emph{labels}, i.e. magnetic nuclei rarely present in biomolecules such as fluorine or deuterium, as well as radicals containing unpaired electrons that do not naturally occur in the biomolecule under study such as nitroxide radicals, allows highly specific measurements of intramolecular distances and orientations between selectable parts of the biomolecule, see \cite{Meyer2020}.

Prior to \cite{Pokern2021}, the standard approach \cite{Epel, Rizzato2014} for extracting ENDOR spectra from the recorded echo signals was equivalent to the \emph{averaging model} \cite{Pokern2021} whereby echo responses are simply averaged across a large number of replications of the ENDOR experiment and only the average response is processed further. However, as ENDOR experiments typically run for several hours and at low temperatures, significant thermal drifts over time occur in practice.
In \cite{Pokern2021}, the \emph{homoscedastic drift model} was introduced for ENDOR experiments at a microwave frequency of \SI{263}{GHz}, which uses the echo signals at each of the $N+1$ (with $N\in\mathbb{N}$) RF frequencies recorded in $B\in\mathbb{N}$ batches over time in a data matrix $Y\in\mathbb{C}^{B\times (N+1)}$. This model accounts for thermal drift by decomposing the data matrix accounting separately for signal drift and spectrum. It is the first of its kind in the field of ENDOR spectroscopy and, relative to common practice in applied statistics, achieves surprisingly good model fit that is maintained across a number of chemical compounds in follow-up studies,  cf. \cite{Pokern2021, Hiller2022, Wiechers2023}, yielding improved SNRs relative to the averaging model.
The homoscedastic drift model enables the application of the parametric bootstrap, which in turn enables hypothesis testing  and confidence intervals for the spectra:
In \cite{Pokern2021}, a flatness and a difference test were introduced and performed, which together confirmed unequivocally the presence of broad features that were suspected on visual inspection.
\cite{Wiechers2023} utilizes the spectral uncertainties provided by the drift model to determine stochastic errors in the estimation of physical parameters from which intramolecular distances can be determined.
The parameter of greatest applied interest in the homoscedastic drift model, $\kappa$, is complex-valued and contains both the ENDOR spectrum as well as an orthogonal component containing a resonance artefact.  It is standardized so that $\sum_{\nu=0}^N \kappa_\nu = 0$ and $\sum_{\nu=0}^N |\kappa_\nu|^2=1$. Additionally, the spectrum $I$ is extracted in a step following MLE estimation of $\kappa$ by selecting a direction in the complex plane that contains the spectrum rather than the resonance artefact based on application-driven criteria so that $I=\Re{\exp\left(i\lambda^\mathrm{opt} \right)\kappa }$ holds for some $\lambda^\mathrm{opt}\in[0,2\pi]$ which is determined from $\kappa$ alone. Indeed, we will show that rotation of $\kappa$ in the complex plane leaves the spectrum $I$ invariant and, thus, it is the application that drives us to consider the complex projective space $\mathbb{C}P^{N-1}$ as the appropriate parameter space in this estimation problem.

This paper addresses two main challenges:  

Firstly, in order to justify the use of the above methods, we will address the asymptotic theory of ENDOR spectra in this paper. More precisely, both strong consistency  and a central limit theorem (CLT) for the parameter $\kappa$ are proved in the limit of large numbers of batches $B$.

To this end, the theory of strong consistency of generalized Fr\'echet means is extended in \autoref{sec: strong consistency} and applied in \autoref{sec: Drift model: Strong consistency and central limit theorem}. Fr\'echet means (introduced by \cite{frechet1948elements}) take the notion of arithmetic mean to the non-Euclidean setting, and generalized Fr\'echet means are non-Euclidean data descriptors that do not necessarily live in the data space, that arises naturally in our application and create challenges arising from their implicit definition and potentially set-valued nature. 

We furthermore establish a CLT for the ENDOR spectrum $I$ justifying the construction of confidence intervals for the ENDOR spectra at least in the case of known noise covariance and comment on the case of unknown noise covariance in \autoref{sec: Consistency Drift model unknown Sigma}.

Secondly, in \autoref{sec: heteroscedastic drift}, we extend the homoscedastic drift model to cover other microwave frequencies such as \SI{94}{GHz} for which EPR spectrometers with an ENDOR capability are more widely available. This necessitates generalizing the drift model to the heteroscedastic case. Given the presence of boundary maxima and the unsatisfactory performance of penalized methods, a carefully devised parametric extension of the homoscedastic drift model is found to work best yielding fairly good fit to the data and notable improvements in SNR.

\subsection{Merging Complex and Real Notation and Complex Projective Space}\label{sec: notation}
Switching conveniently between complex-valued and real-valued matrices, vectors and scalars, the following notation is used throughout the paper.  

For a complex number $z = x+iy \in \CC$ and a complex vector $(z_1,\ldots,z_N)^T \in \CC^N$ define
$$ \vect{z} \coloneqq \begin{pmatrix} x\\ y\end{pmatrix}\in \mathbb{R}^2
\,,\quad
\Mf(z)  \coloneqq \begin{pmatrix} x & - y \\ y & x \end{pmatrix} 
\,, \quad
\vect{\begin{array}{c} z_1\\ \vdots \\ z_N\end{array}} \coloneqq \begin{pmatrix} \vft{w_1}\\ \vdots \\ \vft{w_N}\end{pmatrix}\in \mathbb{R}^{2N}
$$ 
Conversely for real vectors $(x,y)^T \in \RR$ and $(r_1,\ldots,r_{2N})^T \in \RR^{2N}$, define
$$\comp{\begin{array}{c}x\\ y\end{array}}\coloneqq x + i y \in \mathbb{C}
\,,\quad 
\compn{\begin{array}{c}r_1\\ r_{2N}\end{array}}\coloneqq \begin{pmatrix} r_1 + i r_2 \\ \vdots \\ r_{2N-1} + i r_{2N}\end{pmatrix}  \in \mathbb{C}^N 
$$

The following Lemma summarizes basic rules, verified at once.

\begin{lemma}
    For $z,w \in \mathbb{C}$ we have
    \begin{enumerate}
        \item $\Mf(z)^T = \Mf(\bar{z}) $
        \item $\vf(zw) = \Mf(z) \vf(w) = \Mf(w) \vf(z)$
    \end{enumerate}
\end{lemma}

Further, for $z, w\in\mathbb{C}^N$ and $A\in \mathrm{SPD}(2)$ we define
\begin{align*}
   z  \dia_A w &\coloneqq \sum_{i=1}^N \Mf(z_i)^TA\, \Mf(w_i)\in\mathbb{R}^{2\times 2}, \qquad &z\bul_Aw &\coloneqq \sum_{n=1}^N\Mf(z_n)^TA\,\vf(w_n)\in \mathbb{R}^2,
\end{align*}
as well as a Mahalanobis inner product, norm and distance,
\begin{align*}
   \spr{z}{w}_A &\coloneqq \sum_{i=1}^N\vf(z_i)^TA\,\vf(w_i)\in \mathbb{R}, \qquad &||z||_A &\coloneqq \sqrt{\spr{z}{z}_A}\in \mathbb{R}\\
   d_A(z, w) &\coloneqq ||z-w||_A\in \mathbb{R}.\\
\end{align*}

Next, we introduce complex projective space. It is the space of \emph{complex directions} in $\CC^N$ that can be viewed as the space of real directions modulo the \emph{phase}
$$ \lambda = \Arg(re^{i\lambda}) \in [0,2\pi)$$ of a complex number $z=re^{i\lambda} \in \CC$.

For a complex column vector $z\in  \CC^N$, $z^T=(z_1,\ldots,z_N) $, its Hermitian conjugate is the row vector
$$ z^*:= (\overline{z_1},\ldots,\overline{z_N})\,.$$
With the unit sphere 
$$\Scomp := \{\kappa\in \CC^N: \kappa^*\kappa = 1\} $$
of real dimension $2N-1$, the complex projective space of complex dimension $N-1$ and real dimension $2N-2$ is
$$ \CC P^{N-1} := \Scomp/\sim\,,$$
where "$\sim$" denotes the equivalence relation 
$$\kappa\sim \tilde{\kappa}  \quad \Leftrightarrow \quad \exists \lambda \in \mathbb{R}, \quad e^{i\lambda}\tilde{\kappa}=\kappa. $$ Furthermore we define the equivalence class of $\kappa$ by $[\kappa]$.
The distance between $[\kappa],[\tilde{\kappa}] \in \CC P^{N-1}$ is defined by 
\begin{align*}
    d([\kappa],[\tilde{\kappa}] ) = \min_{\lambda \in \mathbb{R}} ||\kappa-e^{i\lambda}\tilde{\kappa} ||
\end{align*}
where $\kappa \in [\kappa], \tilde{\kappa}\in [\tilde{\kappa}]$ are arbitrary representatives. 

We say that $\kappa,\tilde{\kappa} \in \Scomp$ are in optimal position if $$d\left([\kappa],\left[\tilde{\kappa}\right])\right)= \|\kappa -\tilde{\kappa}\|\,.$$

\begin{lemma}\label{lem: opt-pos}
    For arbitrary $\kappa,\tilde{\kappa} \in \Scomp$ we have that they are in optimal position if $\tilde{\kappa}^* \kappa=0$, or else,
    $$ \kappa,\quad \frac{\tilde{\kappa}^* \kappa}{|\tilde{\kappa}^* \kappa|}\, \tilde{\kappa}$$ 
    are in optimal position. 
\end{lemma}

\begin{proof}
The assertion follows at once from
$$d\left([\kappa],\left[\tilde{\kappa}\right])\right)=\min_{\lambda\in \mathbb{R}} \left(\kappa-e^{i\lambda}\tilde{\kappa}\right)^* \left(\kappa-e^{i\lambda}\tilde{\kappa}\right) = \min_{\lambda\in \mathbb{R}}\left( 2- 2\RE\left(e^{i\lambda}\kappa^*\tilde{\kappa}\right)\right)\,.$$ 
\end{proof}

\section{Homoscedastic Drift Model} \label{sec: homoscedastic drift}
In this section, we selectively review those aspects of the ENDOR experiment that are necessary for the present work with more background available in \cite{Gemperle} and full experimental details in \cite{Pokern2021}. We then introduce the setting for the homoscedastic drift model from \cite{Pokern2021} in preparation for its asymptotic analysis.

In the ENDOR experiment, a sequence of MW and RF pulses is sent into a chemical sample that is placed in an external magnetic field with field strength $B_0$. The magnetic field strength $B_0$, as well as the MW frequency $\nu_{MW}$ and MW pulse lengths together determine the set of orientations relative to the external magnetic field of those molecules in the chemical sample that participate in the resonance experiment.
Typically, five different field strengths $B_0$ are used to select five different sets of orientations denoted as $g_x,g_{xy},g_y,g_{yz}$ and $g_z$. The microwave echo signal returned by the participating molecules in the chemical sample is recorded in two separate components: a component that is in phase with a reference MW signal constitutes the real part and a component whose phase is shifted by 90 degrees, known as 'in quadrature', constitutes the imaginary part. 
This echo signal is influenced by a RF pulse that is part of the pulse sequence. While the MW frequency is constant throughout the ENDOR experiment (we report measurements for $\nu_{MW}=263$\,GHz and, in \autoref{sec: heteroscedastic drift}, $\nu_{MW}=94$\,GHz), the RF frequency is varied in a pseudo-random sequence covering each of the RF frequencies $\{f_\nu:\,\nu\in \qty{0,\ldots,N}\}$, $N\in\mathbb{N}$ once. This is known as a \emph{scan}. 
Since the SNR in a single scan is very low, a number $S\in\mathbb{N}$ of scans are performed in succession which constitute a batch of measurements. The batches are enumerated by $b \in\{1,\ldots,B\}$. The resulting echo signals $X_{s,b,\nu}\in\mathbb{C}$ are summed up to form $Y_{b,\nu}=\sum_{s=1}^S X_{s,b,\nu}$. 
Here, $S$ is chosen large enough to yield a SNR sufficient to allow adjustment of experimental parameters based on a single batch $Y_{b,:}\defeq \left(Y_{b,0},\ldots,Y_{b,N}\right)^T$ but small enough for the thermal drift that affects 
phase and amplitude of the echo signal to be negligible. 
Thus, we obtain the data matrix $Y \in \mathbb{C}^{B\times (N+1)}$, and a sample data matrix is illustrated in Figure \ref{fig: raw data} of the \appendixtocname (\appendixname).
Prior to \cite{Pokern2021}, the standard approach \cite{Rizzato2014,Epel} to extract ENDOR spectra from the echo signal $Y$ was the \emph{averaging model}:
\begin{definition}[Averaging Model]\label{def: averaging}
In the averaging model, the batches are averaged according to 
\begin{align}
Z_\nu = \frac{1}{B} \sum_{b=1}^BY_{b, \nu}.\label{eq: Z}
\end{align}
In a second step, a \emph{phase correction}, i.e. a complex multiplication by $e^{i\lambda}$ with a manually tuned $\lambda\in [0, 2\pi)$ to obtain a real valued non-normalized spectrum $\tilde{I}=\Re(e^{i\lambda}Z)$ is applied followed by normalization to obtain the spectrum
\begin{align}
    I_\nu&=\frac{\tilde{I}_\nu - \min_{\nu'\in\{0,\ldots,N\}}\tilde{I}_{\nu'}}{\max_{\nu'\in\{0,\ldots,N\}}\tilde{I}_{\nu'} -\min_{\nu'\in\{0,\ldots,N\}}\tilde{I}_{\nu'}}.
\end{align}   
\end{definition}
In \cite{Pokern2021}, the statistical flaws of this approach were addressed. Firstly, normalization via $Z_\nu = \psi + \phi \kappa_\nu$ with $\psi \in\mathbb{C}, \, \phi\in\mathbb{R}_{\geq 0}, \kappa_\nu \in\mathbb{C}$ and imposing \begin{align}
\sum_{\nu=0}^{N} \kappa_\nu&\stackrel{!}{=}0 \label{eq: meanSpecZero}\\
\sum_{\nu=0}^{N} |\kappa_\nu|^2&\stackrel{!}{=}1 \label{eq: normSpecOne}
\end{align}
is less sensitive to outliers. Note that the condition \ref{eq: meanSpecZero} removes a complex degree of freedom, motivating our choice of $N+1$ rather than $N$ RF frequencies. Secondly, various algorithms for phase correction without potentially biased operator intervention were studied to obtain the spectrum that is now given by $I_{\nu} = \Re(\exp(i\lambda_\mathrm{opt})\kappa_{\nu})$. In this paper, we exclusively utilize the maximum method \cite{Wiechers2023}, in which $\lambda_\mathrm{opt} \in \argmax_{\lambda \in [0, \pi)} \norm{\Re(\exp(i\lambda )\kappa)}$ is chosen so that the norm of $I$ is maximal. In measurements where the spectrum $\hat{I}$ consists of little else than the central peak which carries no conformational information, the minimum method minimizing deviation of $\hat{\omega}$ from a parametric model of the wave has proven to be very effective in \cite{Pokern2021,Hiller2022}. 
In both methods, additionally, a sign flip is performed when required to ensure that the spectrum's central peak points in the positive direction, effectively optimizing $\lambda$ over $[0,2\pi]$. 

However, as ENDOR experiments often run for hours, in practice the aforementioned thermal drift can be substantial, see Figure \ref{fig: raw data} of the \appendixname.
Thus, in \cite{Pokern2021} the drift model was introduced, which allows for thermal drift of $\psi$ and $\phi$, decomposing the data matrix according to the homoscedastic drift model:
\begin{definition}[Homoscedastic Drift Model]\label{def: homosced}
The homoscedastic drift model is given by
\begin{align}
    Y_{b,\nu} = \psi_b + \phi_b \kappa_\nu + \epsilon_{b, \nu},\quad \vect{\epsilon_{b, \nu}} \stackrel{i.i.d.}{\sim} \mathcal{N}(0,\Sigma),~ 1\leq \nu \leq N,~1\leq b \leq B. \label{eq: homoModel}
\end{align}
\end{definition}
By way of interpretation, $\psi \in \mathbb{C}^B$ represents the signal from
electron paramagnetic resonance (i.e. what the echo signal would be if the RF pulse were absent) as well as a possible offset of the measurement apparatus, $\phi\in \mathbb{C}^B$ represents the magnitude and phase of the ENDOR effect, $\kappa\in \mathbb{C}^{N+1}$ comprises the ENDOR spectrum $I$ as well as an orthogonal component $\omega$ which we call the wave (see panels A and B in Figure \ref{fig:DriftModelFit}) and $\epsilon_{b,\nu}$ represents the experimental noise. 
We use the notation $\vect{\epsilon_{b, \nu}}= \begin{bmatrix}\Re{\epsilon_{b,\nu}}\\ \Im{\epsilon_{b,\nu}}\end{bmatrix}$ so that the noise components follow a bivariate normal distribution with positive definite symmetric covariance matrix $\Sigma \in \mathrm{SPD}(2)$. 

The condition \ref{eq: meanSpecZero} serves to eliminate non-identifiability due to 
$\tilde{\kappa}=\kappa+c$, $\tilde{\psi}=\psi-c\phi$ with $\tilde{\psi},\phi,\tilde{\kappa},\Sigma$ yielding the same $Y_{\nu,b}$ 
as $\psi,\phi,\kappa,\Sigma$ for any $c \in\mathbb{C}$.
Similarly, the condition \ref{eq: normSpecOne} eliminates non-identifiability due to $\tilde{\kappa}=r \kappa$, $\tilde{\phi}=r^{-1}\phi$ with $\psi,\tilde{\phi},\tilde{\kappa},\Sigma$ yielding the same $Y_{\nu,b}$ 
as $\psi,\phi,\kappa,\Sigma$ for any $r\in\mathbb{R}_{>0}$.

Maximum likelihood estimators $\hat{\kappa}, \hat{\psi}, \hat{\phi}, \hat{\Sigma}$ are calculated (see \cite{Pokern2021} and \autoref{sec: maximum likelihood estimators} for details) and in a second step, the estimated spectrum $\hat{I}=\Re(e^{i\lambda_\mathrm{opt}}\hat{\kappa})$ and the orthogonal component $\hat{\omega}=\Im(e^{i\lambda_\mathrm{opt}}\hat{\kappa})$ are extracted from $\hat{\kappa}$ using the maximum method. 
Additionally to the above mentioned size non-identifiability of $\kappa$, the maximum (minimum) method and optional sign-flip eliminate the  phase non-identifiability due to $\tilde{\kappa}=\alpha \kappa$, $\tilde{\phi}=\alpha^{-1} \phi$ yielding the same data distribution for $\psi,\phi,\kappa,\Sigma$ and $\psi,\tilde{\phi},\tilde{\kappa},\Sigma$ for all $\alpha \in \mathbb{C}$ with $|\alpha|=1$.

The following data example illustrates that the homoscedastic drift model, Definition \ref{def: homosced}, exhibits unusually good fit to experimental data at $\nu_\mathrm{MW}=263$\,GHz, and yields improved SNR compared to the averaging model, Definition \ref{def: averaging}. Confidence regions are computed and will be justified via the $B\rightarrow \infty$ asymptotics developed in \autoref{sec: CLT Drift model}. It also prepares for extension to the heteroscedastic drift model, Definition \ref{def: hetero}, for $\nu_\mathrm{MW}=94$\,GHz. 

\begin{dexa}[Homoscedastic Drift Model for \SI{263}{\GHz} dataset]
The maximum likelihood estimates obtained using \autoref{alg:hom} for the orientation $g_y$ from a chemical sample of the D2-$Y_{122}^{\bullet}$ E.~coli ribonucleotide reductase using the Davies pulse sequence, see \cite{Davies1974}, studied in \cite{Pokern2021} are presented in Figure \ref{fig:DriftModelFit}. This also includes point-wise confidence bands obtained via parametric bootstrap using $10000$ bootstrap samples. In simulating data for the bootstrap, an additive bias correction for $\hat{\Sigma}$ and a multiplicative bias correction for $\hat{\phi}$ were used owing to substantial bias in these estimators. This bias, which does not disappear with increasing batch number $B$, likely arises from omitting the randomness in $\phi$ from the model as will be set out in detail in Section \ref{sec: Consistency Drift model unknown Sigma}.
\begin{figure}[ht!]
  \centering
    \begin{minipage}{1\textwidth}
    \includegraphics[width=\textwidth]{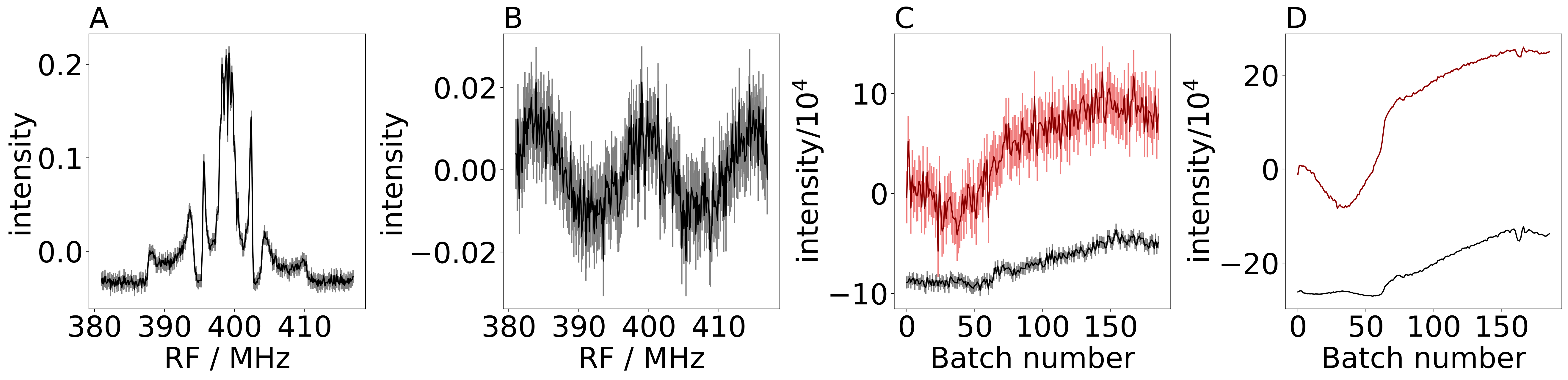}
  \end{minipage}
    \begin{minipage}{0.75\textwidth}
    \includegraphics[width=\textwidth]{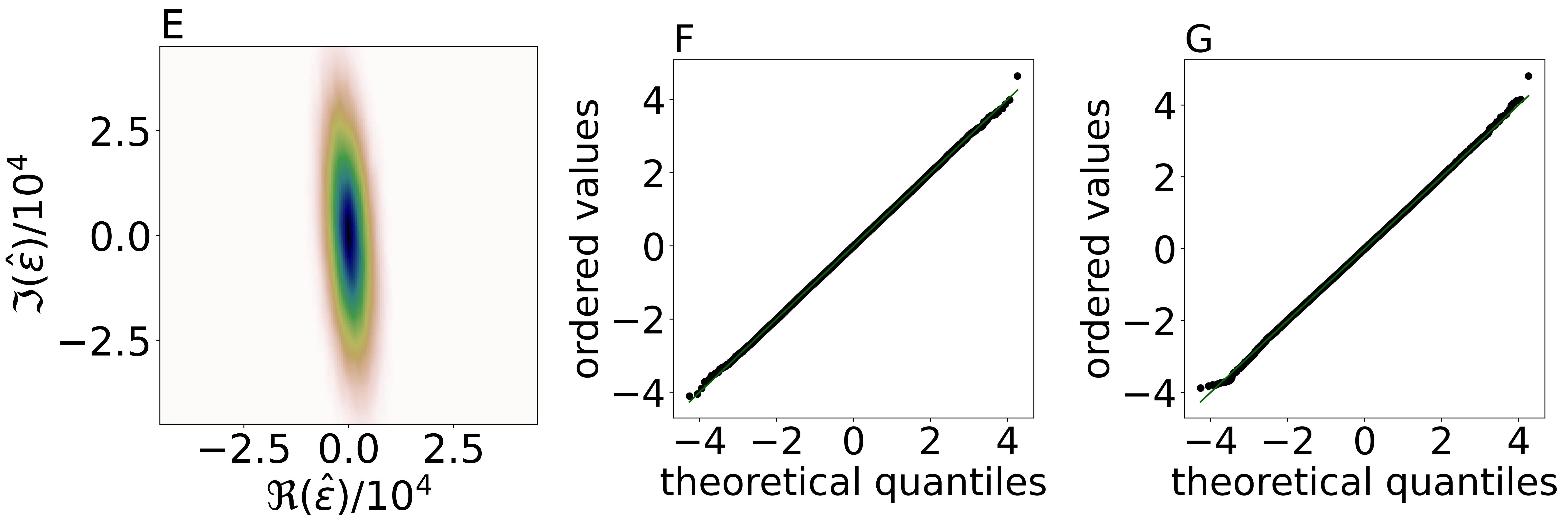}
  \end{minipage}
\caption{Applying the homoscedastic drift model (\ref{eq: homoModel}) to an ENDOR example measurement. 
Panel A displays the estimated spectrum $\hat{I}$, while panel B  displays the component $\hat{\omega}$ that is orthogonal to the estimated spectrum $\hat{I}$.
C and D show the real (black) and imaginary (red) components of $\hat{\phi}$ and $\hat{\psi}$, respectively.
Panel E displays the kernel-density-estimation of the complex residuals $\hat{\epsilon}_{b,\nu}$, while panels F and G depict q-q-plots for the real and imaginary components of the standardized residuals, respectively.
The corresponding raw data are plotted in Figure \ref{fig: raw data} in the \appendixname. 
}
\label{fig:DriftModelFit}
\end{figure}
A detailed analysis of all D2-$Y_{122}^{\bullet}$ measurements is presented in \autoref{sec: drift model plots}. The real and imaginary parts of the residuals $\hat{\epsilon}_{b,\nu}=Y_{b,\nu}-\hat{\psi}_b - \hat{\phi}_b\hat{\kappa}_\nu$ (shown in panels E, F, and G of Figure \ref{fig:DriftModelFit} for orientation $g_y$) for all orientations are examined for goodness of fit using a Kolmogorov-Smirnov test, see \cite{Ford2007}.
The resulting p-values are provided in Table \ref{tab: pvalues drift model} of the \appendixname~and are all clearly above the Bonferroni-corrected critical value of 0.05/10=0.005. Additionally, a comparison of the SNR of the averaging and drift models is performed. The drift model exhibits a better SNR than the averaging model in 4 out of 5 measurements, as shown in Table \ref{tab: drift vs average} and Figure \ref{fig: average vs drift}. This is attributable to partial cancellation in (\ref{eq: Z}) of ENDOR signal components when $\phi_b$ changes significantly over batches. As an extreme example, $\phi_b = \exp(ib/B)$ would lead to the extracted spectrum from the averaging model being nothing but noise. Indeed, the larger the drift of $\mathrm{arg}\phi_b$ shown in Figure \ref{fig: angle phi}, the more pronounced the SNR advantage of the drift model over the averaging model in Table \ref{tab: drift vs average}. This drift in $\phi$ is positively correlated with the drift observed in $\psi$, see panels C and D of Figure \ref{fig:DriftModelFit}. While the correlation is not perfect, which argues against including it as a fixed component of the model, it points to the dominant source of drift for both $\psi$ and $\phi$ originating from phase and amplitude changes due to thermal drift of the MW coupling to the ENDOR resonator containing the chemical sample. 

\end{dexa}

\subsection{Maximum Likelihood Estimation and Parameter Space}\label{sec: maximum likelihood estimators}
Based on the statistical model \ref{eq: homoModel}, the log likelihood is easily found to be
\begin{align}\label{eq: log likelihood}
    \ell_Y (\psi, \phi, \kappa, \Sigma) = -\frac{B(N+1)}{2} \log \left( (2\pi)^2\det(\Sigma) \right) - \frac{1}{2}\sum_{b=1}^B\left\| \tilde{Y}_{b,:}-\phi_b\kappa\right\|^2_P,
\end{align}
where the precision matrix $P\coloneqq\Sigma^{-1}$ and the centered data matrix $\tilde{Y}_{b,\nu}\coloneqq Y_{b,\nu}-\hat{\psi}_b$ have been used. 
Note that, contrary to rank one principal component analysis (PCA) where interest is in the direction $\kappa$ of greatest variability \emph{across repeated measurements}, our interest is in a measure of central tendency for $\kappa$ that shows greatest variability \emph{across frequencies}. Hence, we centre the data to achieve zero empirical row mean (removing $\hat{\psi}_b=\frac{1}{N} \sum_{\nu=0}^N Y_{b,\nu}$) rather than zero empirical column mean.
The interest in the direction of greatest variability across frequencies manifests itself in the use of the maximum method to estimate $\lambda_\mathrm{opt}$ which, given a direction $\hat{\kappa}\in\mathbb{C}^{N+1}$, selects the phase in the complex plane in which the variability across frequencies is greatest. However, we will see shortly that the proposed model is not equivalent to PCA of the transpose of the data matrix.

For each parameter, the MLE when assuming all other parameters known is available in closed form (see \cite{Pokern2021}). Note that $\phi \dia_P \phi $ and $\kappa \dia_P\kappa$ are invertible due to \autoref{lem: det lemma} assuming $\|\phi\|>0$.
\begin{align}
    \hat{\kappa}_\nu(\phi,P, \tilde{Y}) &= \comp{\left(\phi \dia_P \phi  \right)^{-1}  \left(\phi\bul_P \tilde{Y}_{:, \nu}\right)} \label{eq: kappa hat} \\
    \hat{\phi}_b(\kappa,P, \tilde{Y}) &= \comp{\left(\kappa \dia_P \kappa \right)^{-1} \left(\kappa\bul_P \tilde{Y}_{b, :}\right)} \label{eq: phi hat} \\
    \hat{\Sigma}(\phi,\kappa, \tilde{Y}) &=\frac{1}{B(N+1)} \sum_{b=1}^B\sum_{\nu=0}^N\Big(\mathrm{vec}(\tilde{Y}_{b,\nu})-M(\phi_b)\mathrm{vec}(\kappa_\nu)\Big) \Big(\mathrm{vec}(\tilde{Y}_{b,\nu})-M(\phi_b)\mathrm{vec}(\kappa_\nu)\Big)^T. \label{eq: Sigma hat}
\end{align}
In the special case when $\Sigma=r \mathrm{Id}_2$ for $r\in\mathbb{R}_{>0}$ is known, this reduces to a rank one singular value decomposition (SVD) of the centered data matrix $\tilde{Y}$ with $\frac{\hat{\phi}}{\norm{\hat{\phi}}}$ and $\bar{\hat{\kappa}}$ being left and right singular vectors and $\norm{\hat{\phi}}$ the leading singular value, respectively. 
Therefore, an iterative method imitating the standard power iteration method \cite{Trefethen1997} would be a natural algorithm to solve this problem.
Indeed, \cite{Pokern2021} iterate the formulae \ref{eq: kappa hat}, \ref{eq: phi hat}, \ref{eq: Sigma hat} to numerically compute the MLE even though they solve the more general and practically relevant case involving given correlated, non-isotropic noise $\epsilon$.
In this more general case, there is no simple analogy to the SVD and its well-established asymptotic theory, see e.g. \cite{Anderson2003}, is not applicable. 

While the entries of $\tilde{Y}$ are complex numbers, the metric implied by the presence of the $\|\cdot\|_P$-norm in the log likelihood and the fact that the $2\times 2$ matrix $\kappa \dia_P\kappa$ occurring in the conditional MLE \ref{eq: phi hat} cannot generally be written as the matrix representation $M(c)$ of any complex number $c\in\mathbb{C}$, suggest a different approach. It is possible to conceive of the entries of $\tilde{Y}$ as $2\times 2$ matrices $M(\tilde{Y}_{b,\nu})$ so that computation takes place on the ring $R$ of $2\times 2$ real matrices. $\kappa$ is then an element of the Hilbert module $(R^{N+1},\langle\cdot,\cdot\rangle_P)$. A Cauchy-Schwartz type inequality is available on this Hilbert module \cite{Bultheel1982} which would facilitate some of our analysis but we ultimately perceive this algebraic sophistication as a hindrance rather than as a simplification.

Instead, we initially view $\kappa\in\mathbb{C}^{N+1}$ subject to the constraints \ref{eq: meanSpecZero} and \ref{eq: normSpecOne} as an element of an $N+1$ dimensional complex sphere intersected with the hyperplane defined by \ref{eq: meanSpecZero}. Removing an additional phase factor (since $\kappa$ and $\alpha\kappa$ lead to equivalent models for $\alpha\in\mathbb{C}$ with $|\alpha|=1$ as previously noted), we are naturally lead to identifying those $\kappa$ that differ only by a root of unity and hence arrive at the complex projective space $\mathfrak{P}=\mathbb{C}P^{N-1}$, a Riemannian manifold of real dimension $2(N-1)$ as the relevant parameter space for $\kappa$, where the Riemannian metric tensor is implied by the natural quotient embedding in $\mathbb{C}^N$.

We additionally choose a new basis that deals with the constraint \ref{eq: meanSpecZero} by re-writing the noise $\epsilon$ according to $\tilde{Y}_{b,\nu}=\phi_b\kappa_\nu+\tilde{\epsilon}_{b,\nu}$ where 
\begin{align*}
    \tilde{\epsilon}_{b, \nu}\coloneqq \frac{N}{N+1}\epsilon_{b, \nu} - \frac{1}{N+1}\sum_{\tilde{\nu}=0, \tilde{\nu}\neq \nu}^N \epsilon_{b, \tilde{\nu}}.
\end{align*}
Now, we transition from the standard basis vectors $e_k \in\mathbb{R}^{N+1}$ to the Helmert orthonormal basis vectors 
\begin{align*}
    h_j \coloneqq \frac{1}{\sqrt{j(j+1)}}\left(\left(\sum_{k=1}^je_k \right)-je_{j+1}\right), \quad j=1,\dots,N
\end{align*}
to form the Helmert sub-matrix $H=(h_1,\ldots,h_N)^T\in\mathbb{R}^{(N+1)\times N}$ which is in turn used to Helmertize the data matrix $\tilde{Y}^H \coloneqq H \tilde{Y}$, error $\tilde{\epsilon}^H \coloneqq H \tilde{\epsilon}$ and spectral parameter $\kappa^H\coloneqq H \kappa$, see \cite{dryden1998statistical} for details on this standard approach. While the covariance structure of $\tilde{\epsilon}$ is slightly cumbersome, that of the Helmertized error is simply $\vect{\tilde{\epsilon}^H_{b,\nu}}\stackrel{i.i.d.}{\sim} \mathcal{N}(0,\Sigma)$ as shown in Lemma \ref{lem: epsilon dist helmert} in the \appendixname.

\newpage

\section{Extending strong consistency for generalized Fr\'echet means} \label{sec: strong consistency}

For our purpose in Section \ref{sec: Drift model: Strong consistency and central limit theorem} to infer geometric parameters of the drift model, in this section we extend the strong law of large numbers for \emph{generalized Fr\'echet means}, which is usually called \emph{strong consistency} in this context. Let us first introduce this underlying concept.

Noting that expected values of a random variable in a linear space are equivalently described as minimizers of expected squared distance, Fr\'echet \cite{frechet1948elements} used this latter geometric property as a definition for a \emph{mean location} (cf. \cite{HL98}) on a metric space, which was soon called the \emph{Fr\'echet mean} in his honor \cite{karcher2014riemannian}. Further, medians, as minimizers of expected distance lead to \emph{Fr\'echet medians} on metric spaces \cite{fletch9}, and more generally, any $L^p$ mean can thus be generalized \cite{Afsari10}. For Fr\'echet means, two version of set-valued strong consistency under rather broad conditions have been shown \cite{Ziezold1977,bhattacharya2003large}, followed by more versions of strong consistency for Fr\'echet $L^p$ means by \cite{evans2020strong,schotz2022strong}. The general formulation in terms of set valued Fr\'echet means is necessary for data on non-Euclidean spaces, since for example for a sphere with equal point masses on the north and south pole the Fr\'echet mean set is the whole equator, cf. \cite{H_meansmeans_12}. The generalized consistency results also apply to \emph{generalized Fr\'echet means} introduced by \cite{Huckemann2011} to extend and model geometric data descriptors beyond location, such as principal components of the covariance. For instance in a geodesic space, the first principal component can be generalized to a best approximating geodesic. Notably, then minimization has to be conducted no longer over the data space, but over a descriptor space, in case of geodesics, this is the space of geodesics. Likewise, parameters of a parametric model can be viewed as generalized Fr\'echet means.

Curiously, to the best knowledge of the authors, available strong consistency results for generalized Fr\'echet means (\cite{schotz2022strong,Huckemann2011}) always assume a loss function which is bounded from below and thus do not cover the simple case of maximum likelihood parameters of a univariate or multivariate Gaussian. Such a generalization is typically necessary to cover cases where a generalized Fr\'echet mean is estimated along with a (co-)variance-like quantity. This is for example the case for diffusion means with simultaneously estimated variance, see \cite{Eltzner2022}, and for a generalization of the asymptotic theory of the drift model to include the covariance of the noise, as discussed in Section~\ref{sec: Consistency Drift model unknown Sigma}. As it turns out such a generalization of strong consistency results is possible with moderate effort.

For all of the following, let $X_1, X_2, \dots \sim X$ be i.i.d. be random elements mapping from a probability space $(\Omega, \mathcal{A}, \mathcal{P})$ into a topological space $\mathfrak{Q}$ equipped with its Borel $\sigma$-algebra, called the \emph{data space}. 
Moreover, let $(\mathfrak{P}, d)$ be a separable metric space, called the \emph{parameter space}.

\begin{definition}[Sample and Population and Fr\'echet $\rho$-mean]\label{def: Sample and Population and Frechet mean}
  With a function $\rhofun: \mathfrak{Q} \times \mathfrak{P} \mapsto \mathbb{R}$ which is continuous in $\mathfrak{P}$  for all fixed $q\in\mathfrak{Q}$ and  measurable in $\mathfrak{Q}$ for all fixed $p\in\mathfrak{P}$, define, if existent, 
  \begin{align*}
  &\Ffun_n^{(\rhofun)}(\omega, p) \coloneqq \frac{1}{n} \sum_{i=1}^n\rho (X_i(\omega),p),   &&\Ffun^{(\rho)}(p) \coloneqq \mathbb{E}\left(\rho (X,p) \right),\\
  &\ell_n(\omega) \coloneqq \inf_{p\in \mathfrak{P}} \Ffun_n(\omega, p),   &&\ell \coloneqq \inf_{p\in \mathfrak{P}} \Ffun(p),\\
  &E_n^{(\rho)}(\omega) \coloneqq \{p\in \mathfrak{P}\mid \Ffun_n(p)=\ell_n(\omega)\},   &&E^{(\rho)} \coloneqq \{p\in \mathfrak{P}\mid \Ffun^{(\rho)}(p)=\ell\}\,.
  \end{align*}
  The functions $\Ffun^{(\rho)}$ and $\Ffun_n^{(\rho)}$ are called the \emph{population and sample Fr\'echet $\rho$-functions}, respectively, and $E^{(\rho)}$ and $E_n^{(\rho)}$ are the sets of \emph{sample and population Fr\'echet $\rho$-means}, respectively.
\end{definition}
Definition \ref{def: Sample and Population and Frechet mean} is a generalization of a mean  originally introduced for the case $\mathfrak{P}=\mathfrak{Q}$ and $\rho=d^2$ by \cite{frechet1948elements} which is called the \emph{Fr\'echet mean}, see above. 
Due to continuity of $\rho$, $E^{(\rho)}$ is a closed set and $E^{(\rho)}_n(\omega)$ is a random closed set, introduced and studied by \cite{choquet1954theory,kendall1974foundations,matheron1974random}, see also \cite{Molch05}.

\begin{definition}[Two versions of set strong consistency]
  We say that the estimator $E^{(\rho)}_n(\omega)$ for $E^{(\rho)}$ is
  \begin{description}
    \item[ZC:] \emph{Ziezold strongly consistent} if 
    \begin{align*}
    \bigcap_{n=1}^\infty\overline{\bigcup_{k=n}^\infty E^{(\rho)}_k(\omega)} \subseteq E^{(\rho)} \mbox{ for all }\omega\in \Omega \mbox{ almost surely},
    \end{align*}
    \item[BPC:] \emph{Bhattacharya and Patrangenaru strongly consistent}  if $E^{(\rho)}\neq \emptyset$ and if for every $\epsilon>0$ and almost surely for all  $\omega \in \Omega$ there is a number $n=n(\epsilon, \omega)>0$ such that
    \begin{align*}
    \bigcup_{k=n}^\infty E^{(\rho)}_k(\omega) \subseteq \{ p\in \mathfrak{P}: d(E^{(\rho)}, p) \leq \epsilon\}.
    \end{align*}
  \end{description}
\end{definition}

\begin{rem}
  ZC was originally introduced by \cite{Ziezold1977} and established in case of $\mathfrak{P}=\mathfrak{Q}$ and $\rho$ a squared quasi-metric. BP  was originally introduced by  \cite{bhattacharya2003large} and established for Fr\'echet means on Heine-Borel spaces under the additional condition that $E^{(\rho)}$ be not empty. More generally, \cite{evans2020strong} put the two concepts of strong consistency into the more general context of Kuratowski limits, see also  \cite{schotz2022strong} (ZC corresponds to outer limits there and BPC to limits in one-sided Hausdorff distance). 
\end{rem}

As noted above, Fr\'echet $\rho$-means for nonnegative $\rho$ have been introduced by \cite{Huckemann2011}, studying both versions of consistency under a  uniform continuity and a coercivity assumption on $\rho$.  \cite{schotz2022strong} relaxed these assumptions, among others to lower semicontinuity and some assumtpions on bounds. We show ZC and BPC under even weaker assumptions, namely a modulus of continuity along with its prefactor for ZC and using non-emptiness of $E^{\rho}$ for BPC. 

\begin{assumption}\label{ass: rho ziezold}
  In the setup of Definition \ref{def: Sample and Population and Frechet mean} there are
  \begin{enumerate}
    \item $\dot{\rho}: \mathfrak{Q} \times \mathfrak{P} \mapsto [0, \infty)$ which is continuous in $\mathfrak{P}$ for all fixed $q\in\mathfrak{Q}$ and  measurable in $\mathfrak{Q}$ for all fixed $p\in\mathfrak{P}$, with $\mathbb{E}\left[ \dot{\rho}(X,p) \right]< \infty$ for all $p\in\mathfrak{P}$,
    \item $h:[0,\infty) \to [0,\infty)$ continuous with $h(0)=0$, and
    \item $\delta>0$, such that for every $p,p' \in \mathfrak{P}$ with $d(p, p')< \delta$  
    \begin{align}\label{eq: assumption rho}
    \left|\rho(q,p)-\rho(q,p')\right|\leq\dot{\rho}(q,p) ~h\big(d(p,p')\big)\,.
    \end{align}
  \end{enumerate}
  
  Further, assume that $\mathbb{E}\left(\rho (X,p) \right)$ exists  for all $p\in\mathfrak{P}$.
\end{assumption}

\begin{definition}\label{def: Ffun_dot}
  For $\omega \in \Omega, p \in \mathfrak{P}$, under Assumption \ref{ass: rho ziezold}, define 
  \begin{align*}
  \dot{\Ffun}_n(\omega, p) \coloneqq \frac{1}{n} \sum_{i=1}^n\dot{\rho} (X_i(\omega),p),   \qquad \dot{\Ffun}(p) \coloneqq \mathbb{E}\left(\dot{\rho }(X,p) \right).
  \end{align*}
\end{definition}

\begin{lemma}
  \label{lem: P Tilde pointwise}
  Under Assumption \ref{ass: rho ziezold} there is a dense countable subset $\tilde{\mathfrak{P}}\subset \mathfrak{P}$ and measurable $A \subset \Omega$ with $\mathcal{P}(A)=1$ such that for all $\tilde{p}\in \tilde{\mathfrak{P}}$ and all $\omega \in A$ the following hold:
  \begin{itemize}
    \item[(i)]
    $\Ffun_n(\omega, \tilde{p}) \overset{n\to \infty}{\to} \Ffun(\tilde{p}) \quad \text{and}\quad \dot{\Ffun}_n(\omega, \tilde{p}) \overset{n\to \infty}{\to} \dot{\Ffun}(\tilde{p})$,
    \item[(ii)] for all $p\in \mathfrak{P}$ with $d(p,\tilde{p})<\delta/2$ and $(p_n)_{n=1}^\infty \subset \mathfrak{P}$ with $ p_n \to p$, 
    \begin{align*}
    \Ffun(\tilde{p}) - h\big(d(\tilde{p},p)\big)\, \dot{\Ffun}(\tilde{p}) \leq  \liminf_{n\to \infty} \Ffun_n(\omega, p_n) \leq  \limsup_{n\to \infty} \Ffun_n(\omega, p_n)\leq  \Ffun(\tilde{p}) +  h\big(d(\tilde{p},p)\big)\,\dot{\Ffun}(\tilde{p}).
    \end{align*}
  \end{itemize}
  
\end{lemma}

\begin{proof}
  Since $\mathfrak{P}$ is a separable space, there is a countable subset $\tilde{\mathfrak{P}}=\{\tilde{p}_i\}_{i=1}^\infty \subset \mathfrak{P}$ that is dense in $\mathfrak{P}$. 
  For every $\tilde{p}_i \in \tilde{\mathfrak{P}}$ there is, due to the classical strong law of large numbers, a measurable set $A_i \in \mathcal{A}$ with $\mathcal{P}(A_i)=1$ such that
  \begin{align*}
  \Ffun_n(\omega, \tilde{p}_i) \overset{n\to \infty}{\to} \Ffun(\tilde{p}_i) \quad \text{and} \quad \dot{\Ffun}_n(\omega, \tilde{p}_i) \overset{n\to \infty}{\to} \dot{\Ffun}(\tilde{p}_i) \quad \text{for every} \quad i=1,2,\dots \quad \text{and} \quad \omega \in A_i.
  \end{align*}
  Thus, for $A \coloneqq \bigcap_{i=1}^\infty A_i$ we have Assertion (i).
  
  In order to see Assertion (ii), consider $\omega \in A$, $p,p_n \in \mathfrak{P}$ with $p_n \to p$ and $\tilde{p} \in \tilde{\mathfrak{P}}$  with $d(p,\tilde{p})<\delta/2$ and $\delta >0$ from Assumption \ref{ass: rho ziezold}. Then, there is $n_0\in \mathbb{N}$ with $d(p,p_n)<\delta/2$ for all $n \geq n_0$, and hence $d(p_n, \tilde{p})<\delta$  for all $n\geq n_0$ (illustrated in the left panel of Figure \ref{fig:lemma}). Thus
  \begin{align}\label{eq:proof-grand-thm-2}
  \Ffun_n(\omega, \tilde{p}) -  |\Ffun_n(\omega, \tilde{p})-\Ffun_n(\omega, p_n)| \leq  \Ffun_n(\omega, p_n) \leq  \Ffun_n(\omega, \tilde{p}) + |\Ffun_n(\omega, \tilde{p})-\Ffun_n(\omega, p_n)|\,,
  \end{align}
  and from Assumption \ref{ass: rho ziezold} we have for all $n \geq n_0$ 
  \begin{align}\label{eq:proof-grand-thm-3} \nonumber
  |\Ffun_n(\omega, \tilde{p})-\Ffun_n(\omega, p_n)| &\leq \frac{1}{n}\sum_{i=1}^n\left|\rho (X_i(\omega),\tilde{p})-\rho (X_i(\omega),p_n) \right|\\
  &\leq h\big(d(p_n,\tilde{p})\big)\, \frac{1}{n}\sum_{i=1}^n\dot{\rho} (X_i(\omega),\tilde{p}) = h\big(d(p_n,\tilde{p})\big)\,\dot{\Ffun}_n(\tilde{p}).
  \end{align}
  Letting $n \to \infty$ in (\ref{eq:proof-grand-thm-2}), exploiting (\ref{eq:proof-grand-thm-3}), continuity of $d$, continuity of $h$, $h(0)=0$, $h\geq 0$ and Assertion (i) yield at once Assertion (ii).
\end{proof}

\begin{theorem}\label{theo: Ziezold}
  Under Assumption \ref{ass: rho ziezold}, ZC holds for the set of Fr\'echet $\rho$-means on $\mathfrak{P}$.
\end{theorem}

\begin{proof}
  We follow the steps originally introduced by \cite{Ziezold1977} and adopted by \cite{Huckemann2011}.  With $A\subset \Omega$ of full measure and the dense countable subset $\tilde{\mathfrak{P}}$ of  $\mathfrak{P}$, both from Lemma \ref{lem: P Tilde pointwise}, fix $p\in \mathfrak{P}$ and $(p_n)_{n=1}^\infty \subset \mathfrak{P}$ with $ p_n \to p$. We first show that
  \begin{align}\label{eq: Fn(pn) to F}
  \Ffun_n(\omega, p_n) \overset{n\to \infty}{\to} \Ffun(p)\,, 
  \end{align}
  for all $\omega \in A$. 
  
  To this end, with $\delta >0$ from Assumption \ref{ass: rho ziezold}, let $\tilde{p}\in \tilde{\mathfrak{P}}$ with $d(p,\tilde{p})<\delta/2$. Then, due to Assertion (ii) from Lemma \ref{lem: P Tilde pointwise},
  \begin{align}\label{eq: Liminf Limsup Fn}
  \Ffun(\tilde{p}) - h\big(d(\tilde{p},p)\big)\, \dot{\Ffun}(\tilde{p}) \leq  \liminf_{n\to \infty} \Ffun_n(\omega, p_n) \leq  \limsup_{n\to \infty} \Ffun_n(\omega, p_n)\leq  \Ffun(\tilde{p}) +  h\big(d(\tilde{p},p)\big)\,\dot{\Ffun}(\tilde{p}),
  \end{align}
  for all $\omega\in A$. 
  
  Letting $(\tilde{p}_k)_{k=1}^\infty \subset \tilde{\mathfrak{P}}$ with $\tilde{p}_k \overset{k\to \infty}{\to} p$, there is $k_0\in \mathbb{N}$ with $d(p,\tilde{p}_k)<\delta/2$ for all $k \geq k_0$ (illustrated in the right panel of Figure \ref{fig:lemma}). Plugging these in, into (\ref{eq: Liminf Limsup Fn}) we obtain for all $\omega \in A$,
  \begin{align}\label{eq: lim p_k and lim p_n to p}
  \lim_{k\to \infty}\left(\Ffun(\tilde{p}_k) - h\big(d(\tilde{p}_k,p)\big)\,\dot{\Ffun}(\tilde{p}_k)\right) 
  &\leq  \liminf_{n\to \infty} \Ffun_n(\omega, p_n) \\
  &\leq  \limsup_{n\to \infty} \Ffun_n(\omega, p_n)\nonumber
  \leq  \lim_{k\to \infty} \left(\Ffun(\tilde{p}_k) + h\big(d(\tilde{p}_k,p)\big)\,\dot{\Ffun}(\tilde{p}_k) \right).
  \end{align}
  This yields (\ref{eq: Fn(pn) to F}), as, due to continuity of $\Ffun,\dot{\Ffun}$ and $h$, as well as $h(0)=0$,  
  \begin{align*}\
  \lim_{k\to \infty}\left(\Ffun(\tilde{p}_k) - h\big(d(\tilde{p}_k,p)\big)\,\dot{\Ffun}(\tilde{p}_k)\right) 
  = \Ffun(p) = \lim_{k\to \infty}\left(\Ffun(\tilde{p}_k) + h\big(d(\tilde{p}_k,p\big)\,)\dot{\Ffun}(\tilde{p}_k)\right) \,.
  \end{align*}    
  
  Next we show the assertion of the theorem. Since it is trivial in case of $\bigcap_{n=1}^\infty\overline{\bigcup_{k=n}^\infty E_k^{(\rho)}(\omega)} =\emptyset$, it is sufficient to show that 
  \begin{align*}
  \text{if }\bigcap_{n=1}^\infty\overline{\bigcup_{k=n}^\infty E_k^{(\rho)}(\omega)} \neq \emptyset \quad \text{then} \quad \ell_n(\omega) \to \ell \quad \text{for} \quad \omega \in A.
  \end{align*}
  To see this, we show the following two inequalities for all $\omega \in A$
  \begin{align}\label{eq:l_n-geq-l}
  \liminf_{n\to \infty} \ell_n(\omega)  &\geq \ell,\\ \label{eq:l_n-leq-l}
  \limsup_{n\to \infty} \ell_n(\omega)  &\leq \ell \,.
  \end{align}
  Noting that 
  \begin{align*}
  \text{if}\quad p\in \bigcap_{n=1}^\infty\overline{\bigcup_{k=n}^\infty E_k^{(\rho)}(\omega)} \quad \text{then} 
  \quad p\in \overline{\bigcup_{k=j}^\infty E_{n_k}^{(\rho)}(\omega)} \quad \text{for all} \quad j\in\mathbb{N} \,,
  \end{align*}
  where $n_j \to \infty$ is a sequence with
  $$ \lim_{j \to \infty} \ell_{n_j}(\omega) = \liminf_{n\to \infty} \ell_n(\omega)\,,  $$ 
  and recalling that the closure of a set in a metric space is given by all cluster points of sequences in it, 
  there is a sequence $\{p_i\}_{i=1}^\infty$ with $p_i \to p$ and $p_i \in E_{k_i}^{(\rho)}(\omega)$ for a subsequence $\{k_i\}_{i=1}^\infty$ of $n_j$. Using (\ref{eq: Fn(pn) to F}) we obtain
  \begin{align*}
  \liminf_{n\to \infty}\Ffun_{k_n}(\omega, p_n) =  \lim_{i\to \infty} \ell_{k_i}(\omega) = \Ffun(p) \geq \ell.
  \end{align*}
  for all $\omega \in A$, yielding (\ref{eq:l_n-geq-l}). 
  
  To see (\ref{eq:l_n-leq-l}), set $p_n:=p$ for some $p\in E^{(\rho)}$, so that with (\ref{eq: Fn(pn) to F}) there is a nonnegative random sequence $\{\epsilon_n(\omega)\}_{n=1}^\infty $, converging to zero for all $\omega \in A$, with 
  \begin{align*}
  \ell= \Ffun(p) \geq \Ffun_n(\omega, p) - \epsilon_n(\omega) \geq \ell_n(\omega) - \epsilon_n(\omega)\,,
  \end{align*} 
  for all $\omega \in A$, yielding at once (\ref{eq:l_n-leq-l}). This completes the proof.
\end{proof}

\begin{figure}[ht!]
  \centering
  \includegraphics[width = 0.4\textwidth]{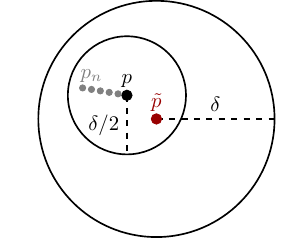}
  \includegraphics[width = 0.4\textwidth]{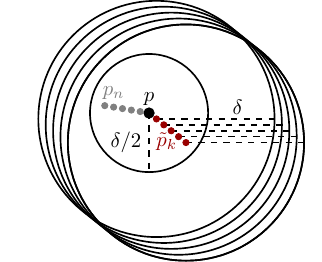}
  \caption{Left: Illustration of Lemma \ref{lem: P Tilde pointwise}. For all $\tilde{p}\in \tilde{\mathfrak{P}}$ (red point) with distance $d(p, \tilde{p})<\delta/2$ to $p\in \mathfrak{P}$ (black point), the inequality (\ref{eq:proof-grand-thm-2}) holds for  $(p_n)_{n=1}^\infty \subset \mathfrak{P}$ with $ p_n \to p$ (gray points). Right: Illustration of (\ref{eq: lim p_k and lim p_n to p}) in the proof of Theorem \ref{theo: Ziezold}. 
    Compared to the left panel, where $\tilde{p}\in \mathfrak{P}$ is fixed, the sequence $(\tilde{p}_k)_{k=1}^\infty \subset \tilde{\mathfrak{P}}$ (red dots) converges also to $p$ (black dot) for $k\to \infty$.}
  \label{fig:lemma}
\end{figure}

\begin{assumption}\label{ass: rho BPC rho to C}
  The population Fr\'echet $\rho$-mean is not empty: $E^{(\rho)}\neq \emptyset$ and
  for all random sequences $\{p_n\}_{n\in\mathbb{N}}$ without accumulation points in $ \mathfrak{P}$, there is a constant $\ell < C \leq \infty$ such that a.s.
  \begin{align}\label{eq: F inf condition}
  \liminf_{n \to \infty} \limits \rho (X, p_n) \geq C.
  \end{align}
\end{assumption}

\begin{theorem}\label{theo: BPC}
  Under Assumptions \ref{ass: rho ziezold}, and \ref{ass: rho BPC rho to C} BPC holds for the set of Fr\'echet $\rho$-means on $\mathfrak{P}$.
\end{theorem}

\begin{proof}
  It suffices to show that for random sequence  $p_n(\omega)\in E_n^{(\rho)}(\omega)$ with underlying random one-sided Hausdorff distance $r_n(\omega)$, i.e.
  \begin{align*}
  p_n(\omega) \in \argmax_{p \in  E_n^{(\rho)}(\omega)}\left(\min_{p'\in  E^{(\rho)}} d(p, p')\right),\qquad r_n(\omega) = \max_{p \in  E_n^{(\rho)}(\omega)}\left(\min_{p'\in  E^{(\rho)}} d(p, p')\right)\,, n\in \mathbb N
  \end{align*}
  $r_n(\omega) \to 0$ for all $\omega \in \Omega$ a.s.
  
  If this was not the case, then there would be  $\tilde{A}\subset \Omega$ with $\mathcal{P}(\tilde{A})>0$ such that for all $\omega \in \tilde{A}$, there is a subsequence $n_k(\omega)$ with $r_{n_k(\omega)}(\omega) \geq r_0(\omega) > 0$ and $r_0(\omega)>0$. We now derive a contradiction.  
  
  Due to Theorem \ref{theo: Ziezold}, we have ZC, so that with a null set $B$, all cluster points of $p_{n_k(\omega)}(\omega)$ lie in $E^{(\rho)}$ for all $\omega \in \tilde{A} \setminus B$. In consequence, $p_{n_k(\omega)}(\omega)$ has no cluster points for all $\omega \in \tilde{A} \setminus B$. Fixing $\omega_0 \in \tilde{A} \setminus B$, set
  $$ \tilde{p}_k(\omega) := \left\{\begin{array}{rcl}
  p_{n_k(\omega)}(\omega) &\mbox{ if} &\omega\in \tilde{A} \setminus B\\
  p_{n_k(\omega_0)}(\omega_0)&\mbox{ if} &\omega\in \Omega \setminus(\tilde{A} \setminus B)
  \end{array}\right.\,,k \in \mathbb N\,,
  $$  
  to obtain a sequence $\tilde{p}_k(\omega)$ without any cluster points for all $\omega\in \Omega$, so that in consequence of Assumption \ref{ass: rho BPC rho to C}, there is $C > \ell$ and a.s. $m(\omega) \in \mathbb N$ such that
  $$ \rho\left(X(\omega),\tilde{p}_{k}(\omega)\right) \geq C$$
  for all $k \geq m(\omega)$, almost surely. Hence, by construction, with a null set $\tilde B\subset \Omega$,
  \begin{eqnarray}\nonumber\label{eq:proof-BPC-1}
  \mathcal{F}_{n_k(\omega)}\big(\omega,p_{n_k(\omega)}(\omega)\big) 
  &=& \frac{1}{n_k(\omega)} \sum_{j=1}^{n_k(\omega)} \rho\left(X_j(\omega),p_{n_k(\omega)}(\omega)\right) \\& \geq& C \mbox{ for all $\omega \in \tilde{A} \setminus \tilde B$, if $n_k(\omega) > m(\omega)$.} 
  \end{eqnarray}
  By hypothesis there is $p \in  E^{(\rho)}$ with, due to the strong law of large numbers, $\mathcal F_{n_k(\omega)}(\omega,p) \stackrel{a.s.}{\to} \mathcal F(p) = \ell < C$, by construction, i.e. there is $n(\omega) \in \mathbb N$ such that 
  $\mathcal F_{n_k(\omega)}(\omega,p) < C$ for all $n_k(\omega) > n(\omega)$, a.s. In conjunction, with (\ref{eq:proof-BPC-1}), letting $n_k(\omega) > \max\{n(\omega),m(\omega)\}$ we have thus
  $$ \mathcal F_{n_k(\omega)}(\omega,p) < C \leq \mathcal{F}_{n_k(\omega)}\big(\omega,p_{n_k(\omega)}(\omega)\big)$$
  on a set of positive measure, a contradiction to $p_{n_k(\omega)}(\omega) \in E^{\rho}_{n_k(\omega)}(\omega)$, completing the proof.
\end{proof}

\begin{rem}
  The BPC version of the strong law in the literature usually requires that a Heine-Borel property, e.g. \cite{bhattacharya2003large, evans2020strong,schotz2022strong}. If $(\mathfrak{P},d)$ satisfies the Heine-Borel property, all sequences without accumulation points diverge, so Assumption \ref{ass: rho BPC rho to C} holds for all $\omega \in \Omega$. If $\mathfrak{P}$ is 
  compact, then Assumption \ref{ass: rho BPC rho to C} holds and BPC follows immediately from ZC.
\end{rem}

\newpage
\section{Strong consistencies and CLTs for the homoscedastic drift model}\label{sec: Drift model: Strong consistency and central limit theorem}

In this section, the theory developed in Section \ref{sec: strong consistency} is applied to prove strong consistency for the homoscedastic drift model as the number $B$ of batches tends to infinity.  For this purpose, we reformulate in Definition \ref{ass: drift model} 
the homoscedastic drift model from Definition \ref{def: homosced}
making $\phi$ explicitly stochastic and assume it to be i.i.d. 
Without loss of generality, we consider $Y$ centered, i.e.  $\psi$ has been subtracted and the basis has been transformed with the Helmert sub-matrix as described in Section \ref{sec: maximum likelihood estimators}. For ease of notation, we assume that the original random variable is $N+1$ dimensional, so that $Y$ is $N$ dimensional, and we omit writing the tilde and the superscript $^H$ symbol from Section \ref{sec: maximum likelihood estimators}. 

Then, below in Section \ref{sec: Strong Consistency drift model} the modulus of continuity $h$ with its prefactor from Assumption \ref{ass: rho ziezold} is explicitly calculated and ZC is proven. With a little more effort  it is shown that the population Fr\'echet mean $E^{\rho}=\{[\kappa^{(0)}]\}$ is unique, in order to establish BPC. Finally, a central limit theorem for $\hat{\kappa}$ is shown in Section \ref{sec: CLT Drift model} and one for $\hat{I}$ in Section \ref{sec: clt I}. 

\begin{definition}[Centered Homoscedastic Drift Model]\label{ass: drift model} 
The complex $N$-dimensional random vector $Y$ is given by 
\begin{align}\label{eq: homo-asympt} Y=\phi\kappa^{(0)} +\epsilon\end{align}
where $\kappa^{(0)}\in \Scomp \coloneqq \{\kappa \in \mathbb{C}^N: ||\kappa||=1\}$ comprises the true but unknown ENDOR spectrum, $\phi$ is  a complex random value and $\epsilon = (\epsilon_1,\ldots,\epsilon_N)$ is a complex $N$-dimensional random vector independent of $\phi$ 
with 
\begin{align*}
    0 < \mathbb{E}[\vect{\phi}^T\vect{\phi}]=c_\phi < \infty, \qquad \vect{\epsilon_{\nu}} \stackrel{i.i.d.}{\sim} \mathcal{N}(0,\Sigma), 1\leq \nu \leq N .
\end{align*}
Moreover, we assume that the precision matrix $P=\Sigma^{-1}$ has two positive eigenvalues $\lambda_1>\lambda_2>0$.

Thus, $\mathfrak{Q} = \mathbb{C}^N$ is the data space, as descriptor space we choose $\mathfrak{P} = \mathbb{C}P^{N-1}$ and for the loss we choose
\begin{align}\label{eq: definition rho}
    \rhofun: \mathfrak{Q} \times \mathfrak{P} \to \RR, \quad (Y, [\kappa]) \mapsto d_P\left(Y,  \hat{\phi}(\kappa,P,Y)\kappa\right)^2,
\end{align}
with
\begin{align}\label{eq: phi-hat-asympt-model}
    \hat{\phi}(\kappa,P,Y) &= \comp{\left(\kappa \dia_P \kappa  \right)^{-1} \left(\kappa\bul_P Y\right)} \in \mathbb{C}.
\end{align}

Thus, for a sample $Y(1), Y(2), \dots \stackrel{i.i.d.}{\sim} Y$ from (\ref{eq: homo-asympt}) we have the following sample and population Fr\'echet functions
\begin{align*}
    \Ffun_B (\omega, [\kappa]) = \frac{1}{B}\sum_{b=1}^B\rhofun(Y(b), [\kappa]), \qquad \Ffun([\kappa]) = \int\rhofun(Y, [\kappa])\diff\mathbb{P}\left(\phi,\epsilon\right).
\end{align*}

\end{definition}

\begin{rem}
Note that (\ref{eq: phi-hat-asympt-model}) is the MLE (\ref{eq: phi hat}) for a single $b\in \{1,\ldots,B\}$ of the homoscedastic drift model (\ref{eq: homoModel}).
 
Further, note that (\ref{eq: definition rho}) is well defined, since for $\kappa, \tilde{\kappa} \in [\kappa]$ there is $\lambda \in \mathbb{R}$ such that $\tilde{\kappa} = e^{i\lambda}\kappa$, whence
\begin{align*}
    \hat{\phi}(\tilde{\kappa},P,Y) = e^{-i\lambda}\comp{\left(\kappa \dia_P \kappa  \right)^{-1} \left(\kappa\bul_P Y\right)} 
\end{align*}
yielding
\begin{align*}
    \hat{\phi}(\tilde{\kappa},P,Y)\tilde{\kappa} = \hat{\phi}(\kappa,P,Y)\kappa.
\end{align*}
\end{rem}

\subsection{Strong consistencies for the centered homoscedastic drift model}\label{sec: Strong Consistency drift model}
Here, we establish that, as more data is accumulated and hence $B\rightarrow \infty$, the estimator $[\hat{\kappa}]$ arising from the centered homoscedastic drift model in Definition \ref{ass: drift model}, i.e. the generalized Fr\'echet mean, is strongly consistent in the sense of ZC and BPC.

\begin{theorem}\label{the: inequality rho}
Under Assumption \ref{ass: drift model} ZC holds for the centered homoscedastic drift model.
In particular, in Assumption \ref{ass: rho ziezold} the modulus of continuity can be chosen as
$$ h([\kappa],[\kappa'])=  d([\kappa],[\kappa'])$$
with prefactor 
\begin{align*}
    \dot{\rho}\left( Y, P\right) \coloneqq \sqrt{\lambda_1}\left(\frac{\lambda_1^2+\lambda_2^2}{\lambda_1 \lambda_2}\right)\left((\lambda_1 +2)\sqrt{2N}+8\sqrt{2}N+\frac{32\sqrt{2}N\left(\lambda_1^2+\lambda_2^2\right)}{\lambda_1\lambda_2}\right)  \left|\left|Y \right|\right|^2 .
\end{align*}
\end{theorem}

\begin{proof}
Let $[\kappa], [\kappa'] \in \mathfrak{P}$ and $\kappa\in[\kappa], \kappa'\in[\kappa']$ arbitrary. 
Recalling
 \begin{align*}
    \rhofun(Y, [\kappa])=  \spr{Y}{Y}_P  - \spr{\hat{\phi}(\kappa,P, Y) \kappa}{Y}_P\,.
\end{align*}
from Lemma \ref{lem: rho formula} and using the Cauchy–Schwarz inequality we obtain
\begin{align*}
\left|\rhofun(Y, [\kappa]) - \rhofun(Y, [\kappa']) \right| &= \left|\spr{\hat{\phi}(\kappa',P, Y) \kappa'-\hat{\phi}(\kappa,P, Y) \kappa}{Y}_P \right|\\
&\leq \sqrt{\spr{Y}{Y}_P}~d_P\Big(\hat{\phi}(\kappa,P, Y)\kappa, \hat{\phi}(\kappa',P, Y)\kappa'\Big)\,.
\end{align*}
Since $\lambda_1$ is the largest eigenvalue of $P$, by definition of the Mahalanobis inner product in Section \ref{sec: notation}, the first term of the bottom line above is bounded by
\begin{align}\label{eq: ypy}
    \sqrt{\spr{Y}{Y}_P} ~=~ \sqrt{\sum_{\nu=1}^N \vect{Y_\nu}^TR \begin{pmatrix}
\lambda_1& 0\\
0 & \lambda_2
\end{pmatrix}
R^T\vect{Y_\nu}} ~\leq~ \sqrt{\lambda_1}\, \|Y\|\,.
\end{align}
A bound for the second term can be obtained from Lemma \ref{lem: dp xa yb inequality} in conjunction with Lemmata \ref{lem: dp phi plus phi and kappa minus kappa} and \ref{lem: dp phi minus phi and kappa plus kappa},
\begin{align*}
    &d_P\Big(\hat{\phi}(\kappa,P, Y)\kappa, \hat{\phi}(\kappa',P, Y)\kappa'\Big)\\ 
    &\leq \left(\frac{\lambda_1^2+\lambda_2^2}{\lambda_1 \lambda_2}\right)\left(\lambda_1 \sqrt{2N}+8\sqrt{2}N+\frac{32\sqrt{2}N\left(\lambda_1^2+\lambda_2^2\right)}{\lambda_1\lambda_2}+2\sqrt{2N}\right)  \left|\left|Y \right|\right| \left|\left|\kappa -\kappa' \right|\right|\,.
\end{align*}
In consequence, since $\kappa\in[\kappa], \kappa'\in[\kappa']$ have been arbitrary,
\begin{align*}
&\left|\rhofun(Y, [\kappa]) - \rhofun(Y, [\kappa']) \right| \\
&\leq \sqrt{\lambda_1}\left(\frac{\lambda_1^2+\lambda_2^2}{\lambda_1 \lambda_2}\right)\left((\lambda_1 +2)\sqrt{2N}+8\sqrt{2}N+\frac{32\sqrt{2}N\left(\lambda_1^2+\lambda_2^2\right)}{\lambda_1\lambda_2}\right)  \left|\left|Y \right|\right|^2 
d([\kappa], [\kappa'])\,,
\end{align*}
yielding the assertion on $h$ and $\dot \rho$.

Further, since $\mathfrak{P}$ is separable, $\dot{\rho}$ does not depend on $\kappa$, and because of Assumption \ref{ass: drift model} the second moment of $Y$ exists, it follows from Theorem \ref{theo: Ziezold} that ZC holds.
\end{proof}

The stronger BPC hinges on existence and uniqueness of the generalized Fr\'echet population mean. To this end we first decompose and compute the generalized Fr\'echet population function.    
\begin{lemma}\label{lem: A2 is constant}
    For the centered homoscedastic drift model from Definition \ref{ass: drift model} we have 
\begin{align*}
    (i)& \qquad \Ffun([\kappa])= \int \rhofun(\phi \kappa^{(0)}, [\kappa])  \diff\mathbb{P}\left(\phi\right)
    +\int \rhofun(\epsilon, [\kappa]) \diff\mathbb{P}\left(\epsilon\right)
    \,,\\
    (ii)& \qquad\int \rhofun(\epsilon, [\kappa]) \diff\mathbb{P}\left(\epsilon\right) = 2N-2\,,\\
    (iii)& \qquad\int \rhofun(\phi \kappa^{(0)}, [\kappa])  \diff\mathbb{P}\left(\phi\right)= \left(\tilde{\eta}^2-\frac{\tilde{\eta}^4}{4} \right)
    \int \vect{\phi}^T S \vect{\phi}\diff\mathbb{P}\left(\phi\right)\,,
\end{align*} 
 where $\tilde{\eta} = d\left([\kappa], \left[\kappa^{(0)}\right]\right)$ and $S$ is a matrix with eigenvalues greater than or equal to $\lambda_2$.
\end{lemma}

\begin{proof}
To see $(i)$, note that by Definition (\ref{eq: phi-hat-asympt-model}),
$$ \hat\phi(\kappa,P,Y) = \hat\phi(\kappa,P,\phi\kappa^{(0)})+\hat\phi(\kappa,P,\epsilon)\,,$$
whence in conjunction with (\ref{eq: definition rho}),
\begin{align*}
    &\Ffun([\kappa]) = \int \left|\left|\phi \kappa^{(0)}+ \epsilon-  \hat{\phi}(\kappa,P,\phi \kappa^{(0)}+ \epsilon)\kappa\right|\right|_P^2\diff\mathbb{P}\left(\phi,\epsilon\right) \\ 
    &= \int \rhofun(\phi \kappa^{(0)}, [\kappa]) 
    + 2\spr{\phi \kappa^{(0)}- \hat{\phi}(\kappa,P,\phi \kappa^{(0)})\kappa}{\epsilon- \hat{\phi}(\kappa,P,\epsilon)\kappa}_P 
    +\rhofun(\epsilon, [\kappa]) \diff\mathbb{P}\left(\phi, \epsilon\right)
\end{align*}
for any $\kappa\in [\kappa], \kappa^{(0)}\in [\kappa^{(0)}]$.
Since $\phi$ and $\epsilon$ are independent and $\mathbb{E}[\epsilon]=0$ the integral over the mixed term vanishes yielding the first asserted equation.

To see the $(ii)$, use Lemma \ref{lem: rho formula} to obtain 
 \begin{align*}
    \int \rhofun(\epsilon, [\kappa]) \diff\mathbb{P}\left(\epsilon\right) = \int \left(\spr{\epsilon}{\epsilon}_P  - \spr{\hat{\phi}(\kappa,P, \epsilon) \kappa}{\epsilon}_P\right)\diff\mathbb{P}\left(\epsilon\right)
\end{align*}
for any $\kappa \in [\kappa]$ with the MLE from (\ref{eq: phi hat}).
By definition of the Mahalanobis type inner product, independence of the $\epsilon_\nu \sim \mathcal{N}(0,\Sigma)$ ($\nu =1,\ldots,N$) and $P=\Sigma^{-1}$, calculate the first term:
\begin{align*}
     \int \spr{\epsilon}{\epsilon}_P\diff\mathbb{P}\left(\epsilon\right)&= 
     \sum_{\nu=1}^N\Tr\left(P~ \vect{\epsilon_\nu}\vect{\epsilon_\nu}^T \diff\mathbb{P}\left(\epsilon\right)\right) =
     \sum_{\nu=1}^N\Tr(P\Sigma) = \sum_{\nu=1}^N\Tr(\mathrm{Id}_2) = 2N\,.
\end{align*}
Similarly, compute the second term: 
\begin{align*}
   &\int \spr{\hat{\phi}(\kappa,P, \epsilon) \kappa}{\epsilon}_P\diff\mathbb{P}\left(\epsilon\right) 
    = \int \sum_{\nu=1}^N\vect{\epsilon_\nu}^TP M(\kappa_\nu) \vect{\hat{\phi}(\kappa,P, \epsilon)}
    \diff\mathbb{P}\left(\epsilon\right)\\
    &= \int \left(\kappa\bul_P \epsilon\right)^T \left(\kappa \dia_P \kappa  \right)^{-1} \left(\kappa\bul_P \epsilon\right)
    \diff\mathbb{P}\left(\epsilon\right) 
    =\Tr\left( \left(\kappa \dia_P \kappa  \right)^{-1} \int \left(\kappa\bul_P \epsilon\right) \left(\kappa\bul_P \epsilon\right)^T
    \diff\mathbb{P}\left(\epsilon\right)\right)\\
    &=\Tr\left( \left(\kappa \dia_P \kappa  \right)^{-1} \left(\kappa \dia_P \kappa \right)\right) = 2\,,
\end{align*}
since 
\begin{align*}
    &\int \left(\kappa\bul_P \epsilon\right) \left(\kappa\bul_P \epsilon\right)^T
    \diff\mathbb{P}\left(\epsilon\right)
    = \int \left(\sum_{\nu=1}^N M(\kappa_\nu)^T P \vect{\epsilon_\nu} \right) \left(\sum_{\nu=1}^N \vect{\epsilon_\nu}^T P M(\kappa_\nu)\right)
    \diff\mathbb{P}\left(\epsilon\right)\\
    &= \sum_{\nu=1}^N M(\kappa_\nu)^T P \left(\int\vect{\epsilon_\nu}   \vect{\epsilon_\nu}^T \diff\mathbb{P}\left(\epsilon\right) \right)
     P M(\kappa_\nu)= \kappa \dia_P \kappa .
\end{align*}
Subtracting the first term from the second gives the second asserted equation.

Proving $(iii)$ is a more elaborate. We have
\begin{align*}
    &\int \rhofun(\phi \kappa^{(0)}, [\kappa])  \diff\mathbb{P}\\
    &= \int \vect{\phi}^T 
    \left(\left(\kappa^{(0)} \dia_P \kappa^{(0)}\right)-\left(\kappa^{(0)} \dia_P\kappa \right)\left(\kappa \dia_P \kappa\right)^{-1}\left(\kappa \dia_P \kappa^{(0)}\right)\right)\vect{\phi} \diff\mathbb{P}\left(\phi\right)
\end{align*}
where $\kappa \in [\kappa]$.
Without loss of generality, assume that $\kappa$ and $\kappa^{(0)}$ are in optimal position, i.e.  $d\left([\kappa],\left[\kappa^{(0)}\right])\right)= ||\kappa -\kappa^{(0)} ||=: \widetilde{\eta}$.
Therefore,  due to Lemma \ref{lem: opt-pos},
\begin{align*}
    \kappa^*\kappa^{(0)} = \Re(\kappa^*\kappa^{(0)})=\vectn{\kappa}^T\vectn{\kappa^{(0)}} = \cos(\eta) \quad \text{where} \quad \eta = 2 \arcsin(\widetilde{\eta}/2).
\end{align*}
We can rewrite the above formula by using matrix notation
\begin{align*}
\kappa \dia_P \kappa = & \begin{pmatrix} M(\kappa_1)^T & \dots &  M(\kappa_{N})^T \end{pmatrix} \begin{pmatrix} P & 0 \dots & 0\\ 0 & \ddots & 0\\ 0 & \dots 0 & P \end{pmatrix} \begin{pmatrix} M(\kappa_1) \\ \vdots \\  M(\kappa_{N}) \end{pmatrix}.
\end{align*}
Now one can define a matrix containing only rotations on the diagonal, such that all $M(\kappa_\nu)$ become diagonal, i.e. real:
\begin{align*}
  R_1 := \begin{pmatrix} R_{\alpha_1} & 0 \dots & 0\\ 0 & \ddots & 0\\ 0 & \dots 0 & R_{\alpha_{N}} \end{pmatrix}\in\mathbb{R}^{2N\times 2N}
\end{align*}
and then define a matrix $R_2 = (\widetilde{R}_2 \otimes \mathrm{Id}_2)\in\mathbb{R}^{2N\times 2N}$ which rotates these real blocks such that we get $R_2R_1 \vectn{\kappa} =(1,0,\dots,0)^T$.
From the construction of $R_1$ and $R_2$ follows directly
    \begin{align*}
    &\compn{R_2 R_1 \vectn{\kappa}}^* \compn{R_2 R_1 \vectn{\kappa^{(0)}}}
    = \compn{R_2 
    \vectn{\begin{matrix} e^{i\alpha_1}\kappa_1\\ \vdots \\ \ e^{i\alpha_N}\kappa_N\end{matrix}}}^*
    \compn{R_2 
    \vectn{\begin{matrix} e^{i\alpha_1}\kappa^{(0)}_1\\ \vdots \\ \ e^{i\alpha_N}\kappa^{(0)}_N\end{matrix}} }\\
    &=\compn{ 
    \vectn{\begin{matrix} e^{i\alpha_1}\kappa_1\\ \vdots \\ \ e^{i\alpha_N}\kappa_N\end{matrix}}}^*\widetilde{R}_2^*
    \widetilde{R}_2\compn{ 
    \vectn{\begin{matrix} e^{i\alpha_1}\kappa^{(0)}_1\\ \vdots \\ \ e^{i\alpha_N}\kappa^{(0)}_N\end{matrix}} }
    = \kappa^*\kappa^{(0)} =\cos(\eta).
\end{align*}
Thus, it follows that $\compn{R_2R_1 \vectn{\kappa^{(0)}}}_1=\cos(\eta)$.
Next, we define a Matrix $R_3$ which rotates all $M(\kappa^{(0)}_\nu)$ for $\nu \ge 2$ to real numbers and leaves the component $\nu=1$ unchanged (thus leaving $\kappa$ unchanged) and a Matrix $R_4 = (\widetilde{R}_4 \otimes \mathrm{Id}_2)$ which rotates only the components $\nu \ge 2$, such that we get $\kappa^{(0)}_\nu = 0$ for $\nu \ge 2$.
As a trade-off for this simplification, the matrix in the center becomes more complicated:
\begin{align*}
  (\mathrm{Id}_N \otimes P) \quad \to \quad Q := R_4 R_3 R_2 R_1 (\mathrm{Id}_N \otimes P) R_1^T R_2^T R_3^T R_4^T.
\end{align*}
This leads to 
\begin{align*}
    &\int \rhofun(\phi \kappa^{(0)}, [\kappa])  \diff\mathbb{P}\\
    &=\int \vect{\phi}^T 
    \Bigg(\cos^2(\eta)  Q_{11}  + \cos(\eta)\sin(\eta) \left(  Q_{12} + Q_{21}  \right) + \sin^2(\eta) Q_{22} \\
    &- 
    \Big(\cos(\eta) Q_{11}  + \sin(\eta) Q_{21} \Big) 
    Q_{11}^{-1}
    \Big(\cos(\eta) Q_{11}  + \sin(\eta) Q_{12}\Big)\Bigg)
    \vect{\phi}\diff\mathbb{P}\left(\phi\right)\\
    &=\sin^2(\eta)\int \vect{\phi}^T 
    \Bigg(Q_{22} - Q_{12}^T Q_{11}^{-1} Q_{12}\Bigg)
    \vect{\phi}\diff\mathbb{P}\left(\phi\right)\\
     &=\left(\tilde{\eta}^2-\frac{\tilde{\eta}^4}{4} \right)\int \vect{\phi}^T 
    \Bigg(Q_{22} - Q_{12}^T Q_{11}^{-1} Q_{12}\Bigg)
    \vect{\phi}\diff\mathbb{P}\left(\phi\right)
\end{align*}
where $Q_{ij}$ are $2\times 2$ blocks from $Q$. Note furthermore that $Q_{ji}^T = Q_{ij}$ since $Q$ is symmetric.
We define $S\coloneqq Q_{22} - Q_{12}^T Q_{11}^{-1} Q_{12}$.
The matrix
\begin{align*}
\widetilde{Q} = & \begin{pmatrix} Q_{11} & Q_{12} \\ Q_{12}^T & Q_{22} \end{pmatrix} \in \mathbb{R}^{4\times 4}
\end{align*}
is positive definite with eigenvalues in $[\lambda_2, \lambda_1]$ since it is a leading principal minor of $Q$.
It follows for $v\in \mathbb{R}^2\setminus \{0\}$
\begin{align*}
    v^TSv&= v^TQ_{22}v- v^T  Q_{12}^TQ_{11}^{-1}Q_{12}v \\
    &= v^TQ_{12}^TQ_{11}^{-1}Q_{11}Q_{11}^{-1}Q_{12}v-v^TQ_{12}^TQ_{11}^{-1}Q_{12}v-v^TQ_{12}^TQ_{11}^{-1}Q_{12}v + v^TQ_{22}v\\
    &=\begin{pmatrix} -Q_{11}^{-1}Q_{12}v\\ v\end{pmatrix}^T \widetilde{Q} \begin{pmatrix} -Q_{11}^{-1}Q_{12}v\\ v\end{pmatrix}\geq \lambda_2 \begin{pmatrix} -Q_{11}^{-1}Q_{12}v\\ v\end{pmatrix}^T \begin{pmatrix} -Q_{11}^{-1}Q_{12}v\\ v\end{pmatrix}\geq \lambda_2||v||^2.
\end{align*}
Thus, $S$ has only eigenvalues greater than or equal to $\lambda_2$.

\end{proof}

\begin{theorem}[Uniqueness]\label{theo: Frechet mean value is unique} For the centered homoscedastic drift model from Definition \ref{ass: drift model} we have for every $\epsilon>0$ that 
$$\inf_{[\kappa]: d([\kappa],[\kappa^{(0)}])>\epsilon} \Ffun([\kappa]) >\Ffun\left(\left[\kappa^{(0)}\right]\right).$$
    In particular, the Fr\'echet population mean is uniquely  
    $\left[\kappa^{(0)}\right]$.
\end{theorem}

\begin{proof}
This follows at once from 
    $$F([\kappa]) = 2N - 2 + \left(\tilde{\eta}^2-\frac{\tilde{\eta}^4}{4} \right)
    \int \vect{\phi}^T S~ \vect{\phi}\,\diff\mathbb{P}\left(\phi\right)
    $$
with $\tilde{\eta} = d\left([\kappa], \left[\kappa^{(0)}\right]\right)$ and $S$ having eigenvalues greater than or equal to $\lambda_2$, due to Lemma \ref{lem: A2 is constant}.  \end{proof}

\begin{corollary}
    Under Assumption \ref{ass: drift model} BPC holds for the centered homoscedastic drift model.
\end{corollary}
\begin{proof}
    Since  $E^{\rho}=\left\{\left[\kappa^{(0)}\right]\right\}$ due to Theorem \ref{theo: Frechet mean value is unique}, the BPC follows directly from Theorem \ref{theo: BPC} as $\mathfrak{P}$ is compact.
\end{proof}

\subsection{The CLT for the centered homoscedastic drift model}\label{sec: CLT Drift model}
To prove a central limit theorem for the centered homoscedastic drift model from Definition \ref{ass: drift model}, we apply Theorem 6 of \cite{HuckemannCLT}. For this we need the following additional assumption.
\begin{assumption}\label{ass: phi moment}
The random variable $\phi$ has a finite fourth moment $\mathbb{E}[(\vect{\phi}^T\vect{\phi})^2] < \infty$.
\end{assumption}
\begin{definition}\label{def: chart clt}
For $\kappa^{(0)}\in [\kappa^{(0)}]\in \mathfrak{P}$ define a unitary matrix $R\in \mathbb{C}^{N\times N}$ satisfying $R\kappa^{(0)}=e_N$ where $e_N$ is the $N$-th vector of the standard basis. For any $\kappa \in [\kappa] $ define $\tilde{\kappa}\coloneqq R \kappa$ and $\mathfrak{U} = \{[\kappa] \mid \tilde{\kappa}_N \neq 0\}$ and define the chart 
\begin{align*}
\beta: \mathfrak{U} \rightarrow \mathbb{R}^{2(N-1)},\quad 
[\kappa] \mapsto 
\left(\Re\left(\frac{\tilde{\kappa}_1}{\tilde{\kappa}_N}\right),\Im\left(\frac{\tilde{\kappa}_1}{\tilde{\kappa}_N}\right),\dots,\Re\left(\frac{\tilde{\kappa}_{N-1}}{\tilde{\kappa}_N}\right),\Im\left(\frac{\tilde{\kappa}_{N-1}}{\tilde{\kappa}_N}\right)\right).
\end{align*}
For $x \in \mathbb{R}^{2(N-1)}$ we define 
\begin{align*}
    \tilde{x}=\left(x_1 + ix_2, \dots,  x_{2(N-1)-1} + i x_{2(N-1)}, 1\right)
\end{align*}
and get
\begin{align*}
    \beta^{-1}: \mathbb{R}^{2(N-1)}  \rightarrow \mathfrak{U}, \quad  x \mapsto \left[R^*\frac{\tilde{x}}{||\tilde{x}||}\right].
\end{align*}
\end{definition}
Note that $\beta([\kappa])$ is indeed independent of the choice of representative $\kappa \in [\kappa]\in \mathfrak{U}$ and thus is well-defined. 
In a local chart $(\beta,\mathfrak{U})$ of $\mathfrak{P}$ near $\beta^{-1}(0)$, we denote the gradient of $x\mapsto \rho(Y, \beta^{-1}(x))$ by $\mathrm{grad}_2\rho(Y, [\kappa])$ and by $H_2\rho(Y, [\kappa])$ the corresponding Hesse matrix. 

\begin{theorem}[CLT]\label{theo: clt}
 For the centered homoscedastic drift model from Definition \ref{ass: drift model}, under Assumption \ref{ass: phi moment}, let $[\hat\kappa^{(B)}(\omega)] \in E_B^{(\rho)}(\omega)$ be a measurable selection for all $\omega\in \Omega$, then, omitting $\omega$,
 \begin{align*}
     \sqrt{B}\beta([\hat\kappa^{(B)}]) \stackrel{\mathcal{D}}{\rightarrow} \mathcal{N}\left(0, \left(\mathbb{E}\left[H_2\rho\left(Y, \left[\kappa^{(0)}\right]\right)\right]\right)^{-1} \left( \mathrm{cov}\left[\mathrm{grad}_2\rho\left(Y, \left[\kappa^{(0)}\right]\right)\right] \right) \left(\mathbb{E}\left[H_2\rho\left(Y, \left[\kappa^{(0)}\right]\right)\right]\right)^{-1}\right)
 \end{align*} 
 holds for the chart $\beta$ defined in Definition \ref{def: chart clt}. 
\end{theorem}

\begin{proof}
We show in this proof that the following conditions for Theorem 6 of \cite{HuckemannCLT} are satisfied.
\begin{enumerate}
        \item $x \mapsto \rho(Y, \beta^{-1}(x))$ is smooth for $|x|<\epsilon$,
        \item $\mathbb{E}[\mathrm{grad}_2\rho\left(Y, \left[\kappa^{(0)}\right]\right)]$ exists,
        \item $\mathbb{E}[H_2\rho(Y, [\kappa])]$ exists for $\kappa$ near $\kappa^{(0)}$ and is continuous at $\kappa=\kappa^{(0)}$,
        \item $\mathrm{cov}[\mathrm{grad}_2\rho\left(Y, \left[\kappa^{(0)}\right]\right)]$ exists,
        \item $\mathbb{E}[H_2\rho\left(Y, \left[\kappa^{(0)}\right]\right))]$ is invertible.
    \end{enumerate}
First, we rewrite $\rho$, see (\ref{eq: definition rho}), to
\begin{align*}
    \rhofun(Y, [\kappa])&=  \spr{Y}{Y}_P-\sum_{\nu=1}^N\sum_{\tilde{\nu}=1}^N\vect{Y_{\tilde{\nu}}}^Tf_{\tilde{\nu},\nu, P}([\kappa])\vect{Y_{\nu}}
\end{align*}
where
\begin{align*}
    f_{\tilde{\nu},\nu, P}([\kappa]) \coloneqq P M(\kappa_{\tilde{\nu}}) \left(\kappa \dia_P \kappa \right)^{-1}M(\kappa_{\nu})^TP
\end{align*}
for $\kappa \in [\kappa]$.
Using Lemma \ref{lem: inverse lemma} we get
\begin{align*}
    f_{\tilde{\nu},\nu, P}([\kappa]) =
    \frac{P M(\kappa_{\tilde{\nu}}) \left(\kappa \dia_{\tilde{P}} \kappa \right)M(\kappa_{\nu})^TP}{\det\left(\kappa \dia_P \kappa \right)}.
\end{align*}

1.) 
Both $\kappa \mapsto P M(\kappa_{\tilde{\nu}}) \left(\kappa \dia_{\tilde{P}} \kappa \right)M(\kappa_{\nu})^TP$ and $\kappa \mapsto \det\left(\kappa \dia_P \kappa \right)$ are fourth degree polynomials and it follows from Lemma \ref{lem: det lemma} that $\det\left(\kappa \dia_P \kappa \right) \geq \lambda_2\lambda_1$ for $\kappa \in \mathfrak{P}$. Using the chain rule, it follows that the function $x \mapsto f_{\tilde{\nu},\nu, P}(\beta^{-1}(x))$ is smooth. Therefore, it follows that the function $x \mapsto \rho(Y, \beta^{-1}(x))$ is also smooth. 

2.)
We start with
\begin{align*}
    \frac{\partial}{\partial x_i}\rho(Y, \beta^{-1}(x))=-\sum_{\nu=1}^N\sum_{\tilde{\nu}=1}^N\vect{Y_{\tilde{\nu}}}^T\left(\frac{\partial}{\partial x_i}f_{\tilde{\nu},\nu, P}(\beta^{-1}(x))\right)\vect{Y_{\nu}}
\end{align*}
Since $\rho$ is smooth we get 
\begin{align*}
    \mathbb{E}\left[\frac{\partial}{\partial x_i}\rho(Y, \beta^{-1}(x))\right]&=-\mathbb{E}\left[\sum_{\nu=1}^N\sum_{\tilde{\nu}=1}^N\vect{Y_{\tilde{\nu}}}^T\left(\frac{\partial}{\partial x_i}f_{\tilde{\nu},\nu, P}(\beta^{-1}(x))\right)\vect{Y_{\nu}}\right]\\
    &=-\sum_{\nu=1}^N\sum_{\tilde{\nu}=1}^N\Tr\left(\left(\frac{\partial}{\partial x_i}f_{\tilde{\nu},\nu, P}(\beta^{-1}(x))\right)\mathbb{E}\left[\vect{Y_{\tilde{\nu}}}\vect{Y_{\nu}}^T\right]\right).
\end{align*}
Since $\rho$ is smooth with respect to $x$ and we know from Assumption \ref{ass: drift model} that the second moment of $Y$ exists, it follows that $\mathbb{E}[\mathrm{grad}_2\rho\left(Y, \left[\kappa^{(0)}\right]\right)]$ exist. 

3.)
Analogously, we conclude that the following function exists for $x$ near $\beta(\kappa^{(0)})$ and is continuous at $x=\beta(\kappa^{(0)})$
\begin{align*}
    \mathbb{E}\left[\frac{\partial^2}{\partial x_i\partial x_j}\rho(Y, \beta^{-1}(x))\right]
    &=-\sum_{\nu=1}^N\sum_{\tilde{\nu}=1}^N\Tr\left(\left(\frac{\partial^2}{\partial x_i\partial x_j}f_{\tilde{\nu},\nu, P}(\beta^{-1}(x))\right)\mathbb{E}\left[\vect{Y_{\tilde{\nu}}}\vect{Y_{\nu}}^T\right]\right).
\end{align*}
Thus $\mathbb{E}[H_2\rho\left(Y, \left[\kappa^{(0)}\right]\right)]$ exists for $\kappa$ near $\kappa^{(0)}$ and is continuous at $\kappa=\kappa^{(0)}$.

4.)
From the smoothness of $\rho$ with respect to $x$ and the Assumption \ref{ass: phi moment}, the existence of
\begin{align*}
    \mathbb{E}\left[\mathrm{grad}_2\rho(Y, \left[\kappa^{(0)}\right])\mathrm{grad}_2\rho(Y, \left[\kappa^{(0)}\right])^T\right]
\end{align*}
follows, since 
$$\left(\frac{\partial}{\partial x_i}\rho(Y, \beta^{-1}(x))\right) \left(\frac{\partial}{\partial x_j}\rho(Y, \beta^{-1}(x))\right)$$
is a polynomial of fourth degree for all $i,j=1,\dots, N$ with respect to $Y$. Consequently,  $\mathrm{cov}[\mathrm{grad}_2\rho\left(Y, \left[\kappa^{(0)}\right]\right)]$ exists.

5.) 
Inserting the calculations from Lemma \ref{lem: d phi kappa} into the results of 
Lemma \ref{lem: A2 is constant}
yields:
\begin{align*}
    \Ffun(\beta(x)) &= 2N-2 + \left(d\left(\beta(x), \left[\kappa^{(0)}\right]\right)^2-\frac{d\left(\beta(x), \left[\kappa^{(0)}\right]\right))^4}{4} \right)
    \int \vect{\phi}^T S \vect{\phi}\diff\mathbb{P}\left(\phi\right)\\
    &\geq 2N-2 + \left(1-\frac{1}{\sqrt{||x||^2+1}}-\frac{\left(1-\frac{1}{\sqrt{||x||^2+1}}\right)^2}{4} \right)
    \lambda_2 c_\phi\\
    &=2N-2 + \frac{||x||^2}{2}\lambda_2 c_\phi + \mathcal{O}(||x||^3).
\end{align*}
It follows from the Taylor expansion of $\Ffun(\beta(x))$ at $0$:
\begin{align*}
    x^T\mathbb{E}[H_2\rho\left(Y, \left[\kappa^{(0)}\right]\right)]x \geq ||x||^2\lambda_2 c_\phi+ \mathcal{O}(||x||^3).
\end{align*}
Consequently $\mathbb{E}[H_2\rho\left(Y, \left[\kappa^{(0)}\right]\right)]$ is invertible.
\end{proof}

\subsection{The CLT for the spectrum $I$}\label{sec: clt I}
In the ENDOR experiment, one is particularly interested in the spectrum $I$ (see Figure~\ref{fig:DriftModelFit}, panel~A). 
Different rotation methods are possible to extract the estimated spectrum $\hat{I}$ from the maximum likelihood estimator $\hat{\kappa}$. As discussed in Section \ref{sec: maximum likelihood estimators}, in 
this section (as well as throughout the paper) we work with the maximum method and use the notation  $\hat{I}=\Re(e^{i\hat\lambda}\hat{\kappa})$.
Lemma \ref{lem: rotation max method} provides an explicit formula for computing $\hat\lambda$ according to the maximum method, from which we derive the function in equation (\ref{eq: ftilde}) below, which maps $\kappa$ to an optimally rotated $I$. This explicit function is used in Corollary \ref{cor: CLT I}, which provides a central limit theorem for $\hat{I}$.  

In fact, in conjunction with the chart $\beta$ from Definition \ref{def: chart clt} we will construct functions $g,f,f_\pm $ making the diagram below commutative (on the corresponding domains) and smooth outside singularity sets $M_1\cup M_2$ (defined in (\ref{eq: M1 and m2 sets})) in $\mathfrak{P}$:

$$
\begin{array}{rcl}
\CC P^{N-1}=\mathfrak{P}&\stackrel{f}{\to}&\mathfrak{I} = \RR^N/\sim_\pm\\
\beta\downarrow && \downarrow f_\pm\\
\RR^{2N-2}&\stackrel{g}{\to}& \RR^N´
\end{array}
$$

For this purpose, we define $\mathbb{S}^1\coloneqq[0,2\pi]/\sim$ where "$\sim$" denotes 
\begin{align*}
    x_1 \sim x_2 \quad \Leftrightarrow \quad x_1=x_2 \quad \text{or} \quad x_1,x_2\in \{0, 2\pi\}
\end{align*}
From Lemma \ref{lem: rotation max method} in the \appendixname, with the definition for the complex argument from Section \ref{sec: notation} follows
\begin{align*}
        \argmax_{\lambda \in \mathbb{S}^1} \left|\left|\Re(e^{i\lambda }\kappa) \right|\right|^2 = \begin{cases}
        \left\{\pi-\frac{\Arg(\kappa^T\kappa)}{2}, 2\pi-\frac{\Arg(\kappa^T\kappa)}{2}\right\}, & \text{if } \kappa^T\kappa\neq 0\\
        \mathbb{S}^1, & \text{else}
        \end{cases}
    \end{align*}
for $\kappa \in \Scomp$.
Note that $\Re(e^{i(\lambda + \pi)}) = -\Re(e^{i\lambda})$, which leads to the metric space  $(\mathfrak{I}, d_\mathfrak{I})$ given by $\mathfrak{I} \coloneqq \mathbb{R}^N/\sim_\pm$,  where the equivalence relation "$\sim_\pm$" is defined as
\begin{align*}
I \sim_\pm I' \quad \Leftrightarrow \quad I=I' \quad \text{or} \quad I=-I'.
\end{align*}
and 
\begin{align*}
    d_\mathfrak{I}([I]_\pm, [I']_\pm) = \min_{k\in \{0, 1\}} \left|\left|I-(-1)^kI' \right|\right|^2 
\end{align*}
where $I\in [I]_\pm, I'\in[I']_\pm$.
This gives rise to the function \begin{align}\label{eq: ftilde}
    \tilde{f}:\Scomp\rightarrow\mathfrak{I}, \quad
    \kappa \mapsto 
    \begin{cases}
    \left[\Re\left(e^{\frac{-i}{2}\Arg(\kappa^T\kappa) }\kappa\right)\right]_\pm& \text{for } \kappa^T\kappa \neq 0,\\
    [(0,\dots, 0)]_\pm & \text{else}.
    \end{cases}
\end{align}
In Lemma~\ref{lem: f well defined} in the \appendixname, we show that the function $\tilde{f}$ is well-defined for $\mathfrak{P}$, which means that $\tilde{f}(\kappa) = \tilde{f}(\tilde{\kappa})$ for all $\kappa, \tilde{\kappa} \in [\kappa]$.
Therefore, we can define a function
    \begin{align*}
    f : \mathfrak{P} \rightarrow \mathfrak{I}, \quad
    [\kappa] \mapsto \tilde{f}(\kappa).
\end{align*}
To obtain the spectrum $I\in\mathbb{R}^N$ from $f([\kappa])=[I]_\pm\in \mathfrak{I}$, an additional sign flip is performed, if necessary, as one usually wants the peaks to be in the positive direction. This can be uniquely achieved under the condition 
\begin{align*}
    \left|\max_{\nu = 1, \dots N} I_\nu\right| > \left|\min_{\nu = 1, \dots N} I_\nu\right|,
\end{align*}
using the following sign flip function 
\begin{align}\label{eq: function g}
    f_{\pm}: \mathfrak{I} \rightarrow \mathbb{R}^N,\quad [I]_\pm \mapsto \begin{cases}
    I, & \text{if } \left|\max_{\nu = 1, \dots N} I_\nu\right| > \left|\min_{\nu = 1, \dots N} I_\nu\right| \\
    -I, & \text{if } \left|\max_{\nu = 1, \dots N} I_\nu\right| < \left|\min_{\nu = 1, \dots N} I_\nu\right| \\
    0, & \text{else.}
    \end{cases}
\end{align}
Using the map $\beta$ from Definition \ref{def: chart clt} we finally define 
\begin{align*}
    g: \mathbb{R}^{2N-2} \rightarrow \mathbb{R}^N, \quad x \mapsto  f_{\pm}(f(\beta^{-1}(x))).
\end{align*}
Ensuring that the functions $\widetilde{f}$ and $f_\pm$ are smooth we excluded the following two singularity sets 
\begin{equation}
\begin{aligned}
    M_1 \coloneqq&~ \{[\kappa]\in \mathfrak{P} \mid \exists \kappa\in[\kappa] \text{ s.t. } \kappa^T\kappa = 0\},\\
    M_2 \coloneqq&~ \left\{[\kappa]\in \mathfrak{P} \mid \exists \kappa\in[\kappa] \text{ s.t. } \left|\max_{\nu = 1, \dots N} f(\kappa)_\nu\right|  = \left|\min_{\nu = 1, \dots N} f(\kappa)_\nu\right|\right\}. \label{eq: M1 and m2 sets}
\end{aligned}
\end{equation}
Note that the defining relations in the Equation (\ref{eq: M1 and m2 sets}) are independent of representative, i.e. "$\exists$" can be replaced with "$\forall$" without loss of generality. With $R$ from Definition \ref{def: chart clt}, the Jacobian matrix of the function $g$ for $[\kappa^{(0)]} \in \mathfrak{P}\setminus(M_1 \cup M_2)$ at location $x=0$ is computed in Lemma~\ref{lem: jacobian g} in the \appendixname~and is of the following form
\begin{align*}
J_x g(0)&= \pm\Re\Biggl(e^{-\frac{i\alpha}{2}}\Biggl(i\kappa^{0}\Re\left(i\frac{\overline{(\kappa^{(0)})^T\kappa^{(0)}}\left(\kappa^{(0)}\right)^TR^*A}{r^2}\right) + R^*A\Biggl) \Biggl)\in \mathbb{R}^{N\times 2(N-1)}.
\end{align*}
where $\left|(\kappa^{(0)})^T\kappa^{(0)}\right| =: r$, $\alpha \coloneqq \Arg\left((\kappa^{(0)})^T\kappa^{(0)}\right)$ and 
\begin{align}\label{eq: A matrix}
        A \coloneqq \begin{pmatrix}
    1 & i & 0 & 0& 0&\dots & 0& 0 \\
    0 & 0 & 1 & i& 0&\dots & 0& 0 \\
    \vdots & \vdots & \vdots & \vdots&\vdots  & \vdots& \vdots& \vdots \\
     0 & 0 & 0 & 0& 0&\dots & 1& i \\
     0 & 0 & 0 & 0& 0&\dots & 0& 0 \\
    \end{pmatrix}\in \mathbb{C}^{N\times 2(N-1)}.
\end{align}

\begin{rem}
    In our applications, we observed a Jacobi-matrix $J_x g(0)$ of full rank (see plots of the singular values for the different orientations from a chemical sample of the D2-$Y_{122}^{\bullet}$ E.~coli ribonucleotide reductase in Figure \ref{fig: singular values jacobian} in the \appendixname).
    However, this may not be the case in general. For example, for $\kappa^{(0)}=e_N$ we have $R= \mathrm{Id}_N$ so that $J_x g(0)=\pm\Re(A)$, which has rank $N-1$. 
\end{rem}

\begin{corollary}\label{cor: CLT I}
    If $[\kappa^{(0)}]\in \mathfrak{P}\setminus(M_1 \cup M_2)$, in the centered homoscedastic drift model from Definition \ref{ass: drift model}, we have the true spectrum $I^{(0)} = g\circ \beta([\kappa^{(0)}])$ with $\beta$ from Definition \ref{def: chart clt} and under Assumption \ref{ass: phi moment} with the estimator $\hat I^{(B)} = g\circ \beta([\hat\kappa^{(B)}])$ from a measurable selection $[\hat\kappa^{(B)}]\in {E}^{(\rho)}_B$ that
    \begin{align*}
    \sqrt{B}\left( \hat I^{(B)} - I^{(0)}\right)
        \stackrel{\mathcal{D}}{\rightarrow} \mathcal{N}\left(0, \Big(J_x g(0)\Big)\mathfrak{G}_\beta \Big(J_x g(0)\Big)^T\right)
    \end{align*}
    where $\mathfrak{G}_\beta = \left(\mathbb{E}\left[H_2\rho\left(Y, \left[\kappa^{(0)}\right]\right)\right]\right)^{-1} \left( \mathrm{cov}\left[\mathrm{grad}_2\rho\left(Y, \left[\kappa^{(0)}\right]\right)\right] \right) \left(\mathbb{E}\left[H_2\rho\left(Y, \left[\kappa^{(0)}\right]\right)\right]\right)^{-1}$ and $\beta$ is defined as in Definition \ref{def: chart clt}.
\end{corollary}

\begin{proof}
Follows directly from Theorem \ref{theo: clt} and Lemma \ref{lem: jacobian g} using the delta method (see, for example, Section 3 of \cite{vanderVaar2000astat}).
\end{proof}

\newpage
\section{Inconsistency for joint estimation of $\kappa$ and $\Sigma$ in the homoscedastic drift model}\label{sec: Consistency Drift model unknown Sigma} 
In contrast to Section \ref{sec: Drift model: Strong consistency and central limit theorem}, in this section we do not work with the assumption that $\Sigma$ is known, but we investigate the more complicated case that $\Sigma$ and $\kappa$ are estimated simultaneously.
To this end, we reformulate in Definition \ref{ass: drift model sigma} the Centered Homoscedastic Drift Model from Definition \ref{ass: drift model} from Section \ref{sec: Drift model: Strong consistency and central limit theorem} by introducing the true but unknown $\Sigma^{(0)}$. 
Particularly, this section demonstrates that the joint estimation of $\kappa$ and $\Sigma$ is not consistent. 
For ease of notation, as in Section \ref{sec: Drift model: Strong consistency and central limit theorem}, we assume that the original random variable is $N+1$ dimensional, so that $Y$ below is $N$ dimensional and we omit writing the tilde and the superscript $^H$ symbol from Section \ref{sec: maximum likelihood estimators}. 

\begin{definition}[Centered Extended Homoscedastic Drift Model]\label{ass: drift model sigma}
The complex $N$-dimensional random vector $Y$ is given by
\begin{align}\label{eq: y with unknown P}
    Y=\phi\kappa^{(0)} +\epsilon
\end{align}
where $\kappa^{(0)}\in \Scomp \coloneqq \{\kappa \in \mathbb{C}^N: ||\kappa||=1\}$ comprises the true but unknown ENDOR spectrum, $\phi$ is  a complex random value and $\epsilon = (\epsilon_1,\ldots,\epsilon_N)$ is a complex $N$-dimensional random vector independent of $\phi$ 
with 
\begin{align*}
    0 < \mathbb{E}[\vect{\phi}^T\vect{\phi}]=c_\phi < \infty, \qquad \vect{\epsilon_{\nu}} \stackrel{i.i.d.}{\sim} \mathcal{N}(0,\Sigma^{(0)}), 1\leq \nu \leq N .
\end{align*}
Moreover, we assume that the precision matrix $P^{(0)}=\left(\Sigma^{(0)}\right)^{-1}$ has two positive eigenvalues $\lambda_1>\lambda_2>0$.

Thus, $\mathfrak{Q} = \mathbb{C}^N$ is the data space. Since, in contrast to Section \ref{sec: Drift model: Strong consistency and central limit theorem}, we additionally want to estimate the strictly positive definite symmetric matrix $\Sigma$, we obtain the following parameter space
\begin{align*}
    \mathfrak{P}\coloneqq\mathbb{C}P^{N-1} \times \mathrm{SPD}(2)\,,
\end{align*}
and for the loss we choose
\begin{align}\label{eq: definition rho with P}
    \rhofun(Y, ([\kappa],P)) = d_P\left(Y,  \hat{\phi}(\kappa,P,Y)\kappa\right)^2-N\log(\det(P))
\end{align}
with
\begin{align*}
    \hat{\phi}(\kappa,P,Y) &= \comp{\left(\kappa \dia_P \kappa  \right)^{-1} \left(\kappa\bul_P Y\right)} \in \mathbb{C}.
\end{align*}

Thus, for a sample $Y(1), Y(2), \dots \stackrel{i.i.d.}{\sim} Y$ from (\ref{eq: y with unknown P}) we have the following sample and population Fr\'echet functions
\begin{align*}
    \Ffun_B (\omega, ([\kappa], P)) = \frac{1}{B}\sum_{b=1}^B\rhofun(Y(b), [\kappa]), \qquad \Ffun([\kappa], P) = \int\rhofun(Y, [\kappa])\diff\mathbb{P}\left(\phi,\epsilon\right).
\end{align*}
\end{definition}
Analogous to Lemma \ref{lem: A2 is constant}, 
we decompose the first term of $\Ffun([\kappa], P)$ into the $\epsilon$ and $\phi$ parts
\begin{align*}
    &\int d_P\left(Y,  \hat{\phi}(\kappa,P,Y)\kappa\right)^2 
    \diff\mathbb{P}\left(\phi,\epsilon\right) \\ 
    &= \int \Bigg( d_P\left(\phi \kappa^{(0)},  \hat{\phi}(\kappa,P,\phi \kappa^{(0)})\kappa\right)^2 
    + 2\spr{\phi \kappa^{(0)}- \hat{\phi}(\kappa,P,\phi \kappa^{(0)})\kappa}{\epsilon- \hat{\phi}(\kappa,P,\epsilon)\kappa}_P \\
    &\hspace{8ex}+d_P\left(\epsilon,  \hat{\phi}(\kappa,P,\epsilon)\kappa\right)^2\Bigg)\diff\mathbb{P}\left(\phi, \epsilon\right),
\end{align*}
for any $\kappa \in [\kappa]$.
Since $\phi$ and $\epsilon$ are independent and $\mathbb{E}[\epsilon]=0$ the integral over the mixed term vanishes and consequently
\begin{align*}
    &\Ffun([\kappa], P)= \int d_P\left(\phi \kappa^{(0)},  \hat{\phi}(\kappa,P,\phi \kappa^{(0)})\kappa\right)^2  \diff\mathbb{P}\left(\phi\right)
    +\int d_P\left(\epsilon,  \hat{\phi}(\kappa,P,\epsilon)\kappa\right)^2\diff\mathbb{P}\left(\epsilon\right) -N\log(\det(P)).
\end{align*}
In contrast to Section \ref{sec: Strong Consistency drift model}, the expression
\begin{align*}
    \int d_P\left(\epsilon,  \hat{\phi}(\kappa,P,\epsilon)\kappa\right)^2\diff\mathbb{P}\left(\epsilon\right)= N\Tr\left(\Sigma^{(0)}P\right) -\Tr\left( \left(\kappa \dia_P \kappa  \right)^{-1} \left(\kappa \dia_{P\Sigma^{(0)}P} \kappa \right)\right)
\end{align*} 
depends on $\kappa$, see Lemma \ref{lem: A2 is not constant} in the \appendixname.
In Lemma \ref{lem: Partial_derivative P} in the \appendixname
\begin{align*}
    \frac{\partial}{\partial P}\left(N\Tr\left(\Sigma^{(0)}P\right) -\Tr\left( \left(\kappa \dia_P \kappa  \right)^{-1} \left(\kappa \dia_{P\Sigma^{(0)}P} \kappa \right)\right)-N\log(\det(P))\right)
\end{align*}
is calculated and from Lemma \ref{lem: insert in kappa0 P0} in the \appendixname~follows 
\begin{align*}
    &\frac{\partial}{\partial P}\Ffun([\kappa], P)\Big|_{[\kappa]=\left[\kappa^{(0)}\right], P=P^{(0)}}
    = 2\left(\bar{\kappa}^{(0)} \dia_{\left(\kappa^{(0)} \dia_{P^{(0)}} \kappa^{(0)}  \right)^{-1}} \bar{\kappa}^{(0)} \right)  -\mathrm{diag}\left(\bar{\kappa}^{(0)} \dia_{\left(\kappa^{(0)} \dia_{P^{(0)}} \kappa^{(0)}  \right)^{-1}} \bar{\kappa}^{(0)} \right).
\end{align*}
In general, $\bar{\kappa}^{(0)} \dia_{\left(\kappa^{(0)} \dia_{P^{(0)}} \kappa^{(0)}  \right)^{-1}} \bar{\kappa}^{(0)}$ is not  equal to $0$.
Thus, in general, it does not hold that $E=\left\{\left(\left[\kappa^{(0)}\right], P^{(0)}\right)\right\}$.  This shows that the matrix $\Sigma$, which describes the random vector $\epsilon$, cannot be estimated consistently from the profile likelihood. To achieve a jointly consistent estimator for $\kappa$ and $\Sigma$, one would heuristically expect that a proper treatment of the randomness in both $\epsilon$ and $\phi$ is required, which the profile likelihood does not provide for $\phi$. See \autoref{sec: outlook} for a fuller discussion. 

\newpage
\section{Heteroscedastic Drift Model} \label{sec: heteroscedastic drift} 
The homoscedastic drift model has been found to fit spectroscopic data recorded at MW frequency \SI{263}{\GHz} well, across a range of RF frequencies (\SI{20}{\MHz}-\SI{400}{\MHz}) and nuclei ($^1$H, $^2$H, $^{19}$F). \cite{Pokern2021,Hiller2022,Wiechers2023}. However, application to the lower MW frequency of \SI{94}{\GHz}, more commonly encountered in biochemistry groups, reveals very poor fit arising from the noise containing a {\em phase noise} component that is affected by phase drift. This necessitated development of a heteroscedastic drift model that we will detail in this section. It exhibits much improved fit and again results in improved SNR compared to the averaging model.

\subsection{Modelling of the Heteroscedastic Drift Model}
We test the homoscedastic drift model with ENDOR data recorded at a MW frequency of \SI{94}{\GHz} targeting the $^2$H resonance in the twice deuterated $Y_{122}$ Tyrosyl radical \cite{Hiller2022} and using the Mims pulse sequence, see \cite{Mims1965}. This pulse sequence is known for yielding strong EPR echos, so we expect $|\psi_b|$ to be large. As in \autoref{sec: homoscedastic drift}, the goodness of fit was assessed by applying Kolmogorov-Smirnov tests to the real and imaginary parts of the standardized residuals for each of the five datasets (orientations $g_x,g_{xy},g_y,g_{yz},g_z$) yielding the $p-$values reported in \autoref{tab:hom_p_values_94} in \autoref{ssec: heteroscedastic GOF} in the \appendixname. For all orientations except $g_z$, at least one of the two p-values falls far below the Bonferroni-corrected significance level of $0.005$, with some $p-$values of order $10^{-20}$ indicating very poor fit. This lack of fit can also be observed from the kernel density estimates and q-q-plots shown in \autoref{fig: Goodness of Fit hom}.
\begin{figure}[h]
  \centering
  \includegraphics{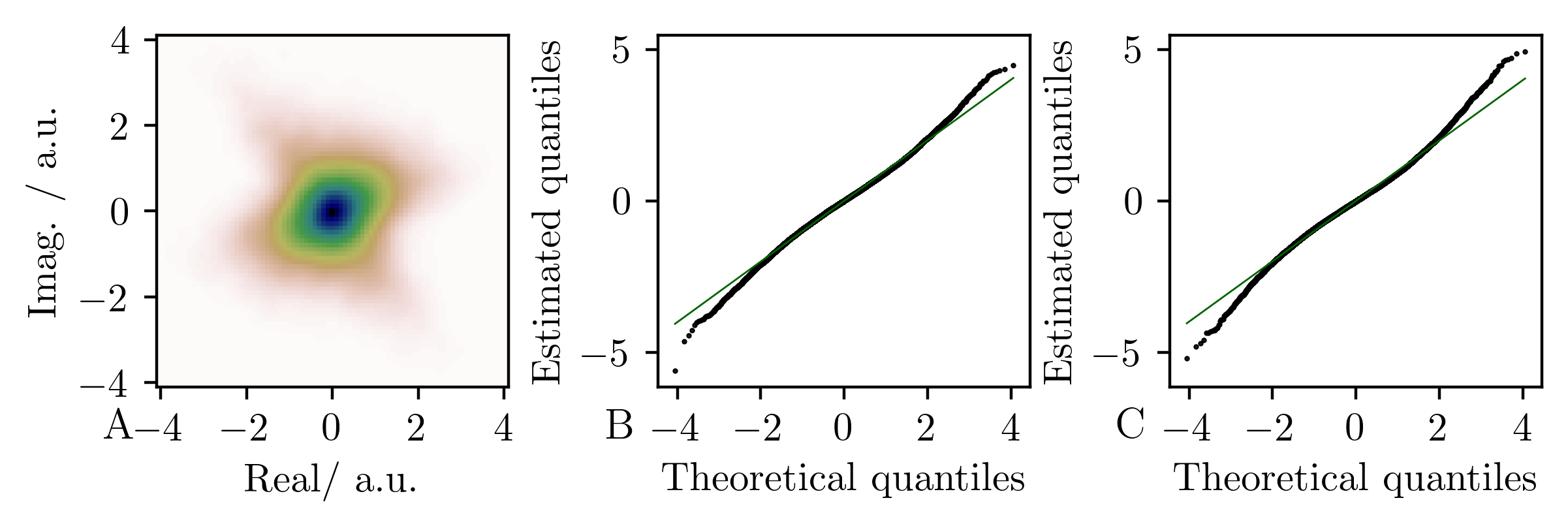}
  \caption{Results of the goodness of fit methods for the homoscedastic drift model applied to \SI{94}{\GHz} Mims data. Panel A displays the kernel density estimator applied to the standardized residuals $\hat{\Sigma}^{-\frac{1}{2}}\hat{\epsilon}_{b,\nu}$ 
    with $\hat{\Sigma}^{-\frac12}$ the inverse of the matrix square root of $\hat{\Sigma}_b$. Panels B and C show q-q-plots for the real part (B) and the imaginary part (C) of the standardized residuals against a standard normal (black) and a reference of perfect fit (green).}
  \label{fig: Goodness of Fit hom}
\end{figure}
Further examination of the residuals shown in Panel A of \autoref{fig: examination heteroscedastic drift model} hints at an underlying heteroscedastic noise structure w.r.t. the batches. 
Taking a general batch-dependent covariance matrix for the noise $\epsilon_{b,\nu}\overset{ind}{\sim}\mathcal{N}(0,\Sigma_b), \Sigma_b\in \mathrm{SPD}(2)$, constitutes a very flexible extension of the homoscedastic drift model which is non-parametric in the sense that the number of parameters $\{\Sigma_b|b\in\{1,\ldots,B\}\}$ increases with the amount of data available. However, the likelihood of this model has boundary maxima which can be obtained by choosing the parameters $\phi,\kappa$ such as to yield zero residuals $\hat{\epsilon}_{b^*,\nu} = 
0 $ for one particular batch $b^*$ and letting  the covariance matrix $\Sigma_{b^{*}}$ tend to zero resulting in $-\frac{1}{2}\log\mathrm{\det}\Sigma_{b^{*}}$ tending to infinity. 
Such an approach therefore needs additional penalization for the $\Sigma_b$ resulting in shrinkage and a parametric extension was pursued instead. Based on the empirical observation presented in Panels B and C of \autoref{fig: examination heteroscedastic drift model} that the batch-wise principal component of the homoscedastic residuals $\hat{\epsilon}_{b,:}$ is rotated by \ang{90} compared to the spectrum mean $\vect{\hat{\psi}_b}$, the noise was modelled as a sum of a homoscedastic noise source and one whose covariance is given as a function of $\vect{\psi}$. This batch-dependent noise is attributed to the phase noise of the EPR echo $\psi_b$ as phase noise is known to be orthogonal in phase and proportional in amplitude to the carrier signal it arises from \cite{hagen_2009}. Both these properties of phase noise are empirically found to apply to the residuals of the homoscedastic drift model arising from our data: see panels A and B of Figure \ref{fig: examination heteroscedastic drift model} for orthogonality and panel C of that figure for amplitude.

The expansion of a phase noise term modeled as a wrapped Gaussian 
\begin{align*}
\Tilde{\psi}_{b,\nu} = \psi_b\exp{i\tilde{\sigma} \varphi_{b,\nu}}\\
\varphi_{b,\nu}\overset{i.i.d}{\sim}\mathcal{N}(0,1)
\end{align*}
in small $\tilde{\sigma} > 0 $ up to linear order gives
\begin{align}
\vect{\Tilde{\psi}_{b,\nu}} = \vect{\psi_b}+\tilde{\sigma}\varphi_{b,\nu}\vect{i\psi_b}+\mathcal{O}_P(\tilde{\sigma}^2).\label{eqn:phasenoise_linear}
\end{align}
This random variable has mean $\vect{\psi_b}+\mathcal{O}_P(\tilde{\sigma}^2)$ and covariance matrix $\tilde{\sigma}^2\vect{i\psi_b}\vect{i\psi_b}^T+\mathcal{O}_P(\tilde{\sigma}^4)$. 
So, this expansion reproduces the homoscedastic drift model mean $\psi_b$ and the empirical orthogonality structure of the residuals to linear and quadratic order in $\tilde{\sigma}$, respectively.
It adds a further dependency of the moments on $\psi_b$ and therefore $\hat{\psi}_b = \frac{1}{N+1}\sum_{\nu=0}^NY_{b,\nu}$ is not the MLE estimate anymore. 
In the homoscedastic case, we needed the condition in \autoref{eq: meanSpecZero} of standardized mean $0$ spectra for identifiability. To retain this standardization for $\kappa$, we introduce an additional parameter, the spectrum mean $c\in\mathbb{C}$. In total, the heteroscedastic drift model hence decomposes the data matrix as follows:
\begin{definition}[Heteroscedastic Drift Model]\label{def: hetero}
  \begin{align*}
  Y_{b,\nu}
  &= \psi_b + \phi_b (\kappa_{\nu}+c) +\epsilon_{b,\nu}\\
  \epsilon_{b,\nu}
  \overset{ind.}&{\sim} \mathcal{N}\qty(0,\Sigma_b)\\
  \Sigma_b
  &= \Sigma_0 + \Tilde{\sigma}^2\vect{i\psi_b}\vect{i\psi_b}^T.
  \end{align*}
\end{definition}
Even though this model also has boundary maxima in the limit when $\Sigma_0$ is rank deficient, see \autoref{lem: boundary max} in the \appendixname, this is easily overcome by specifying lower bounds on the eigenvalues of $\Sigma_0$ that arise from reasonable estimates of minimal MW receiver noise. No penalization is needed in practice to enforce these bounds when starting optimizers from parameter estimates derived from the homoscedastic drift model, see \autoref{ssec: heteroscedastic boundary maximum} in the \appendixname\, for details. 

\begin{figure}[h]
  \centering
  \includegraphics[width=\textwidth]{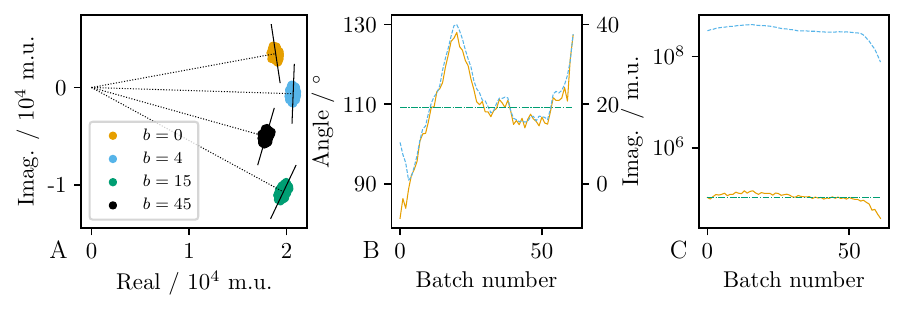}
  \caption{Examination of the $\SI{94}{\GHz}$ data using the results of the homoscedastic drift model. Panel A shows the comparison of the principal component (direction of black line segment) of the data from one batch $Y_{b,:}$. The batch mean ${\hat{\psi}_b}$ is visualized as the dotted line segment $\overline{0\hat{\psi}_b}$. Panel B shows the  angle with the real axis of the principal component of the residuals  per batch (orange) and the eigenvector corresponding to the maximal eigenvalue of $\hat{\Sigma}$  (constant, green dot-dashed) both quantified on the left y-axis and the angle of $\hat{\psi}_b$ (light blue dashed) quantified on the right y axis which is offset by \ang{90}. In Panel C, the magnitude of the largest singular value of the residuals per batch (orange), $\abs{\hat{\psi}_b}^2$ (light blue dashed) and the largest eigenvalue of $\hat{\Sigma}$ (constant green dot-dashed) are depicted.}
  \label{fig: examination heteroscedastic drift model}
\end{figure}
In the heteroscedastic drift model, we truncate the expansion of the phase noise term in $\tilde{\sigma}$ at the linear order for the mean and quadratic order for the variance, respectively. Careful comparison of higher order terms with empirically observed values of $\Sigma_0$, see \autoref{ssec: heteroscedastic truncation} in the \appendixname, reveals that the former are at least two orders of magnitude smaller than the latter which justifies our chosen truncation.

\subsection{Results of the Heteroscedastic Drift Model}
The algorithmic implementation of heteroscedastic drift model estimates a local MLE by iteratively updating the parameters by their conditional MLE. In contrast to the homoscedastic drift model, the conditional MLE's have to be approximated for $\Sigma_0,\Tilde{\sigma}$ and $\psi$. In order to improve convergence properties of the algorithm, we included an additional step wherein we calculate the conditional MLE for the parameter $\Delta_c$ where $\Tilde{\psi} = \psi - \Delta_c\phi$, $\Tilde{c} = c + \Delta_c$, see \autoref{sec: heteroscedastic appendix} of the \appendixname\, for the algorithm and further details.

Finding the optimal rotation and flip of $\hat{\kappa}$ to obtain the final spectrum $\hat{I}$ is done as in the homoscedastic drift model. The algorithm will report if one eigenvalue of $\Sigma_0$ is lower than the empirical cut-off $\delta = \num{1e-20}$.
The results of applying this algorithm to the \SI{94}{\GHz} data can be found in Data Example \ref{dexa: hetero}. 

\begin{dexa}[Heteroscedastic Drift Model for \SI{94}{\GHz} Dataset]\label{dexa: hetero}
  The result of applying \autoref{alg:het} to the \SI{94}{\GHz} dataset are shown in \autoref{fig: results het drift model94}. The goodness of fit methods are applied to the real and imaginary part of the standardized residuals $\Tilde{\epsilon}_{b,\nu} = \hat{\Sigma}_b^{-\frac12}\qty(Y_{b,\nu}-\hat{\psi}_b-\hat{\phi}_b\hat{\breve{\kappa}}_{\nu})$. The results of the kernel-density estimation are in Panel F and the q-q plots in Panel G and H of \autoref{fig: results het drift model94}. The Kolmogorov-Smirnov test is carried out at the Bonferroni-corrected level of $0.005 = 0.05/10$. The p-values can be found in \autoref{tab:heteroscedastic p values}. The model is narrowly rejected by the K-S-test as the p value in orientation $y$ is significant. Still, the graphical goodness of fit results are improved over the results from applying the homoscedastic drift model. The SNR estimated from the heteroscedastic drift model is visibly larger than the one from the averaging model in all orientations as can be seen in \autoref{fig:Signal to noise ratio heterovs average} and \autoref{tab:SNR hetero, avg} in the \appendixname. Thus, while there is potential improvement to be gained by further modelling, the heteroscedastic drift model is already a successful extension of the homoscedastic drift model.
  \begin{figure}[h]
    \centering
    \includegraphics[width=\textwidth]{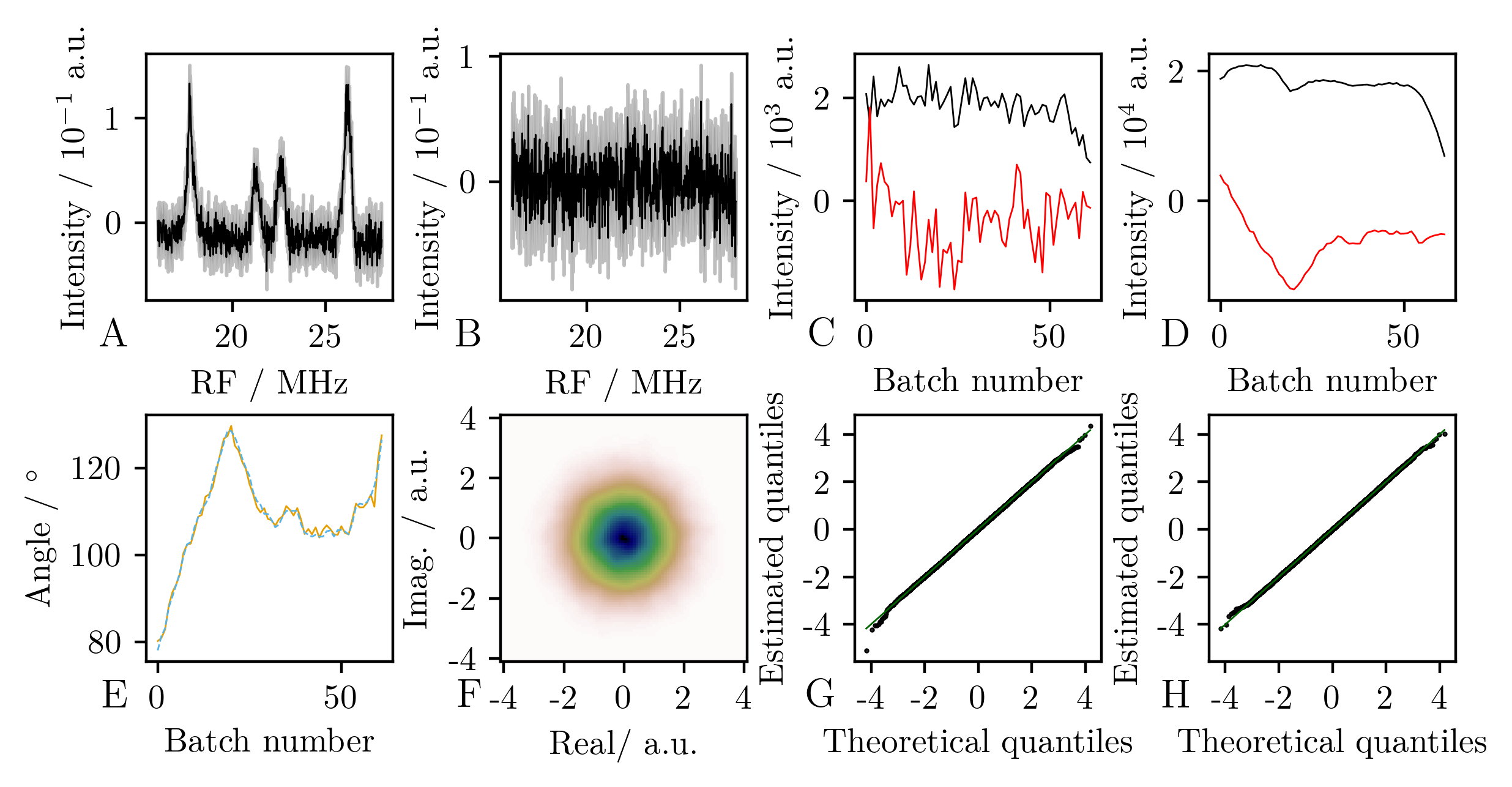}
    \caption{The results of applying the heteroscedastic drift model to the \SI{94}{\GHz} data.
      Panel A displays the estimated spectrum $\hat{I}$, while panel B  displays the component $\hat{\omega}$ that is orthogonal to the estimated spectrum $\hat{I}$. Both panels show Bootstraped confidence intervalls based on 50 Bootstrap samples with out bias reduction.
      C and D show the real (black) and imaginary (red) components of $\exp{i\alpha_{opt}}\hat{\phi}$ and $\hat{\psi}$, respectively, where the $\alpha_{opt}$ was chosen to maximize the correlation between the rotated $\hat{\phi}$ and $\hat{\psi}$. Panel E shows the angle of the principal component of the residuals  per batch (orange) and  $\psi_b$ (light blue dashed) with the real axis. 
      Panel F displays the kernel-density-estimation of the complex residuals standardized residuals $\Tilde{\epsilon}_{b,\nu}$, while panels G and H depict q-q-plots for the real and imaginary components of the residuals, respectively in black with the identity shown in green.}
    \label{fig: results het drift model94}
  \end{figure}
\end{dexa}

\newpage
\section{Outlook}\label{sec: outlook}
For the homoscedastic drift model, asymptotic theory was developed for the case where $\Sigma$ is known (see Section \ref{sec: Drift model: Strong consistency and central limit theorem}). Since the joint estimate of $\kappa$ and $\Sigma$ is not consistent in the profile likelihood model (see Section \ref{sec: Consistency Drift model unknown Sigma}), it would be desirable to obtain a consistent estimate by including the randomness of $\phi$ in the statistical model. Possible approaches are to model the $\{\phi_b\}_{b=1}^B$ as i.i.d. Gaussian or, to reflect the likelihood being invariant under permutations of batches, as  exchangeable random variables or, perhaps most realistically, as a Gaussian process. 
For the latter two approaches, one would need to generalize the theory about generalized strong consistency of generalized Fr\'echet means (see Section \ref{sec: strong consistency}) for random variables that are not i.i.d..
Furthermore, an asymptotic analysis for the heteroscedastic model is future work. Here, particular challenges arise as the mean $\psi_b$ and the variance $\Sigma_b$ are dependent on each other.
In addition, it is challenging to develop drift models for all microwave frequencies and pulse sequences to make them usable for a large audience. Initial work on Davies $\nu_{MW}=94$\,GHz data (a special pulse sequence, see \cite{Davies1974}) shows the heteroscedastic drift model not to fit well in this case, likely due to cancellation of the main echo signal leading to small and noisy $\hat{\psi}_b$. In addition, there are other experiments at $\nu_{MW}=34$\,GHz and $\nu_{MW}=9$\,GHz for which drift models are not yet available. A possible avenue may be separate modelling of mean and variance via $\Sigma_b = \Sigma_0 + \alpha_b\alpha_b^T$ with $\alpha_b\in\mathbb{R}^2$ to be estimated which may subsume homoscedastic and heteroscedastic noise models and would also apply to pulse sequences where $\hat{\psi}$ is afflicted by noise and cancellation effects.

\section{Acknowledgements}
H.W., B.E., S.H., M.B. and Y.P. thank the DFG --- project-ID 432680300 --- CRC 1456 for financial support. M.B. acknowledges the ERC Advanced Grant 101020262 BIO-enMR. 
We thank the Max Planck Society for financial support. S.H. acknowledges the Niedersachsen Vorab of the Volkswagen foundation, DFG-HU 1575/7 and the IMSI workshop on Object Oriented Data Analysis in Health Sciences 2023. Y.P. gratefully acknowledges Royal Society International Exchanges grant IE150666. 
\newpage
\bibliographystyle{alpha}
\bibliography{bibtex}
\newpage
\begin{appendices}
  \section{Homoscedastic drift model}\label{sec: drift model plots}
  
  The real part and imaginary part of the raw data matrix $Y$ for orientation $g_y$ from a chemical sample of D2-$Y_{122}^{\bullet}$ are presented in Figure \ref{fig: raw data}.
  \begin{figure}[ht!]
    \centering
    \includegraphics[width = 0.49\textwidth]{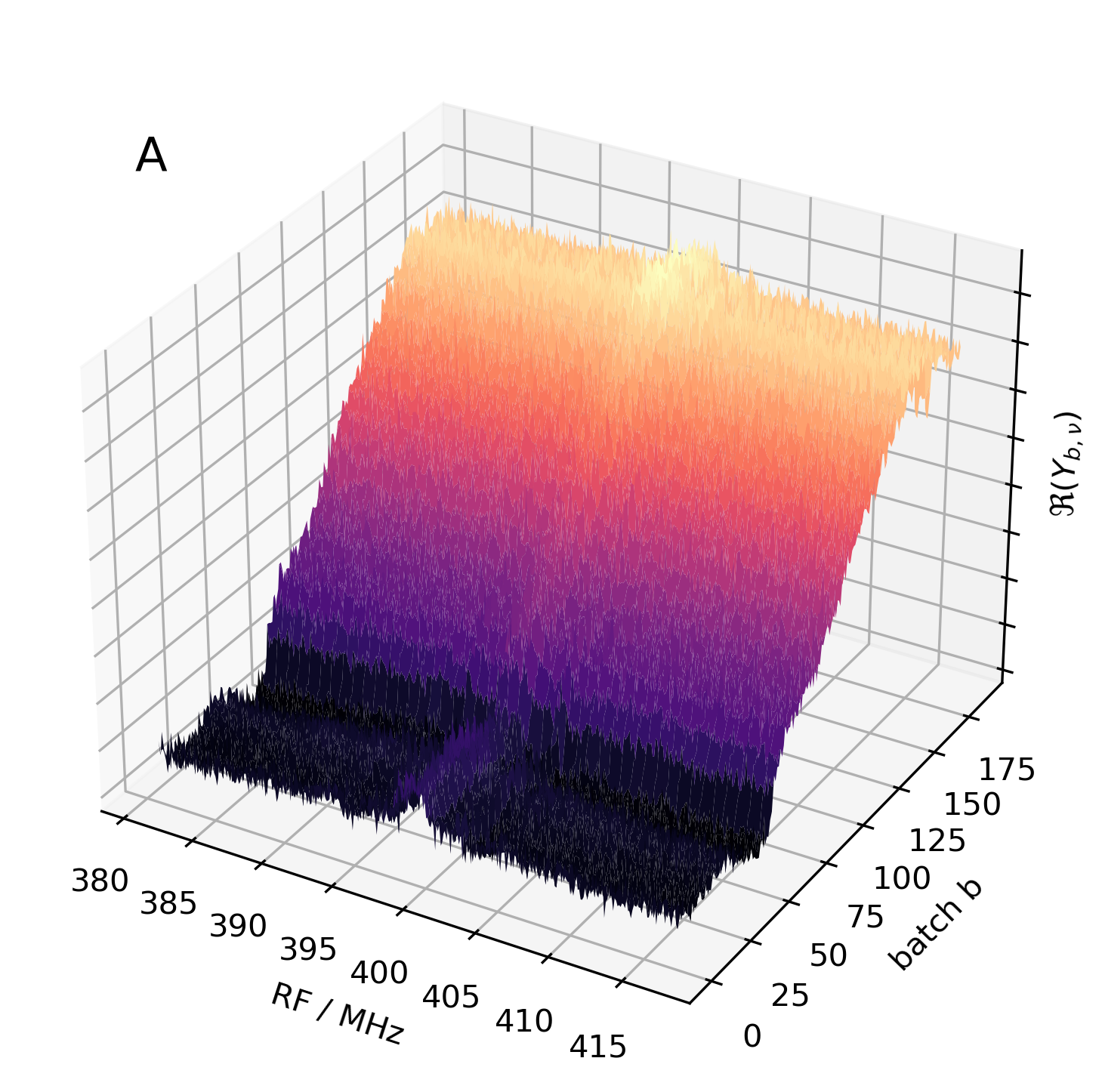}
    \includegraphics[width = 0.49\textwidth]{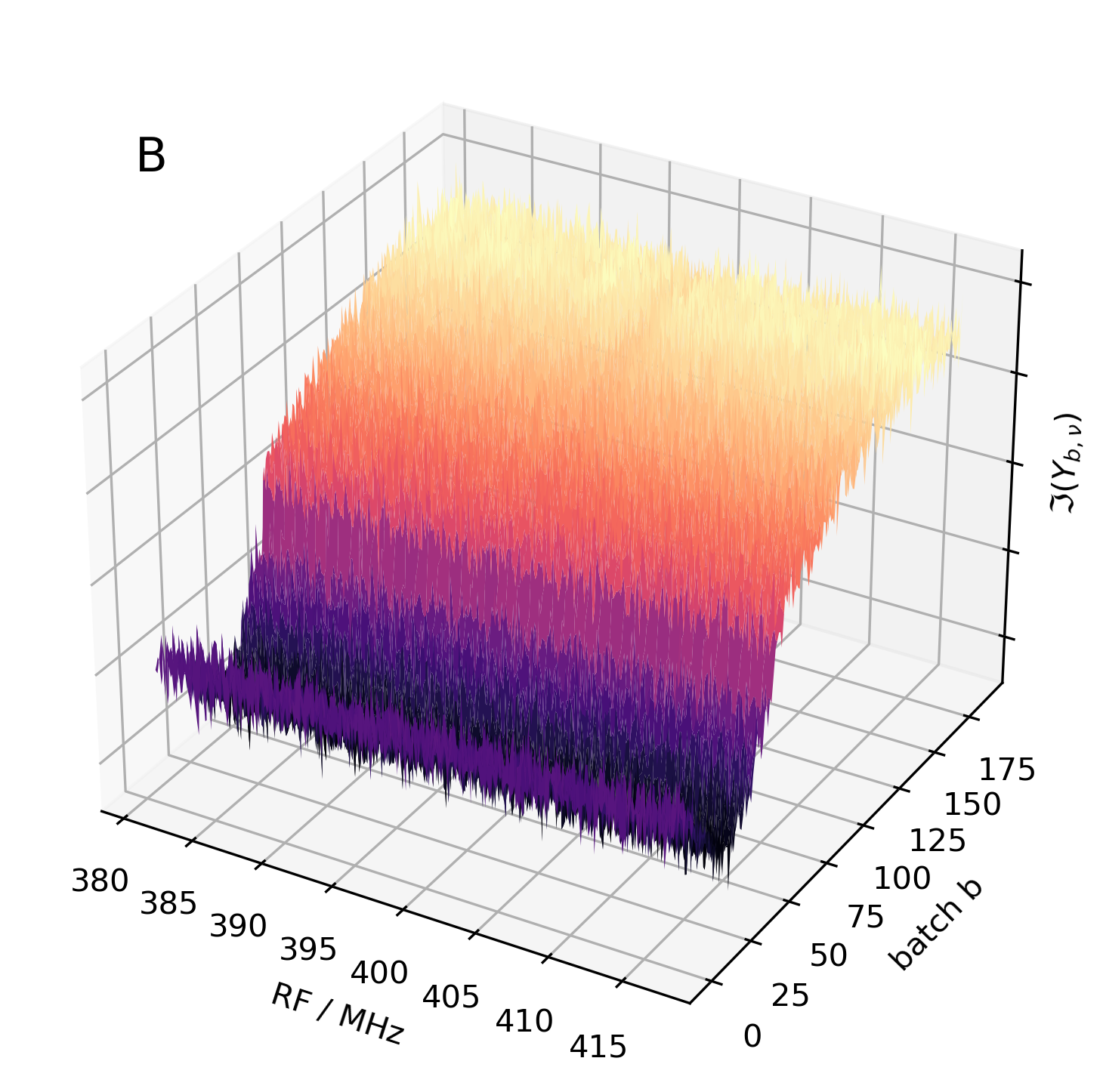}
    \caption{Panel A shows the real component of the raw data matrix $Y$ for orientation $g_y$ from a chemical sample of D2-$Y_{122}^{\bullet}$, while panel B displays the imaginary component of the same data matrix.}
    \label{fig: raw data}
  \end{figure}
  
  The algorithm used to fit the homoscedastic drift model is given in \autoref{alg:hom}.
  
  \begin{algorithm}
    \caption{Homoscedastic drift model MLE}
    \label{alg:hom}
    \begin{algorithmic}
      \State \Load $\bm y$
      \State $\widehat{\bm \psi} \gets \widehat{\bm \psi}_{hom}(\bm y)$
      \State $\bm{\Tilde{y}} \gets \bm y - \widehat{\bm \psi}$
      \State $\bm u,\eta,\overline{\bm v} \gets $ SVD$(\Tilde{\bm y}$,1st component)
      \State $\widehat{\bm{\phi}}^{(0)} \gets \bm u \eta$
      \State $\widehat{\bm{\kappa}}^{(0)} \gets \overline{\bm v}$
      \State $k \gets 0$
      \While{$k \leq maxiter=200$}
      \State $\widehat{\bm{\Sigma}}^{(k)} \gets  \widehat{\bm{\Sigma}}_{hom}(\widehat{\bm{\phi}}^{(k)},\widehat{\bm{\kappa}}^{(k)},\bm{\Tilde{y}})$
      \State $\widehat{\bm{\phi}}^{(k+1)} \gets  \widehat{\bm{\phi}}_{hom}(\widehat{\bm{\kappa}}^{(k)},(\widehat{\bm{\Sigma}}^{(k)})^{-1},\bm{\Tilde{y}})$
      \State $\widehat{\bm{\kappa}}^{(k+1)} \gets  \widehat{\bm{\kappa}}_{hom}(\widehat{\bm{\phi}}^{(k+1)},(\widehat{\bm{\Sigma}}^{(k)})^{-1},\bm{\Tilde{y}})$
      \State$ \widehat{\bm\kappa}^{(k+1)},\widehat{\bm\phi}^{(k+1)} \gets\frac{\widehat{\bm{\kappa}}^{(k+1)}}{\norm{\widehat{\bm{\kappa}}^{(k+1)}}},\norm{\widehat{\bm{\kappa}}^{(k+1)}}\widehat{\bm{\phi}}^{(k+1)}$
      \State $\ell^{(k)} \gets \ell(\bm{\Tilde{y}},\widehat{\bm{\phi}}^{(k+1)},\widehat{\bm{\kappa}}^{(k+1)},\widehat{\bm{\Sigma}}^{(k)})$
      \If{$k>0$}
      \If{$\ell^{(k)}-\ell^{(k-1)} < min\_delta\_loglik =10^{-4} $}
      \State \textbf{break}
      \EndIf
      \EndIf
      \State $k\gets k+1$
      \EndWhile
      
      \State \Return $\widehat{\bm\psi},\widehat{\bm\phi}^{(k)},\widehat{\bm\kappa}^{(k)},\widehat{\bm\Sigma}^{(k-1)}$
    \end{algorithmic}
  \end{algorithm}

  In Figure \ref{fig: average vs drift} we compare the SNR of the averaging model with SNR of the homoscedastic drift model. 
  For this purpose, the spectrum is extracted from the data matrices of the different orientations of the measurements of the chemical sample D2-$Y_{122}^{\bullet}$ using both the averaging model (plotted in green) and the homoscedastic drift model (plotted in black). In both models, the maximum method is used for phase correction. 
  The regions of RF frequencies where the true ENDOR spectrum is judged to be constant, referred to as \emph{flat frequency regions} defined in \cite{Pokern2021}, are plotted in the right panel and the standard deviations of the spectrum in the flat frequency regions are listed in Table \ref{tab: drift vs average}. 
  In four out of five orientations, the homoscedastic drift model provides an improved SNR. Only at orientation $g_x$ is the SNR of the averaging model slightly better than that of the homoscedastic drift model, which can be explained by the fact that at orientation $g_x$ the least phase drift of $\hat{\phi}$ is observed (see \autoref{fig: angle phi}).
  \begin{figure}[ht!]
    \centering
    \includegraphics[width = 1\textwidth]{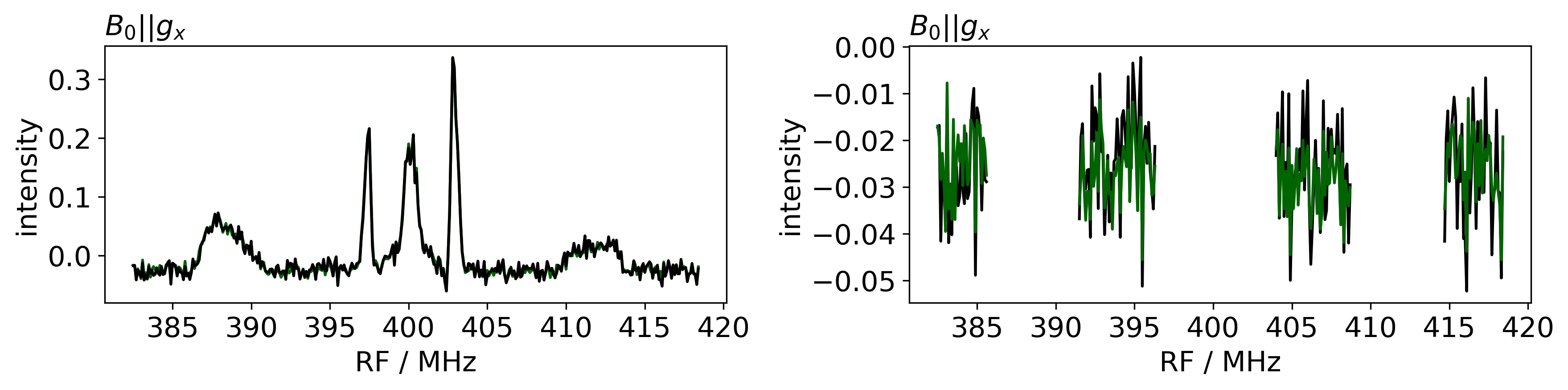}
    \includegraphics[width = 1\textwidth]{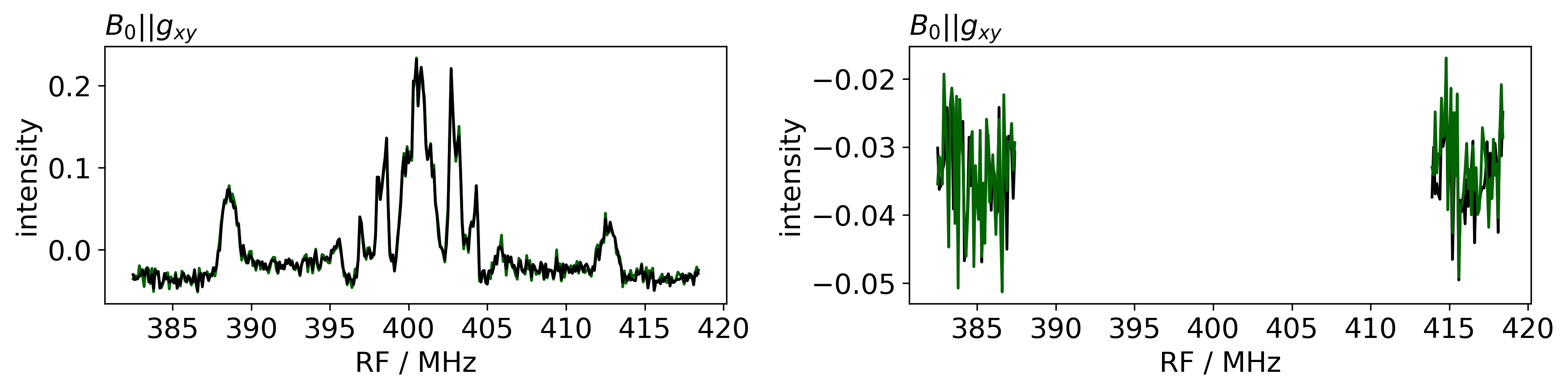}
    \includegraphics[width = 1\textwidth]{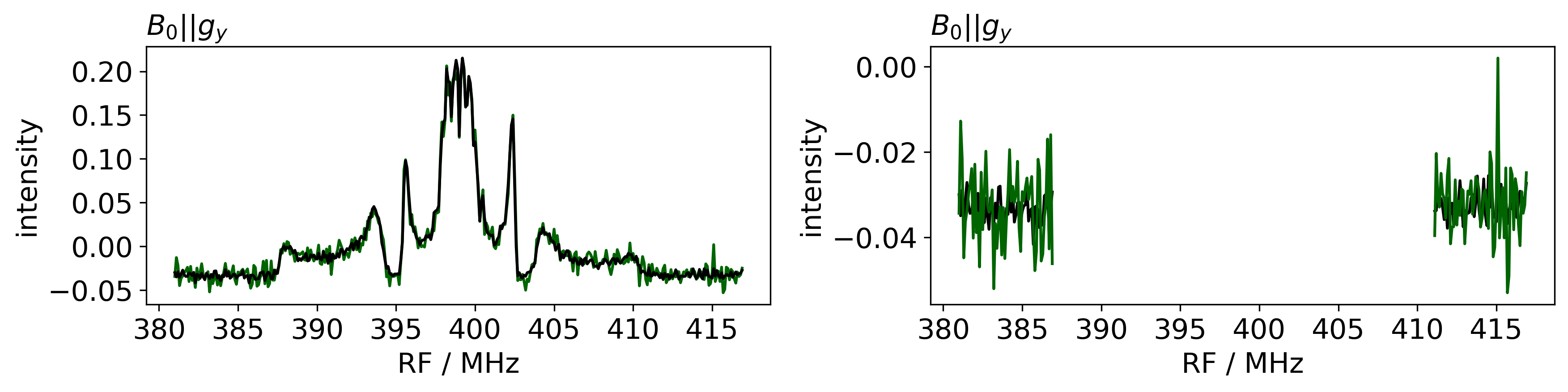}
    \includegraphics[width = 1\textwidth]{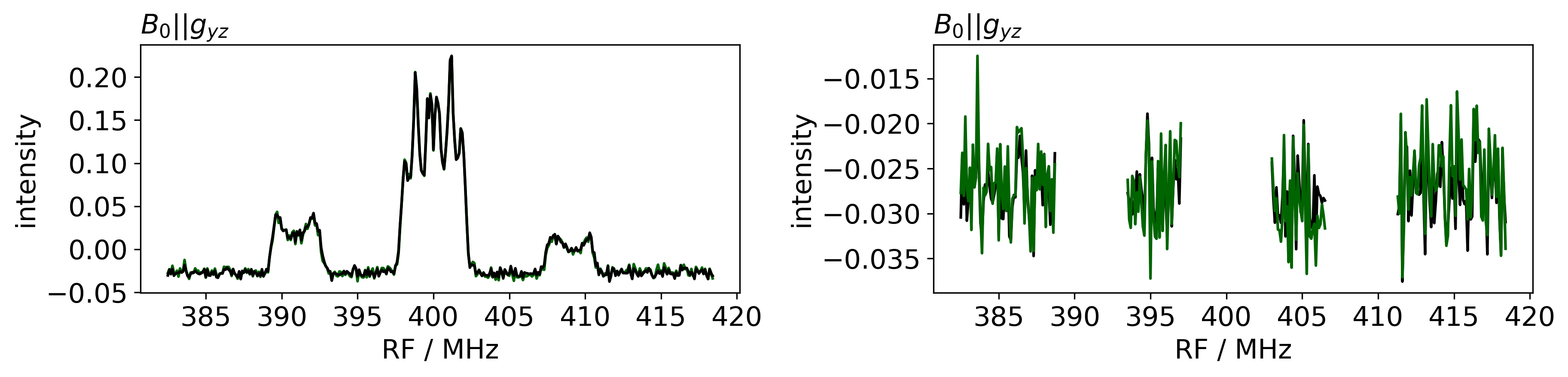}
    \includegraphics[width = 1\textwidth]{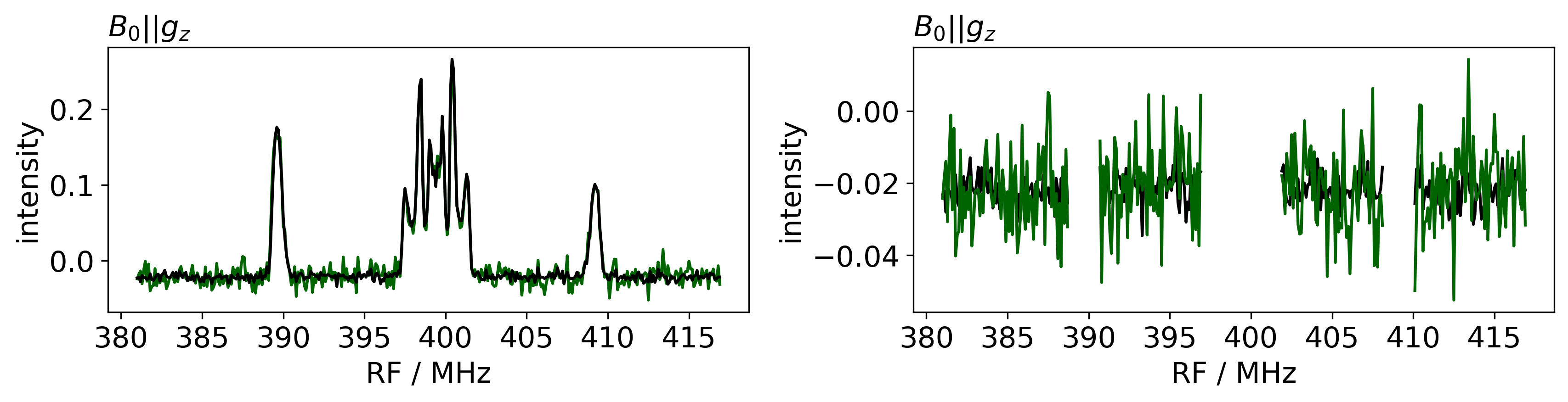}
    \caption{The estimated spectra $\hat{I}$ for all different orientations, using the homoscedastic drift model (black) and the averaging model (green) for all frequencies (left) and only for the flat frequency regions (right). The standard deviations of the spectra across the flat frequency regions from the right panel is given in Table \ref{tab: drift vs average}.
    }
    \label{fig: average vs drift}
  \end{figure}

  \begin{table}[ht!]
    \centering
    \begin{tabular}{c|c c}
      \hline
      Orientation & Averaging model & Homoscedastic drift model\\
      \hline
      $g_x$ & 0.0074 & 0.0107  \\
      $g_{xy}$ & 0.0072 & 0.0061 \\
      $g_y$ & 0.0085 & 0.0033 \\
      $g_{yz}$ & 0.0045 & 0.0035  \\
      $g_z$ & 0.0111 &0.0042  \\
    \end{tabular}
    \caption{The standard deviation of the spectrum across the flat frequency regions shown in Panel B of Figure \ref{fig: average vs drift}, computed for both the averaging model and the homoscedastic drift model.}
    \label{tab: drift vs average}
  \end{table}
  
  \begin{table}[ht!]
    \centering
    \begin{tabular}{c|c c}
      \hline
      Orientation & Real & Imag \\
      \hline
      $g_x$ & 0.098 &  0.220 \\
      $g_{xy}$ & 0.023 & 0.237 \\
      $g_y$ & 0.736 &  0.938 \\
      $g_{yz}$ & 0.373 & 0.271\\
      $g_z$ & 0.022 & 0.374\\
    \end{tabular}
    \caption{Results of Kolmogorov–Smirnov tests for Gaussianity applied to the real and imaginary parts of the residuals $\hat{\epsilon}_{b,\nu}=Y_{b,\nu}-\hat{\psi}_b - \hat{\phi}_b\hat{\kappa}_\nu$, pooled over $b$ and $\nu$, obtained from the homoscedastic drift model applied to all measurements.}
    \label{tab: pvalues drift model}
  \end{table}
  
  \begin{figure}[ht!]
    \centering
    \includegraphics[width = 1\textwidth]{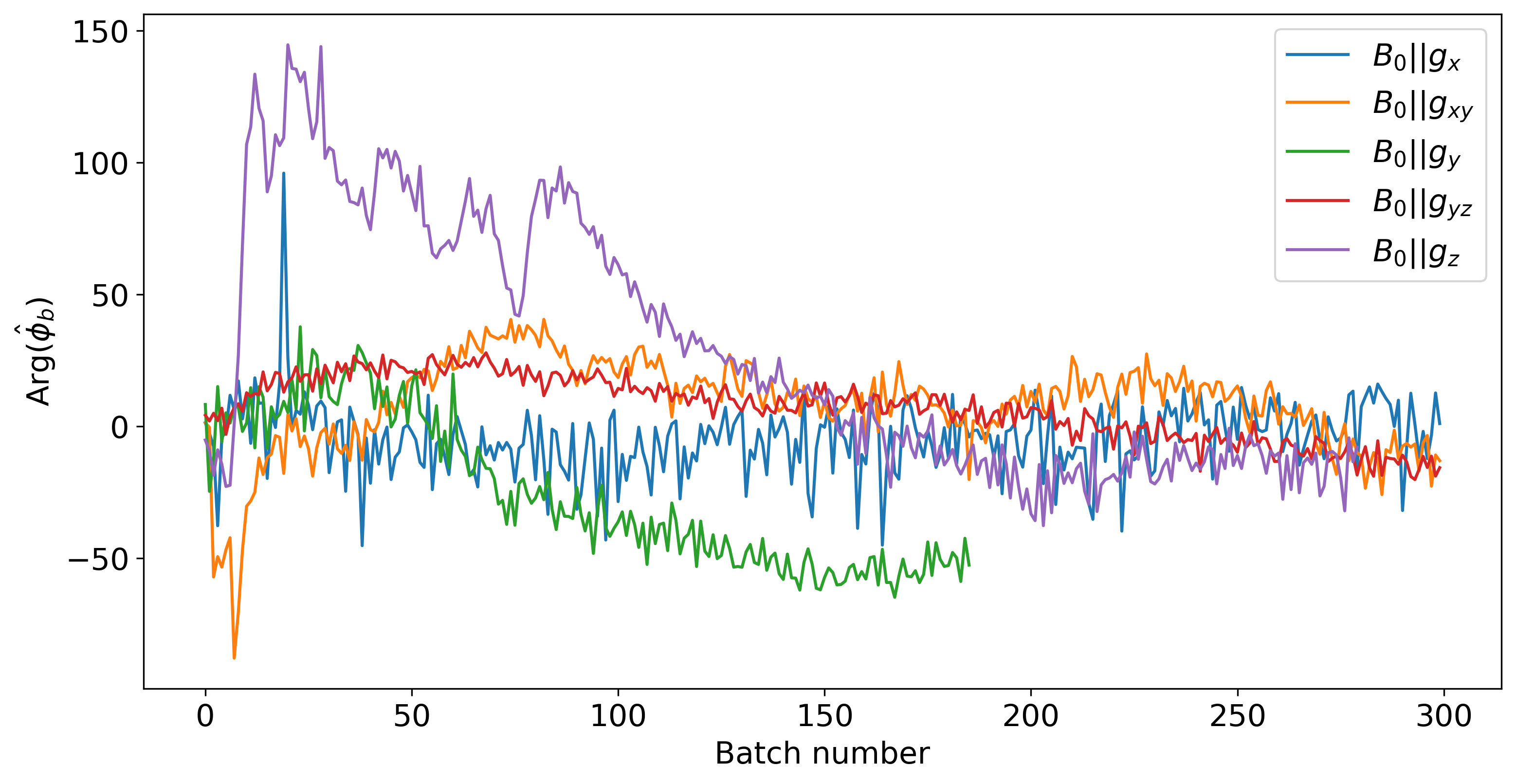}
    \caption{The angle of $\hat{\phi}$ for all five orientations of the chemical sample D2-$Y_{122}^{\bullet}$ at \SI{263}{\GHz}.}
    \label{fig: angle phi}
  \end{figure}

  \FloatBarrier        
  \newpage
  \section{Lemma for the Helmert Matrix }\label{sec: Technical Theorems and Lemmas}
  We prove in the following a lemma related to Section \ref{sec: maximum likelihood estimators} from the main text and consequently have the same notation. In particular, the vectors $h_1,\ldots,h_N$, the $\tilde{\epsilon}_{b, \nu}^H$ and $\vect{\tilde{\epsilon}_{b, \nu}^H}$ are defined as in Section \ref{sec: maximum likelihood estimators} in the main text.  
  \begin{lemma}\label{lem: epsilon dist helmert}
    In the new basis $h_1,\ldots,h_N$, the $\tilde{\epsilon}_{b, \nu}^H$ are i.i.d distributed and $\vect{\tilde{\epsilon}_{b, \nu}^H}\sim \mathcal{N}(0, \Sigma)$ for all $b=1\dots, B$ and $\nu = 1,\dots, N$.
  \end{lemma}
  
  \begin{proof}
    It holds $\mathbb{E}\left[ \vect{ \tilde{\epsilon}_{b,\nu}}\right]=0$ for all $b=1,\dots,B$ and $\nu=1,\dots,N$. 
    For all $b=1,\dots,B$ and $\nu=1,\dots,N$ we get
    \begin{align*}
    &\mathbb{E}\left[ \vect{ \tilde{\epsilon}^H_{b,\nu}}\vect{\tilde{\epsilon}^H_{b,\nu}}^T\right]=\mathbb{E}\left[ \vect{h_\nu^T \tilde{\epsilon}_b}\vect{h_\nu^T \tilde{\epsilon}_b}^T\right]\\
    &=\frac{1}{\nu(\nu+1)}\mathbb{E}\left[\left(\sum_{k=1}^\nu \vect{ \tilde{\epsilon}_{b,k}}-\nu\vect{ \tilde{\epsilon}_{b,\nu+1}}\right)\left(\sum_{k=1}^\nu \vect{ \tilde{\epsilon}_{b,k}}-\nu\vect{ \tilde{\epsilon}_{b,\nu+1}}\right)^T\right]\\
    &=\frac{1}{\nu(\nu+1)}\mathbb{E}\Bigg[\left(\sum_{k=1}^\nu \vect{ \tilde{\epsilon}_{b,k}}\right)\left(\sum_{k=1}^\nu \vect{ \tilde{\epsilon}_{b,k}}\right)^T-\nu\vect{ \tilde{\epsilon}_{b,\nu+1}}\left(\sum_{k=1}^\nu \vect{ \tilde{\epsilon}_{b,k}}\right)^T \\
    & \qquad \qquad \qquad -\nu \left(\sum_{k=1}^\nu \vect{ \tilde{\epsilon}_{b,k}}\right)
    \vect{ \tilde{\epsilon}_{b,\nu+1}} ^T + \nu\vect{ \tilde{\epsilon}_{b,\nu+1}}\nu\vect{ \tilde{\epsilon}_{b,\nu+1}}^T\Bigg]\\
    &=\frac{1}{\nu(\nu+1)}\left(\nu\left(-\frac{\nu}{N}\Sigma + \Sigma  \right) + \frac{2\nu^2}{N}\Sigma + \nu^2 \left(1-\frac{1}{N}\right)\Sigma\right)=\frac{1}{\nu+1}\left(\nu +\nu^2 \right)\Sigma=\Sigma.
    \end{align*}
    For $\nu_1<\nu_2$ we get
    \begin{align*}
    &\mathbb{E}\left[ \vect{ \tilde{\epsilon}^H_{b,\nu_1}}\vect{\tilde{\epsilon}^H_{b,\nu_2}}^T\right]=\mathbb{E}\left[ \vect{h_{\nu_1}^T \tilde{\epsilon}_b}\vect{h_{\nu_2}^T \tilde{\epsilon}_b}^T\right]\\
    &=\frac{1}{\sqrt{\nu_1(\nu_1+1)\nu_2(\nu_2+1)}}\mathbb{E}\left[\left(\sum_{k=1}^{\nu_1} \vect{ \tilde{\epsilon}_{b,k}}-{\nu_1}\vect{ \tilde{\epsilon}_{b,{\nu_1}+1}}\right)\left(\sum_{k=1}^{\nu_2} \vect{ \tilde{\epsilon}_{b,k}}-{\nu_2}\vect{ \tilde{\epsilon}_{b,{\nu_2}+1}}\right)^T\right]\\
    &=\frac{1}{\sqrt{\nu_1(\nu_1+1)\nu_2(\nu_2+1)}}\mathbb{E}\left[\left(\sum_{k=1}^{\nu_1} \vect{ \tilde{\epsilon}_{b,k}}\right)\left(\sum_{k=1}^{\nu_2} \vect{ \tilde{\epsilon}_{b,k}}\right)^T-\nu_1\vect{ \tilde{\epsilon}_{b,{\nu_1}+1}}\left(\sum_{k=1}^{\nu_2} \vect{ \tilde{\epsilon}_{b,k}}\right)^T \right]\\
    &+ \frac{1}{\sqrt{\nu_1(\nu_1+1)\nu_2(\nu_2+1)}}\mathbb{E}\left[-\nu_2\vect{ \tilde{\epsilon}_{b,{\nu_2}+1}}\left(\sum_{k=1}^{\nu_1} \vect{ \tilde{\epsilon}_{b,k}}\right)^T+ 
    \nu_1\vect{ \tilde{\epsilon}_{b,\nu_1+1}}\nu_2\vect{ \tilde{\epsilon}_{b,\nu_2+1}}^T\right]\\
    &=\frac{1}{\sqrt{\nu_1(\nu_1+1)\nu_2(\nu_2+1)}}\left( \nu_1\left(1 -\frac{\nu_2}{N}\right)\Sigma - \nu_1\left(1-\frac{\nu_2}{N}\right)\Sigma+ \frac{\nu_1\nu_2}{N}\Sigma-\frac{\nu_1\nu_2}{N}\Sigma\right)=0.
    \end{align*}
    The independence of $\tilde{\epsilon}^H_{b_1, \nu_1}$ from $\tilde{\epsilon}^H_{b_2, \nu_2}$ for $b_1 \neq b_2$ and $\nu_1,\nu_2 = 1,\dots, N$ follows directly from the independence of $\tilde{\epsilon}_{b_1, \nu_1}$ from $\tilde{\epsilon}_{b_2, \nu_2}$.
  \end{proof}
  \newpage
  
  \section{Example Strong Consistency}\label{sec: example strong Consistency}
  We show that using the theory developed in Section \ref{sec: strong consistency} in the main text we can prove strong consistency for the simultaneous estimation of $\mu$ and $\sigma$ in the univariate normal distribution. 
  \begin{assumption}\label{ass: norm dist}
    The random variable $X$ has distribution $X \sim \mathcal{N}\left(\mu^{(0)}, \left(\sigma^{(0)}\right)^2\right)$ where $\mu^{(0)}$ and $\left(\sigma^{(0)}\right)^2$ are the true but unknown parameters of the normal distribution.
  \end{assumption}
  For observations $x_1,\dots,x_n$ we get the following log-likelihood function
  \begin{align*}
  \ell_{x} (\mu, \sigma) = -\frac{n}{2}\ln( \sigma^2)-\frac{1}{2\sigma^2}\sum_{i=1}^n(x_i-\mu)^2= \sum_{i=1}^n\left(-\ln(\sigma)-\frac{1}{2\sigma^2} (x_i-\mu)^2\right).
  \end{align*}
  Since our theory was developed for minimization we have to change the sign and get
  \begin{align*}
  \rho (x, (\mu, \sigma)) = \ln(\sigma)+\frac{1}{2\sigma^2} (x-\mu)^2,
  \end{align*}
  the data space $\mathfrak{Q}=\mathbb{R}$ and the parameter space $\mathfrak{P}\coloneqq \mathbb{R} \times \mathbb{R}_{>0}$ with the metric 
  \begin{align*}
  d((\mu, \sigma), (\tilde{\mu}, \tilde{\sigma})) = \left|\mu - \tilde{\mu} \right| 
  + \left|\ln(\sigma) - \ln(\tilde{\sigma}) \right|
  +\left|\frac{1}{\sigma^2}- \frac{1}{\tilde{\sigma}^2} \right|.
  \end{align*}
  \begin{rem}
    We cannot use \cite{evans2020strong}, because there $\mathfrak{Q}=\mathfrak{P}$ is required and we cannot use \cite{Huckemann2011}, because there $\rho \geq 0$ is required.
    \cite{schotz2022strong} requires
    \begin{align*}
    \mathbb{E}\left(\inf_{(\mu, \sigma)\in \mathbb{R}\times \mathbb{R}_{+}}\rho (x, (\mu, \sigma))\right)>- \infty.
    \end{align*}
    However, one sees directly, if one inserts $x=\mu$  
    \begin{align*}
    \rho (x, (x, \sigma)) = \ln(\sigma) \overset{\sigma \rightarrow 0}{\longrightarrow} - \infty.
    \end{align*}
  \end{rem}

  \begin{theorem} 
    Under Assumption \ref{ass: norm dist} ZC holds for the normal distribution. 
  \end{theorem}
  \begin{proof}
    Let $(\mu, \sigma), (\tilde{\mu}, \tilde{\sigma}) \in \mathfrak{P}$ with $d((\mu, \sigma), (\tilde{\mu}, \tilde{\sigma})) < 1$ 
    \begin{align}\label{eq: normaldist lipschitz}
    &\left|\rho (x, (\mu, \sigma)) - \rho (x, (\tilde{\mu}, \tilde{\sigma})) \right|\leq \left|\ln(\sigma) - \ln(\tilde{\sigma}) \right| + \frac{1}{2}\left|\frac{1}{\sigma^2}\left(x-\mu\right)^2 - \frac{1}{\tilde{\sigma}^2}\left(x-\tilde{\mu}\right)^2\right|\\
    & \leq \left|\ln(\sigma) - \ln(\tilde{\sigma}) \right| + \frac{1}{4} \left|\frac{1}{\sigma^2}-\frac{1}{\tilde{\sigma}^2} \right| \left|(x-\mu)^2 +(x-\tilde{\mu})^2  \right|+ \frac{1}{4}\left|\frac{1}{\sigma^2}+\frac{1}{\tilde{\sigma}^2} \right|\left|(x-\mu)^2 -(x-\tilde{\mu})^2  \right|.\nonumber
    \end{align}
    It follows directly  
    \begin{align*}
    \left|\ln(\sigma) - \ln(\tilde{\sigma}) \right|\leq d((\mu, \sigma), (\tilde{\mu}, \tilde{\sigma})), \qquad \left|\frac{1}{\sigma^2}-\frac{1}{\tilde{\sigma}^2} \right|\leq d((\mu, \sigma), (\tilde{\mu}, \tilde{\sigma})).
    \end{align*}
    We get 
    \begin{align*}
    \left|(x-\mu)^2 +(x-\tilde{\mu})^2  \right|\leq 2x^2 + \left|2x\right|\left|\mu + \tilde{\mu}\right| + \left|\mu^2 + \tilde{\mu}^2\right|
    \leq 2x^2 +2\left|x\right|(2\left|\mu \right| + 1)+ 2\mu^2 + 2\left|\mu\right| + 1.
    \end{align*}
    We get for the last part of (\ref{eq: normaldist lipschitz})
    \begin{align*}
    \left|\frac{1}{\sigma^2}+\frac{1}{\tilde{\sigma}^2} \right|\leq 2\left|\frac{1}{\sigma^2} \right| + \left|\frac{1}{\sigma^2}-\frac{1}{\tilde{\sigma}^2} \right| 
    \leq 2\left|\frac{1}{\sigma^2} \right| + 1
    \end{align*}
    and 
    \begin{align*}
    \left|(x-\mu)^2 -(x-\tilde{\mu})^2  \right|&\leq 2\left|x\right|\left|\mu-\tilde{\mu}\right|+2\left|\mu\right|d((\mu, \sigma), (\tilde{\mu}, \tilde{\sigma}))+d((\mu, \sigma), (\tilde{\mu}, \tilde{\sigma}))^2\\
    &\leq d((\mu, \sigma), (\tilde{\mu}, \tilde{\sigma}))\Big(2\left|x\right|  +  2\left|\mu\right| + 1\Big).
    \end{align*}
    Substituting the inequalities into (\ref{eq: normaldist lipschitz}) results in
    \begin{align*}
    &\left|\rho (x, (\mu, \sigma)) - \rho (x, (\tilde{\mu}, \tilde{\sigma})) \right|\leq d((\mu, \sigma), (\tilde{\mu}, \tilde{\sigma}))
    \dot{\rho}(x, (\mu, \sigma))
    \end{align*}
    where $\dot{\rho}(x, (\mu, \sigma))$ is defined as
    \begin{align*}
    \dot{\rho}(x, (\mu, \sigma))\coloneqq 1
    + \frac{1}{4}\left(2x^2 +2\left|x\right|(2\left|\mu \right| + 1)+ 2\mu^2 + 2\left|\mu\right| + 1\right)
    +\frac{1}{4}\left(\frac{2}{\sigma^2}  + 1\right)\left(
    2\left|x\right|  +  2\left|\mu\right| + 1\right).
    \end{align*}
    Therefore we get
    \begin{align*}
    \dot{\Ffun}(\mu, \sigma) = 1
    + \frac{1}{4}\left(2\mathbb{E}\left(x^2\right) +2\mathbb{E}\left(\left|x\right|\right)(2\left|\mu \right| + 1)+ 2\mu^2 + 2\left|\mu\right| + 1\right)
    +\frac{1}{4}\left(\frac{2}{\sigma^2}  + 1\right)\left(
    2\mathbb{E}(\left|x\right|)  +  2\left|\mu\right| + 1\right),
    \end{align*}
    which is smaller than infinity and continuous.
  \end{proof}
  
  \begin{theorem} 
    Under Assumption \ref{ass: norm dist} BPC holds for the normal distribution. 
  \end{theorem} 
  \begin{proof}
    It holds 
    \begin{align*}
    \left\{\left(\mu^{(0)}, \sigma^{(0)}\right)\right\} = E^{(\rho)}.
    \end{align*}
    If $(\mu_n, \sigma_n^2)_n^\infty\subset \mathfrak{P}$ is without accumulation points then a.s. $\liminf \rhofun(X, (\mu_n, \sigma_n)) \rightarrow \infty$.
    Thus BPC follows immediately.
  \end{proof}
  \newpage
  \section{Technical Theorems and Lemmas for the homoscedastic drift model}\label{sec: Strong Consistency: Technical Theorems and Lemmas}
  In this section, we prove technical theorems and lemmas needed for the strong consistency and central limit theorem in Section~\ref{sec: Drift model: Strong consistency and central limit theorem} in the main text. In particular, we use the definitions of $\rho,\mathfrak{Q}$ and $\mathfrak{P}$ from Section~\ref{sec: Drift model: Strong consistency and central limit theorem}. 
  \begin{lemma}\label{lem: rho formula}
    For $\rho:  \mathfrak{Q} \times \mathfrak{P} \mapsto \mathbb{R}$ that is defined, as in (\ref{eq: definition rho}) holds 
    \begin{align*}
    \rhofun(Y, [\kappa]) = \spr{Y}{Y}_P  - \spr{\hat{\phi}(\kappa,P, Y) \kappa}{Y}_P
    \end{align*}
    for all $\kappa\in [\kappa]$.
  \end{lemma}
  
  \begin{proof}
    We start with the Definition of $\rho$
    \begin{align*}
    \rhofun(Y, [\kappa]) &= \sum_{\nu=1}^N\normp{\vect{Y_\nu}-M(\kappa_\nu)\vect{\hat{\phi}(\kappa,P, Y)}}^2\\
    &=\spr{Y}{Y}_P - 2\spr{\hat{\phi}(\kappa,P, Y) \kappa}{Y}_P+\spr{\hat{\phi}(\kappa,P, Y) \kappa}{\hat{\phi}(\kappa,P, Y) \kappa}_P.
    \end{align*}
    We get for the last part of the equation
    \begin{align*}
    &\spr{\hat{\phi}(\kappa,P, Y) \kappa}{\hat{\phi}(\kappa,P, Y) \kappa}_P=\sum_{\nu=1}^N\vect{\hat{\phi}(\kappa,P, Y)}^TM(\kappa_\nu)^TPM(\kappa_\nu)\vect{\hat{\phi}(\kappa,P, Y)}\\
    &=\vect{\hat{\phi}(\kappa,P, Y)}^T\left(\kappa \dia_P \kappa \right)\vect{\hat{\phi}(\kappa,P, Y)}
    =\vect{\hat{\phi}(\kappa,P, Y)}^T\left(\kappa \dia_P \kappa \right)\left(\kappa \dia_P \kappa \right)^{-1} \left(\kappa\bul_P Y\right)\\
    &= \vect{\hat{\phi}(\kappa,P, Y)}^T \left(\kappa\bul_P Y\right).
    \end{align*}
    It follows directly $\vect{\hat{\phi}(\kappa,P, Y)}^T \left(\kappa\bul_P Y\right)=\spr{\hat{\phi}(\kappa,P, Y) \kappa}{Y}_P$ and therefore
    \begin{align*}
    \rhofun(Y, [\kappa]) &=  \spr{Y}{Y}_P  - \spr{\hat{\phi}(\kappa,P, Y) \kappa}{Y}_P .
    \end{align*}
  \end{proof}
  \begin{definition}\label{def: ptilde}
    For $P = R \begin{pmatrix}
    \lambda_1& 0\\\
    0 & \lambda_2
    \end{pmatrix}
    R^T\in \mathrm{SPD}(2)$ we define
    $
    \tilde{P} = R \begin{pmatrix}
    \lambda_2& 0\\
    0 & \lambda_1
    \end{pmatrix}
    R^T
    $
    where $\lambda_1 \ge \lambda_2 >0$ and $R$ is a rotation matrix.
  \end{definition}
  
  For the rest of this section, we simplify the notation of sums. Every sum symbol without bounds, $\sum$, is to be understood as a sum $\sum_{\nu =1}^N$.
  
  \begin{lemma}\label{lem: inverse lemma}
    For $\kappa \in \Scomp$ and $P, \tilde{P}\in \mathrm{SPD}(2)$ as defined in Definition \ref{def: ptilde} holds
    \begin{align*}
    \left(\kappa \dia_P \kappa \right)^{-1}=\frac{1}{\det\left( \kappa \dia_P \kappa \right) } \left(\kappa \dia_{\tilde{P}} \kappa \right).
    \end{align*}
  \end{lemma}
  \begin{proof}
    First, we define $M(\tilde{\kappa}_\nu) \coloneqq R M(\kappa_\nu)$ for all $\nu =1,\dots, N$.
    Therefore,
    \begin{align*}
    \kappa \dia_P \kappa = \tilde{\kappa} \dia_{\mathrm{diag}(\lambda_1, \lambda_2)}\tilde{\kappa} =\begin{pmatrix} 
    \lambda_1\sum\Re(\tilde{\kappa}_{\nu})^2 + \lambda_2\sum\Im(\tilde{\kappa}_{\nu})^2& (\lambda_2 -\lambda_1)\sum\Re(\tilde{\kappa}_{\nu})\Im(\tilde{\kappa}_{\nu}) \\
    (\lambda_2 -\lambda_1)\sum\Re(\tilde{\kappa}_{\nu})\Im(\tilde{\kappa}_{\nu}) &  \lambda_2\sum\Re(\tilde{\kappa}_{\nu})^2 + \lambda_1\sum\Im(\tilde{\kappa}_{\nu})^2  \end{pmatrix}.
    \end{align*}
    Using the standard rule for calculating a $2\times 2$ inverse matrix, we get the desired result
    \begin{align*}
    \left(\kappa \dia_P \kappa \right)^{-1}&= \frac{1}{\det\left( \kappa \dia_P \kappa \right) }
    \begin{pmatrix} 
    \lambda_2\sum\Re(\tilde{\kappa}_{\nu})^2 + \lambda_1\sum\Im(\tilde{\kappa}_{\nu})^2
    & (\lambda_1 -\lambda_2)\sum\Re(\tilde{\kappa}_{\nu})\Im(\tilde{\kappa}_{\nu}) \\
    (\lambda_1 -\lambda_2)\sum\Re(\tilde{\kappa}_{\nu})\Im(\tilde{\kappa}_{\nu}) &  \lambda_1\sum\Re(\tilde{\kappa}_{\nu})^2 + \lambda_2\sum\Im(\tilde{\kappa}_{\nu})^2 \end{pmatrix}\\ 
    &= \frac{1}{\det\left( \kappa \dia_P \kappa \right)}\left(\kappa \dia_{\tilde{P}} \kappa  \right).
    \end{align*}
  \end{proof}
  \begin{lemma}\label{lem: det lemma}
    For all $\kappa \in \Scomp$ and $P\in \mathrm{SPD}(2)$ as defined in Definition \ref{def: ptilde} holds
    \begin{align*}
    \det\left(\kappa \dia_P \kappa \right) \geq \lambda_1 \lambda_2.
    \end{align*}
  \end{lemma}
  \begin{proof}
    We define $\tilde{\kappa}_\nu$ in the same way as in the proof of Lemma \ref{lem: inverse lemma}.
    We get using the Cauchy-Schwarz inequality
    \begin{align*}
    &\det\left(\kappa \dia_P \kappa \right)=\det\begin{pmatrix} 
    \lambda_1\sum\Re(\tilde{\kappa}_{\nu})^2 + \lambda_2\sum\Im(\tilde{\kappa}_{\nu})^2& (\lambda_2 -\lambda_1)\sum\Re(\tilde{\kappa}_{\nu})\Im(\tilde{\kappa}_{\nu}) \\
    (\lambda_2 -\lambda_1)\sum\Re(\tilde{\kappa}_{\nu})\Im(\tilde{\kappa}_{\nu}) &  \lambda_2\sum\Re(\tilde{\kappa}_{\nu})^2 + \lambda_1\sum\Im(\tilde{\kappa}_{\nu})^2  \end{pmatrix} \\
    &=\left(\lambda_1\sum\Re(\tilde{\kappa}_{\nu})^2 + \lambda_2\sum\Im(\tilde{\kappa}_{\nu})^2\right)\left( \lambda_2\sum\Re(\tilde{\kappa}_{\nu})^2 + \lambda_1\sum\Im(\tilde{\kappa}_{\nu})^2\right)\\
    &-\left( (\lambda_2 -\lambda_1)\sum\Re(\tilde{\kappa}_{\nu})\Im(\tilde{\kappa}_{\nu})\right)^2\\
    &\geq\left(\lambda_1\sum\Re(\tilde{\kappa}_{\nu})^2 + \lambda_2\sum\Im(\tilde{\kappa}_{\nu})^2\right)\left( \lambda_2\sum\Re(\tilde{\kappa}_{\nu})^2 + \lambda_1\sum\Im(\tilde{\kappa}_{\nu})^2\right)\\
    &- (\lambda_2 -\lambda_1)^2\left(\sum\Re(\tilde{\kappa}_{\nu})^2\right)\left(\sum\Im(\tilde{\kappa}_{\nu})^2\right)\\
    &= \lambda_1\lambda_2\left(\sum\Re(\tilde{\kappa}_{\nu})^2 \right)^2 + \lambda_1\lambda_2\left(\sum\Im(\tilde{\kappa}_{\nu})^2 \right)^2
    + 2\lambda_1\lambda_2\left(\sum\Re(\tilde{\kappa}_{\nu})^2\right)\left(\sum\Im(\tilde{\kappa}_{\nu})^2\right)\\
    &= \lambda_1\lambda_2\left(\sum\Re(\tilde{\kappa}_{\nu})^2 + \sum\Im(\tilde{\kappa}_{\nu})^2 \right)^2= \lambda_1\lambda_2.
    \end{align*}
  \end{proof}
  
  \subsection{Calculating the modulus of continuity along with its prefactor}
  This section contains all the calculations needed for the Theorem \ref{the: inequality rho} from the main text. 
  \ref{the: inequality rho}
  
  \begin{lemma}\label{lem: dp phi plus phi and kappa minus kappa}
    For $\kappa, \kappa' \in \Scomp, Y\in \mathbb{C}^N$ and $P\in \mathrm{SPD}(2)$ as defined in Definition \ref{def: ptilde}  holds
    \begin{align*}
    \frac{1}{2} d_P\Bigg(\left(\hat{\phi}(\kappa,P, Y) +\hat{\phi}(\kappa',P, Y)\right)(\kappa - \kappa'), 0\Bigg)\leq \lambda_1 \sqrt{2N}\frac{\lambda_1^2 + \lambda_2^2}{\lambda_1 \lambda_2}\left|\left|Y \right|\right|||\kappa - \kappa'||.
    \end{align*}
  \end{lemma}
  
  \begin{proof}
    Using first Lemma  \ref{lem: dp xa} and then Lemma \ref{lem: abs of phi} yields
    \begin{align*}
    & \frac{1}{2} d_P\Bigg(\left(\hat{\phi}(\kappa,P, Y) +\hat{\phi}(\kappa',P, Y)\right)(\kappa - \kappa'), 0\Bigg)\\
    &\leq \frac{\lambda_1}{2} \left| \hat{\phi}(\kappa,P, Y) +\hat{\phi}(\kappa',P, Y)\right| ||\kappa - \kappa'|| \leq
    \lambda_1 \sqrt{2N}\frac{\lambda_1^2 + \lambda_2^2}{\lambda_1 \lambda_2}\left|\left|Y \right|\right|||\kappa - \kappa'||.
    \end{align*}
  \end{proof}

  \begin{lemma}\label{lem: dp xa}
    For $x \in \mathbb{C}, a\in \mathbb{C}^N$ and $P\in \mathrm{SPD}(2)$ as defined in Definition \ref{def: ptilde} with $\lambda_1\geq \lambda_2$ holds
    \begin{align*}
    d_P(xa, 0) \leq\sqrt{\lambda_1} |x|\cdot ||a||.
    \end{align*}
  \end{lemma}
  \begin{proof}
    Since $P\in \mathrm{SPD}(2)$ we can write $P=R^T \mathrm{diag}(\lambda_1, \lambda_2) R$ where $R=\begin{pmatrix}
    \cos(\alpha) & - \sin(\alpha)\\
    \sin(\alpha) & \cos(\alpha)
    \end{pmatrix}$ is a rotation matrix. Therefore
    \begin{align*}
    d_P(xa, 0) &= \sqrt{\spr{xa}{xa}_P}=\sqrt{\spr{e^{i\alpha}xa}{{e^{i\alpha}xa}}_{\mathrm{diag}(\lambda_1, \lambda_2)}}\leq \sqrt{\lambda_1\spr{e^{i\alpha}xa}{{e^{i\alpha}xa}}_{\mathrm{Id}_2}}=\sqrt{\lambda_1} |x|\cdot ||a||.
    \end{align*}
  \end{proof}
  
  \begin{lemma}\label{lem: norm kappa dia P}
    For $\kappa\in\Scomp$ and $P\in \mathrm{SPD}(2)$ as defined in Definition \ref{def: ptilde} holds
    \begin{align*}
    ||\kappa \dia_P \kappa|| \leq  \sqrt{\lambda_1^2 + \lambda_2^2}.
    \end{align*}
  \end{lemma}
  
  \begin{proof}
    We define $\tilde{\kappa}_\nu$ in the same way as in the proof of Lemma \ref{lem: inverse lemma}. Using Cauchy–Schwarz we get 
    \begin{align*}
    ||\kappa \dia_P \kappa||^2 &= \left|\left| \begin{pmatrix} 
    \lambda_1\sum\Re(\tilde{\kappa}_{\nu})^2 + \lambda_2\sum\Im(\tilde{\kappa}_{\nu})^2& (\lambda_2 -\lambda_1)\sum\Re(\tilde{\kappa}_{\nu})\Im(\tilde{\kappa}_{\nu}) \\
    (\lambda_2 -\lambda_1)\sum\Re(\tilde{\kappa}_{\nu})\Im(\tilde{\kappa}_{\nu}) &  \lambda_2\sum\Re(\tilde{\kappa}_{\nu})^2 + \lambda_1\sum\Im(\tilde{\kappa}_{\nu})^2  \end{pmatrix}\right|\right|^2\\
    &=\left(\lambda_1\sum\Re(\tilde{\kappa}_{\nu})^2 + \lambda_2\sum\Im(\tilde{\kappa}_{\nu})^2 \right)^2 + 2\left((\lambda_2 -\lambda_1)\sum\Re(\tilde{\kappa}_{\nu})\Im(\tilde{\kappa}_{\nu})\right)^2 \\
    &+ \left(\lambda_2\sum\Re(\tilde{\kappa}_{\nu})^2 + \lambda_1\sum\Im(\tilde{\kappa}_{\nu})^2 \right)^2\\
    &\leq \left(\lambda_1\sum\Re(\tilde{\kappa}_{\nu})^2 + \lambda_2\sum\Im(\tilde{\kappa}_{\nu})^2 \right)^2 
    + 2(\lambda_2 -\lambda_1)^2\left(\sum\Re(\tilde{\kappa}_{\nu})^2\right)\left(\sum\Im(\tilde{\kappa}_{\nu})^2\right)\\
    &+ \left(\lambda_2\sum\Re(\tilde{\kappa}_{\nu})^2 + \lambda_1\sum\Im(\tilde{\kappa}_{\nu})^2 \right)^2\\
    &= \lambda_1^2\left(\sum\Re(\tilde{\kappa}_{\nu})^2 + \sum\Im(\tilde{\kappa}_{\nu})^2\right)^2 + \lambda_2^2\left(\sum\Re(\tilde{\kappa}_{\nu})^2 + \sum\Im(\tilde{\kappa}_{\nu})^2\right)^2 = \lambda_1^2 + \lambda_2^2.
    \end{align*}
  \end{proof}

  \begin{lemma}\label{lem: norm of inverse sum}
    For $\kappa\in \Scomp$ and $P\in \mathrm{SPD}(2)$ as defined in Definition \ref{def: ptilde} holds
    \begin{align*}
    \left|\left|\left(\kappa \dia_P \kappa\right)^{-1}\right|\right| \leq \frac{\sqrt{\lambda_1^2 + \lambda_2^2}}{\lambda_1\lambda_2}.\\
    \end{align*}
  \end{lemma}
  \begin{proof}
    We get by using Lemma \ref{lem: inverse lemma},  Lemma \ref{lem: det lemma} and Lemma \ref{lem: norm kappa dia P}
    \begin{align*}
    \left|\left|\left(\kappa \dia_P \kappa\right)^{-1}\right|\right| &= \left|\left|  \frac{1}{\det\left(\kappa \dia_P \kappa\right)}\left(\kappa \dia_{\tilde{P}} \kappa \right) \right|\right| \leq \frac{\sqrt{\lambda_1^2 + \lambda_2^2}}{\lambda_1\lambda_2}. \\ 
    \end{align*}
  \end{proof}

  \begin{lemma}\label{lem: norm bul Y}
    For $\kappa\in \Scomp, Y\in \mathbb{C}^N$ and $P\in \mathrm{SPD}(2)$ as defined in Definition \ref{def: ptilde} holds
    \begin{align*}
    ||\kappa \bul_P Y|| \leq  \sqrt{2N}\sqrt{\lambda_1^2+\lambda_2^2}\left|\left|Y \right|\right|.
    \end{align*}
  \end{lemma}
  
  \begin{proof}
    We define $\tilde{\kappa}_\nu$ in the same way as in the proof of Lemma \ref{lem: inverse lemma}. We calculate 
    \begin{align*}
    &||\kappa \bul_P Y|| = \left|\left|\begin{pmatrix} M(\kappa_1)^T & \dots &  M(\kappa_{N})^T \end{pmatrix} (\mathrm{Id}_N \otimes P) \begin{pmatrix}
    \vect{Y_1}\\ \vdots \\\vect{Y_N}
    \end{pmatrix} \right|\right|\\
    &\leq \left|\left|\begin{pmatrix} M(\kappa_1)^T & \dots &  M(\kappa_{N})^T \end{pmatrix}\right|\right|\left|\left|  \left(\mathrm{Id}_N \otimes P\right)\right|\right| \left|\left|Y \right|\right|
    =\sqrt{2N}\sqrt{\lambda_1^2+\lambda_2^2}\left|\left|Y \right|\right|.
    \end{align*}
  \end{proof}
  
  \begin{lemma}\label{lem: abs of phi}
    For $\kappa\in \Scomp, Y\in \mathbb{C}^N$ and $P\in \mathrm{SPD}(2)$ as defined in Definition \ref{def: ptilde} holds
    \begin{align*}
    \left|\hat{\phi}(\kappa,P, Y)\right| \leq  \sqrt{2N}\frac{\lambda_1^2 + \lambda_2^2}{\lambda_1 \lambda_2}\left|\left|Y \right|\right|.
    \end{align*}
  \end{lemma}
  \begin{proof}
    Using Lemma \ref{lem: norm of inverse sum} and Lemma \ref{lem: norm bul Y} we get
    \begin{align*}
    \left|\hat{\phi}(\kappa,P, Y)\right| &= \left|\left|\left(\kappa \dia_P \kappa \right)^{-1} \left(\kappa \bul_P Y\right)\right|\right|\leq \left|\left|\left(\kappa \dia_P \kappa \right)^{-1}\right|\right| \left|\left|\left(\kappa \bul_P Y\right)\right|\right| \leq \sqrt{2N}\frac{\lambda_1^2 + \lambda_2^2}{\lambda_1 \lambda_2}\left|\left|Y \right|\right|.
    \end{align*}
  \end{proof}

  \begin{lemma}\label{lem: dp phi minus phi and kappa plus kappa}
    For $\kappa, \kappa' \in \Scomp$ and $Y\in \mathbb{C}^N$ and $P\in \mathrm{SPD}(2)$ as defined in Definition \ref{def: ptilde} holds
    \begin{align*}
    &\frac{1}{2} d_P\Bigg(\left(\hat{\phi}(\kappa,P, Y) -\hat{\phi}(\kappa',P, Y)\right)(\kappa + \kappa'), 0\Bigg)\\
    &\leq \left(\frac{\lambda_1^2+\lambda_2^2}{\lambda_1 \lambda_2}\right)\left(8\sqrt{2}N+\frac{32\sqrt{2}N\left(\lambda_1^2+\lambda_2^2\right)}{\lambda_1\lambda_2}+2\sqrt{2N}\right)  \left|\left|Y \right|\right| \left|\left|\kappa -\kappa' \right|\right|.
    \end{align*}
  \end{lemma}
  
  \begin{proof}
    First, we use Lemma \ref{lem: dp xa}
    \begin{align*}
    &\frac{1}{2} d_P\Bigg(\left(\hat{\phi}(\kappa,P, Y) -\hat{\phi}(\kappa',P, Y)\right)(\kappa + \kappa'), 0\Bigg)\leq \frac{\sqrt{\lambda_1}}{2} \left| \hat{\phi}(\kappa,P, Y) -\hat{\phi}(\kappa',P, Y)\right|\cdot ||\kappa + \kappa'||\\
    &\leq \sqrt{\lambda_1} \left| \hat{\phi}(\kappa,P, Y) -\hat{\phi}(\kappa',P, Y)\right|
    =\sqrt{\lambda_1} \left| \left|\left(\kappa \dia_P \kappa \right)^{-1}\left(\kappa \bul_P Y \right)-\left(\kappa' \dia_P \kappa' \right)^{-1}\left(\kappa' \bul_P Y \right)\right|\right|\\
    &\leq \frac{\sqrt{\lambda_1}}{2}\left|\left|\left(\kappa \dia_P \kappa \right)^{-1}-\left(\kappa' \dia_P \kappa' \right)^{-1} \right|\right| \left|\left| \left(\kappa \bul_P Y \right)+\left(\kappa' \bul_P Y \right) \right|\right|\\
    &+ \frac{\sqrt{\lambda_1}}{2}\left|\left|\left(\kappa \dia_P \kappa \right)^{-1}+\left(\kappa' \dia_P \kappa' \right)^{-1}\right|\right| \left|\left| \left(\kappa \bul_P Y \right)-\left(\kappa' \bul_P Y \right) \right|\right|.
    \end{align*}
    From Lemma \ref{lem: norm of inverse sum}, Lemma \ref{lem: norm bul Y}, Lemma \ref{lem: kappa bul Y- kappa'bul Y} and Lemma \ref{lem: inv(kappa dia P) - inv(kappa' dia P) } it follows that 
    \begin{align*}
    &\frac{1}{2} d_P\Bigg(\left(\hat{\phi}(\kappa,P, Y) -\hat{\phi}(\kappa',P, Y)\right)(\kappa + \kappa'), 0\Bigg)\\
    &\leq
    \left(\left(\frac{4\sqrt{N}\sqrt{\lambda_1^2+\lambda_2^2}}{\lambda_1\lambda_2}+\frac{16\sqrt{N}(\lambda_1^2+\lambda_2^2)^{3/2}}{(\lambda_1\lambda_2)^2}\right)\left|\left|\kappa-\kappa'\right|\right| \right) 
    \left(2\sqrt{2N}\sqrt{\lambda_1^2+\lambda_2^2}\left|\left|Y \right|\right| \right) \\
    &+ 
    \left(2 \frac{\sqrt{\lambda_1^2 + \lambda_2^2}}{\lambda_1\lambda_2}\right) 
    \left(\sqrt{2N}\sqrt{\lambda_1^2 + \lambda_2^2}\left|\left|Y \right|\right|\left|\left|\kappa -\kappa' \right|\right| \right)\\
    &=\left(\frac{\lambda_1^2+\lambda_2^2}{\lambda_1 \lambda_2}\right)\left(8\sqrt{2}N+\frac{32\sqrt{2}N\left(\lambda_1^2+\lambda_2^2\right)}{\lambda_1\lambda_2}+2\sqrt{2N}\right)  \left|\left|Y \right|\right| \left|\left|\kappa -\kappa' \right|\right|.
    \end{align*}
  \end{proof}
  
  \begin{lemma}\label{lem: kappa bul Y- kappa'bul Y}
    For $\kappa,\kappa'\in \Scomp, Y\in \mathbb{C}^N$ and $P\in \mathrm{SPD}(2)$ as defined in Definition \ref{def: ptilde} holds
    \begin{align*}
    \left|\left| \left(\kappa \bul_P Y \right)-\left(\kappa' \bul_P Y \right) \right|\right|\leq \sqrt{2N}\sqrt{\lambda_1^2 + \lambda_2^2}\left|\left|Y \right|\right|\left|\left|\kappa -\kappa' \right|\right|.
    \end{align*}
  \end{lemma}
  \begin{proof}
    Analogous to the proof from Lemma \ref{lem: norm bul Y} it follows directly
    \begin{align*}
    &\left|\left| \left(\kappa \bul_P Y \right)-\left(\kappa' \bul_P Y \right) \right|\right|
    = \left|\left|\begin{pmatrix} M(\kappa_1-\kappa_1')^T & \dots &  M(\kappa_{N}-\kappa_N')^T \end{pmatrix} (\mathrm{Id}_N \otimes P) \begin{pmatrix}
    \vect{Y_1}\\ \vdots \\\vect{Y_N}
    \end{pmatrix} \right|\right|\\
    &\leq \left|\left|\begin{pmatrix} M(\kappa_1-\kappa1')^T & \dots &  M(\kappa_{N}-\kappa_N')^T \end{pmatrix}\right|\right|\left|\left|  \left(\mathrm{Id}_N \otimes P\right)\right|\right| \left|\left|Y \right|\right| = \sqrt{2N}\sqrt{\lambda_1^2 + \lambda_2^2}\left|\left|Y \right|\right|\left|\left|\kappa -\kappa' \right|\right|.
    \end{align*}
  \end{proof}
  \begin{lemma}\label{lem: inv(kappa dia P) - inv(kappa' dia P) }
    For $\kappa, \kappa' \in \Scomp$ and $P\in \mathrm{SPD}(2)$ as defined in Definition \ref{def: ptilde} holds
    \begin{align*}
    &\left|\left|\left(\kappa \dia_P \kappa \right)^{-1} - \left(\kappa' \dia_P \kappa' \right)^{-1}\right|\right|
    \leq\left(\frac{4\sqrt{N}\sqrt{\lambda_1^2+\lambda_2^2}}{\lambda_1\lambda_2}+\frac{16\sqrt{N}(\lambda_1^2+\lambda_2^2)\sqrt{\lambda_1^2 + \lambda_2^2}}{(\lambda_1\lambda_2)^2}\right)\left|\left|\kappa-\kappa'\right|\right|.
    \end{align*}
  \end{lemma}
  
  \begin{proof}
    We get by using Lemma \ref{lem: inverse lemma}
    \begin{align*}
    &\left|\left|\left(\kappa \dia_P \kappa \right)^{-1} - \left(\kappa' \dia_P \kappa' \right)^{-1}\right|\right| = \left|\left| \frac{1}{\det\left(\kappa \dia_P \kappa \right)}\left(\kappa \dia_{\tilde{P}} \kappa  \right)- \frac{1}{\det\left(\kappa' \dia_P \kappa' \right)}\left(\kappa' \dia_{\tilde{P}} \kappa' \right)\right|\right|\\
    &\leq \frac{1}{2}\left|\frac{1}{\det\left(\kappa \dia_P \kappa \right)}+\frac{1}{\det\left(\kappa' \dia_P \kappa' \right)}\right|
    \left|\left| \left(\kappa \dia_{\tilde{P}} \kappa  \right)- \left(\kappa' \dia_{\tilde{P}} \kappa' \right)\right|\right|\\
    &+\frac{1}{2}\left|\frac{1}{\det\left(\kappa \dia_P \kappa \right)}-\frac{1}{\det\left(\kappa' \dia_P \kappa' \right)}\right|
    \left|\left| \left(\kappa \dia_{\tilde{P}} \kappa  \right)+ \left(\kappa' \dia_{\tilde{P}} \kappa' \right)\right|\right|.
    \end{align*}
    From Lemma \ref{lem: det lemma}, Lemma \ref{lem: norm kappa dia P}, Lemma \ref{lem: kappa dia P - kappa' dia P} and Lemma \ref{lem: 1/det - 1/det} it follows that 
    
    \begin{align*}
    &\left|\left|\left(\kappa \dia_P \kappa \right)^{-1} - \left(\kappa' \dia_P \kappa' \right)^{-1}\right|\right|\\
    &\leq \frac{1}{2}\left(\frac{2}{\lambda_1\lambda_2}\right)\left(4\sqrt{N}\sqrt{\lambda_1^2+\lambda_2^2}\left|\left|\kappa-\kappa'\right|\right|\right)+\frac{1}{2}\left(\frac{16\sqrt{N}(\lambda_1^2+\lambda_2^2)}{(\lambda_1\lambda_2)^2}\left|\left|\kappa-\kappa'\right|\right|\right) \left(2\sqrt{\lambda_1^2 + \lambda_2^2}\right)\\
    &\leq\left(\frac{4\sqrt{N}\sqrt{\lambda_1^2+\lambda_2^2}}{\lambda_1\lambda_2}+\frac{16\sqrt{N}(\lambda_1^2+\lambda_2^2)\sqrt{\lambda_1^2 + \lambda_2^2}}{(\lambda_1\lambda_2)^2}\right)\left|\left|\kappa-\kappa'\right|\right|.
    \end{align*}
  \end{proof}
  
  \begin{lemma}\label{lem: kappa dia P - kappa' dia P}
    For $\kappa, \kappa' \in \Scomp$ and $P\in \mathrm{SPD}(2)$ as defined in Definition \ref{def: ptilde} holds
    \begin{align*}
    \left|\left| \left(\kappa \dia_{\tilde{P}} \kappa  \right)- \left(\kappa' \dia_{\tilde{P}} \kappa' \right)\right|\right|\leq 4\sqrt{N}\sqrt{\lambda_1^2+\lambda_2^2}\left|\left|\kappa-\kappa'\right|\right|.
    \end{align*}
  \end{lemma}
  
  \begin{proof}
    First, we use the triangle inequality and the sub-multiplicative property of the Frobenius norm: 
    \begin{align*}
    & \left|\left| \left(\kappa \dia_{\tilde{P}} \kappa  \right)- \left(\kappa' \dia_{\tilde{P}} \kappa' \right)\right|\right|=
    \frac{1}{2}\Big|\Big| \big( (\kappa-\kappa') \dia_{\tilde{P}} (\kappa+\kappa')  \big)+ \big((\kappa+\kappa') \dia_{\tilde{P}} (\kappa-\kappa') \big)\Big|\Big|\\
    &\leq \Big|\Big| \big((\kappa+\kappa') \dia_{\tilde{P}} (\kappa-\kappa') \big)\Big|\Big|\\
    &\leq \left|\left|\begin{pmatrix} M(\kappa_1+\kappa_1')^T & \dots &  M(\kappa_{N}+\kappa_N')^T \end{pmatrix}\right|\right|\left|\left| (\mathrm{Id}_N \otimes \tilde{P}) 
    \right|\right|
    \left|\left|\begin{pmatrix} M(\kappa_1-\kappa_1') \\ \vdots \\  M(\kappa_{N}-\kappa_N') \end{pmatrix}\right|\right|.
    \end{align*}
    Analogously to Proof of Lemma \ref{lem: norm bul Y} we get 
    \begin{align*}
    \left|\left|\begin{pmatrix} M(\kappa_1+\kappa_1')^T & \dots &  M(\kappa_{N}+\kappa_N')^T \end{pmatrix}\right|\right|\left|\left| (\mathrm{Id}_N \otimes \tilde{P}) 
    \right|\right|\leq 2\sqrt{2N}\sqrt{\lambda_1^2+\lambda_2^2}.
    \end{align*}
    Thus, we get the desired result
    \begin{align*}
    \left|\left| \left(\kappa \dia_{\tilde{P}} \kappa  \right)- \left(\kappa' \dia_{\tilde{P}} \kappa' \right)\right|\right|\leq 4\sqrt{N}\sqrt{\lambda_1^2+\lambda_2^2}\left|\left|\kappa-\kappa'\right|\right|.
    \end{align*}
  \end{proof}
  
  \begin{lemma}\label{lem: inequality_det}
    For $A, B \in \mathbb{R}^{2\times 2}$ we have
    \begin{align*}
    \Big| \det(A)- \det(B) \Big| \leq 2|| A+B|||| A-B||.
    \end{align*}
  \end{lemma}
  \begin{proof}
    We directly calculate the determinant of the $2\times 2$ matrix and use the triangle inequality
    \begin{align*}
    &\Big| \det(A)- \det(B) \Big| 
    = \Big| a_{11}a_{22}-a_{12}a_{21} -( b_{11}b_{22}-b_{12}b_{21}) \Big|
    \leq \Big| a_{11}a_{22}-b_{11}b_{22}  \Big| + \Big| a_{12}a_{21} -b_{12}b_{21}\Big|\\
    &= \frac{1}{2}\Big| (a_{11}+b_{11})(a_{22}-b_{22})+  (a_{11}-b_{11})(a_{22}+b_{22})\Big| \\
    &+ \frac{1}{2}\Big| (a_{12}+b_{12})(a_{21}-b_{21})+  (a_{12}-b_{12})(a_{21}+b_{21})\Big|\\
    &\leq \frac{1}{2}\Big|a_{11}+b_{11}\Big|\Big|a_{22}-b_{22}\Big|+  \frac{1}{2}\Big|a_{11}-b_{11}\Big|\Big|a_{22}+b_{22}\Big| \\
    &+ \frac{1}{2}\Big| a_{12}+b_{12}\Big|\Big|a_{21}-b_{21}\Big|+  \frac{1}{2}\Big|a_{12}-b_{12}\Big|\Big|a_{21}+b_{21})\Big|.
    \end{align*}
    Using that $\Big| a_{ij}+b_{ij}\Big| \leq ||A+B||$ and  $\Big| a_{ij}-b_{ij}\Big| \leq ||A-B||$ for all $i,j=1,2$ we get 
    \begin{align*}
    \Big| \det(A)- \det(B) \Big| \leq 2|| A+B|||| A-B||.
    \end{align*}
  \end{proof}
  
  \begin{lemma}\label{lem: 1/det - 1/det}
    For $\kappa, \kappa' \in \Scomp$ and $P\in \mathrm{SPD}(2)$ as defined in Definition \ref{def: ptilde} holds
    \begin{align*}
    \left|\frac{1}{\det\left(\kappa \dia_P \kappa \right)}-\frac{1}{\det\left(\kappa' \dia_P \kappa' \right)}\right|\leq \frac{16\sqrt{N}(\lambda_1^2+\lambda_2^2)}{(\lambda_1\lambda_2)^2}\left|\left|\kappa-\kappa'\right|\right|.
    \end{align*}
  \end{lemma}
  \begin{proof}
    It follows from Lemma \ref{lem: det lemma} and Lemma \ref{lem: inequality_det}.
    \begin{align*}
    &\left|\frac{1}{\det\left(\kappa \dia_P \kappa \right)}-\frac{1}{\det\left(\kappa' \dia_P \kappa' \right)}\right|
    \leq\frac{1}{(\lambda_1\lambda_2)^2}\left|\det\left(\kappa' \dia_P \kappa' \right)-\det\left(\kappa \dia_P \kappa \right)\right|\\
    &\leq\frac{2}{(\lambda_1\lambda_2)^2}\left|\left|(\kappa' \dia_P \kappa') +(\kappa \dia_P \kappa )\right|\right|\left|\left|(\kappa' \dia_P \kappa') -(\kappa \dia_P \kappa )\right|\right|.
    \end{align*}
    Consequently, from Lemma \ref{lem: norm kappa dia P} and Lemma \ref{lem: kappa dia P - kappa' dia P} follows the desired result:
    \begin{align*}
    \left|\frac{1}{\det\left(\kappa \dia_P \kappa \right)}-\frac{1}{\det\left(\kappa' \dia_P \kappa' \right)}\right|&\leq \frac{2}{(\lambda_1\lambda_2)^2}\left(2\sqrt{\lambda_1^2+\lambda_2^2}\right)\left(4\sqrt{N}\sqrt{\lambda_1^2+\lambda_2^2}\left|\left|\kappa-\kappa'\right|\right|\right) \\
    &= \frac{16\sqrt{N}(\lambda_1^2+\lambda_2^2)}{(\lambda_1\lambda_2)^2}\left|\left|\kappa-\kappa'\right|\right|.
    \end{align*}
  \end{proof}
  
  \begin{lemma}\label{lem: dp xa yb inequality}
    Let $x, y\in \mathbb{C}$ and $a, b \in \mathbb{C}^N$. It holds
    \begin{align}
    d_P(xa, yb) \leq \frac{1}{2}  d_P\Bigg(\left(x + y\right)(a - b), 0\Bigg) +  \frac{1}{2}d_P\Bigg(\left(x - y\right)(a + b), 0\Bigg).
    \end{align}
  \end{lemma}
  \begin{proof}
    Using the Triangle inequality we get
    \begin{align*}
    d_P(xa, yb) &=  d_P\left(\frac{1}{2}(x+y)(a-b) + \frac{1}{2}(x-y)(a+b), 0\right)\\
    &\leq \frac{1}{2}  d_P\Bigg(\left(x + y\right)(a - b), 0\Bigg) +  \frac{1}{2}d_P\Bigg(\left(x - y\right)(a + b), 0\Bigg).
    \end{align*}
  \end{proof}
  \subsection{CLT}
  In this section we have the same notation as in the Section \ref{sec: CLT Drift model} in the main text. 
  \begin{lemma}\label{lem: d phi kappa}
    For $\beta^{-1}$ from Definition \ref{def: chart clt} in the main text holds  
    \begin{align*}
    d(\left[\beta^{-1}(x)\right],\left[\kappa^{(0)}\right] )^2= 1-\frac{1}{\sqrt{||x||^2+1}}.
    \end{align*}
  \end{lemma}
  
  \begin{proof}
    For $[\kappa], [\tilde{\kappa}]\in \mathfrak{P}$ it holds
    \begin{align*}
    d\left([\kappa], [\tilde{\kappa}] \right)^2 = \min_{\lambda \in \mathbb{R}} ||\kappa-e^{i\lambda}\tilde{\kappa} ||^2  = \min_{\lambda \in \mathbb{R}}2\left(1- \Re(e^{i\lambda}\kappa^*\kappa)\right).
    \end{align*}
    Consequently, if $\kappa^*\tilde{\kappa} \in \mathbb{R}_{>0}$, then $\kappa\in[\kappa], \tilde{\kappa}\in [\tilde{\kappa}]$ are in optimal position and it holds $ d\left([\kappa], [\tilde{\kappa}] \right)^2=||\kappa-\tilde{\kappa} ||^2$. 
    For $\left(R^*\frac{\tilde{x}}{||\tilde{x}||}\right)\in \beta^{-1}(x)$ from Definition \ref{def: chart clt} in the main text, 
    \begin{align*}
    \left(R^*\frac{\tilde{x}}{||\tilde{x}||}\right)^T\kappa^{(0)}=\frac{1}{||\tilde{x}||}\in \mathbb{R}_{>0}
    \end{align*} holds and thus 
    \begin{align*}
    d(\beta^{-1}(x),[\kappa^{(0)}] )^2 = \left|\left|\left(R^*\frac{\tilde{x}}{||\tilde{x}||}\right)-\kappa^{(0)}\right|\right|^2=1-\frac{1}{||\tilde{x}||} = 1-\frac{1}{\sqrt{||x||^2+1}}.
    \end{align*}
  \end{proof}
  
  \subsection{Auxiliary calculations for Section \ref{sec: clt I}}
  In this section we have the same notation as in the Section \ref{sec: clt I} in the main text. 
  \begin{lemma}\label{lem: rotation max method}
    For $\kappa \in \Scomp$ holds 
    \begin{align*}
    \argmax_{\lambda \in \mathbb{S}^1} \left|\left|\Re(e^{i\lambda }\kappa) \right|\right|^2 = \begin{cases}
    \left\{\pi-\frac{\Arg(\kappa^T\kappa)}{2}, 2\pi-\frac{\Arg(\kappa^T\kappa)}{2}\right\}, & \text{if } \kappa^T\kappa\neq 0\\
    \mathbb{S}^1, & \text{else.}
    \end{cases}
    \end{align*}
  \end{lemma}
  
  \begin{proof}
    Euler's formula gives us
    \begin{align*}
    &\argmax_{\lambda \in \mathbb{S}^1} \left|\left|\Re(e^{i\lambda }\kappa) \right|\right|^2
    =\argmax_{\lambda \in \mathbb{S}^1}
    \left|\left|e^{i\lambda}\kappa + e^{-i\lambda}\bar{\kappa} \right|\right|^2
    = \argmax_{\lambda \in \mathbb{S}^1}
    \left(\left|\left|\kappa\right|\right|^2 + e^{2i\lambda}\kappa^T\kappa + \overline{e^{2i\lambda}\kappa^T\kappa} \right).
    \end{align*}
    If $\kappa^T\kappa=0$, then all $\lambda\in \mathbb{S}^1$ maximize the expression.
    If $\kappa^T\kappa\neq 0$, then 
    there is exactly one $\alpha\in \mathbb{S}^1$ with $\alpha = \Arg(\kappa^T\kappa)$ and we get $\kappa^T\kappa = re^{i\alpha}$, where $r = |\kappa^T\kappa|>0$. 
    Substituting $\kappa^T\kappa = re^{i\alpha}$ and using the angle addition and subtraction theorems gives us: 
    \begin{align*}
    &\argmax_{\lambda \in \mathbb{S}^1} \left|\left|\Re(e^{i\lambda }\kappa) \right|\right|^2=\argmax_{\lambda \in \mathbb{S}^1} \left(r\cos(2\lambda + \alpha)\right).
    \end{align*}
    The expression $\cos(2\lambda + \alpha)$ is maximized exactly when $2\lambda + \alpha=0 \mod 2\pi$ holds. Therefore
    \begin{align*}
    \argmax_{\lambda \in \mathbb{S}^1} \left|\left|\Re(e^{i\lambda }\kappa) \right|\right|^2 = \left\{\pi-\frac{\alpha}{2}, 2\pi-\frac{\alpha}{2}\right\}.
    \end{align*}
  \end{proof}
  
  \begin{lemma}\label{lem: f well defined}
    Let $[\kappa]\in \mathfrak{P}$ then it holds for all $\kappa, \tilde{\kappa} \in [\kappa]$ that $\tilde{f}(\kappa) = \tilde{f}(\tilde{\kappa})$, where $\tilde{f}$ is defined as in Equation (\ref{eq: ftilde}) in the main text.
  \end{lemma}
  \begin{proof}
    For $\kappa \in \Scomp$ with $\kappa^T\kappa= 0$ the proposition is trivially satisfied. 
    Let $\kappa \in \Scomp$ with $\left|\kappa^T\kappa\right|= r > 0$ and let $\tilde{\kappa} \in [\kappa]$ then there is a $\lambda \in \mathbb{S}^1$ with $\tilde{\kappa} = e^{i\lambda}\kappa$.
    It follows 
    \begin{align*}
    \Arg\left(\tilde{\kappa}^T\tilde{\kappa}\right) = \Arg\left(\left(e^{i\lambda}\kappa\right)^T\left(e^{i\lambda}\kappa\right)\right)=\Arg\left(r e^{i\left(2\lambda+ \Arg\left(\kappa^T\kappa\right)\right)}\right)=2\lambda+ \Arg\left(\kappa^T\kappa\right)\mod 2\pi.
    \end{align*}
    Thus it follows 
    \begin{align*}
    \tilde{f}(\tilde{\kappa}) = \left[\Re\left(e^{\frac{-i}{2}\left(2\lambda + \Arg\left(\kappa^T\kappa\right) \right)}\left(e^{i\lambda}\kappa\right)\right)\right]_\pm=\tilde{f}(\kappa).
    \end{align*}
  \end{proof}
  
  \begin{lemma}\label{lem: jacobian g}
    Let $g$ be defined as in Equation (\ref{eq: function g}) and let $\kappa^{(0)} \in \mathfrak{P}\setminus(M_1 \cup M_2)$ then the Jacobian matrix at  $x=0$ is given by
    \begin{align*}
    J_x g(0)&= \pm\Re\Biggl(e^{-\frac{i\alpha}{2}}\Biggl(i\kappa^{0}\Re\left(i\frac{\overline{(\kappa^{(0)})^T\kappa^{(0)}}\left(\kappa^{(0)}\right)^TR^*A}{r^2}\right) + R^*A\Biggl) \Biggl).
    \end{align*}
    where $\left|(\kappa^{(0)})^T\kappa^{(0)}\right| =: r$ and $\alpha \coloneqq \Arg\left((\kappa^{(0)})^T\kappa^{(0)}\right)$.
  \end{lemma}
  
  \begin{proof}
    As $\kappa^{(0)} \notin M_1$, it follows $\left|(\kappa^{(0)})^T\kappa^{(0)}\right| =: r > 0$ and we can write $(\kappa^{(0)})^T\kappa^{(0)} \coloneqq r e^{i \alpha}$, where $\alpha \coloneqq \Arg\left((\kappa^{(0)})^T\kappa^{(0)}\right)$.
    Since $\kappa \notin M_2$, the outer function $f_{\pm}$ is the identity or minus the identity, only the sign of the Jacobian matrix is determined by this function. 
    We therefore get
    \begin{align*}
    g: \mathbb{R}^{2N-2} \rightarrow \mathbb{R}^N, \quad x \mapsto  \pm f(\beta^{-1}(x))= \pm \Re\left(e^{\frac{-i}{2}\Arg\left(\left(R^*\frac{\tilde{x}}{||\tilde{x}||}\right)^T\left(R^*\frac{\tilde{x}}{||\tilde{x}||}\right)\right) }\left(R^*\frac{\tilde{x}}{||\tilde{x}||}\right)\right)
    \end{align*}
    where $\tilde{x}$ and $R$ are defined as in Definition \ref{def: chart clt} in the main text.
    We first calculate 
    \begin{align}\label{eq: jac arg}
    &J_x\Arg\left(\left(R^*\frac{\tilde{x}}{||\tilde{x}||}\right)^T\left(R^*\frac{\tilde{x}}{||\tilde{x}||}\right)\right)=J_x \Arg\left(\left(R^*\tilde{x}\right)^T\left(R^*\tilde{x}\right)\right)\\
    & = \frac{\Re\left((R^*\tilde{x})^T(R^*\tilde{x})\right) J_x\left( \Im\left((R^*\tilde{x})^T(R^*\tilde{x})\right) \right)-\Im\left((R^*\tilde{x})^T(R^*\tilde{x})\right) J_x\left( \Re\left((R^*\tilde{x})^T(R^*\tilde{x})\right) \right) }{\Re\left((R^*\tilde{x})^T(R^*\tilde{x})\right)^2 + \Im\left((R^*\tilde{x})^T(R^*\tilde{x})\right)^2}\nonumber
    \end{align}
    where
    \begin{align*}
    J_x\Re\left((R^*\tilde{x})^T(R^*\tilde{x})\right) = \frac{1}{2}\left(J_x\left( (R^*\tilde{x})^T(R^*\tilde{x})\right)+ \overline{J_x(R^*\tilde{x})^T(R^*\tilde{x})}\right),\\
    J_x\Im\left((R^*\tilde{x})^T(R^*\tilde{x})\right) = \frac{-i}{2}\left(J_x\left( (R^*\tilde{x})^T(R^*\tilde{x})\right)- \overline{J_x(R^*\tilde{x})^T(R^*\tilde{x})}\right).
    \end{align*}
    Substitution of this into (\ref{eq: jac arg}) results in 
    \begin{align*}
    &J_x \Arg\left(\left(R^*\tilde{x}\right)^T\left(R^*\tilde{x}\right)\right)  \\
    & = \frac{-\frac{1}{2}\left(\Im\left((R^*\tilde{x})^T(R^*\tilde{x})\right) + \Re\left((R^*\tilde{x})^T(R^*\tilde{x})\right) i  \right)\left(J_x (R^*\tilde{x})^T(R^*\tilde{x})\right)}{\Re\left((R^*\tilde{x})^T(R^*\tilde{x})\right)^2 + \Im\left((R^*\tilde{x})^T(R^*\tilde{x})\right)^2}\\
    &-  \frac{\frac{1}{2}\left(\Im\left((R^*\tilde{x})^T(R^*\tilde{x})\right) - \Re\left((R^*\tilde{x})^T(R^*\tilde{x})\right) i  \right)\left(\overline{J_x(R^*\tilde{x})^T(R^*\tilde{x})}\right)
    }{\Re\left((R^*\tilde{x})^T(R^*\tilde{x})\right)^2 + \Im\left((R^*\tilde{x})^T(R^*\tilde{x})\right)^2}\\
    & =  -\Re \left( \frac{\left(\Im\left((R^*\tilde{x})^T(R^*\tilde{x})\right) + \Re\left((R^*\tilde{x})^T(R^*\tilde{x})\right) i  \right)J_x\left( (R^*\tilde{x})^T(R^*\tilde{x})\right)}{\Re\left((R^*\tilde{x})^T(R^*\tilde{x})\right)^2 + \Im\left((R^*\tilde{x})^T(R^*\tilde{x})\right)^2}\right)
    \end{align*}
    where 
    \begin{align*}
    J_x (R^*\tilde{x})^T(R^*\tilde{x}) = 2\tilde{x}^T\left(R^*\right)^TR^*A, \quad \text{where} \quad         A \coloneqq \begin{pmatrix}
    1 & i & 0 & 0& 0&\dots & 0& 0 \\
    0 & 0 & 1 & i& 0&\dots & 0& 0 \\
    \vdots & \vdots & \vdots & \vdots&\vdots  & \vdots& \vdots& \vdots \\
    0 & 0 & 0 & 0& 0&\dots & 1& i \\
    0 & 0 & 0 & 0& 0&\dots & 0& 0 \\
    \end{pmatrix}\in \mathbb{C}^{N\times 2(N-1)}.
    \end{align*}
    Using that same matrix $A$ we get
    \begin{align}\label{eq: Jac general}
    J_x R^*\frac{\tilde{x}}{||\tilde{x}||} 
    = R^*\left(\tilde{x}\left(J_x\frac{1}{||\tilde{x}||}\right)+\frac{1}{||\tilde{x}||}\left(J_x\tilde{x}\right) \right)
    =R^*\left(-\frac{1}{||\tilde{x}||^{3/2}}\tilde{x}x^T+\frac{1}{||\tilde{x}||}A\right).
    \end{align}
    Consequently, we get 
    \begin{align*}
    J_x g(x)= \pm \Re\Biggl(&e^{\frac{-i}{2}\Arg\left(\left(R^*\tilde{x}\right)^T\left(R^*\tilde{x}\right)\right)} \Biggl(-\frac{i}{2}\left(R^*\frac{\tilde{x}}{||\tilde{x}||}\right)J_x\left( \Arg\left(\left(R^*\tilde{x}\right)^T\left(R^*\tilde{x}\right)\right)\right) + J_x R^*\frac{\tilde{x}}{||\tilde{x}||}\Biggl) \Biggl).
    \end{align*}
    Inserting $x=0$ gives us 
    \begin{align*}
    e^{\frac{-i}{2}\Arg\left(\left(R^*\tilde{x}\right)^T\left(R^*\tilde{x}\right)\right)}\Biggl|_{x=0} &= \pm \exp\left(\frac{-i\alpha}{2}\right)\\
    J_x R^*\frac{\tilde{x}}{||\tilde{x}||}\Biggl|_{x=0} &= R^*A \\
    -\frac{i}{2}\left(R^*\frac{\tilde{x}}{||\tilde{x}||}\right)J_x\left( \Arg\left(\left(R^*\tilde{x}\right)^T\left(R^*\tilde{x}\right)\right)\right)\Biggl|_{x=0} &= i\kappa^{0}\Re\left(i\frac{\overline{(\kappa^{(0)})^T\kappa^{(0)}}\left(\kappa^{(0)}\right)^TR^*A}{r^2}\right).
    \end{align*}
    By substituting into Equation (\ref{eq: Jac general}) we get
    \begin{align*}
    J_x g(0)&= \pm\Re\Biggl(e^{-\frac{i\alpha}{2}}\Biggl(i\kappa^{0}\Re\left(i\frac{\overline{(\kappa^{(0)})^T\kappa^{(0)}}\left(\kappa^{(0)}\right)^TR^*A}{r^2}\right) + R^*A\Biggl) \Biggl).
    \end{align*}
  \end{proof}
  \FloatBarrier
  \begin{figure}[ht!]
    \centering
    \includegraphics[width=0.335\textwidth]{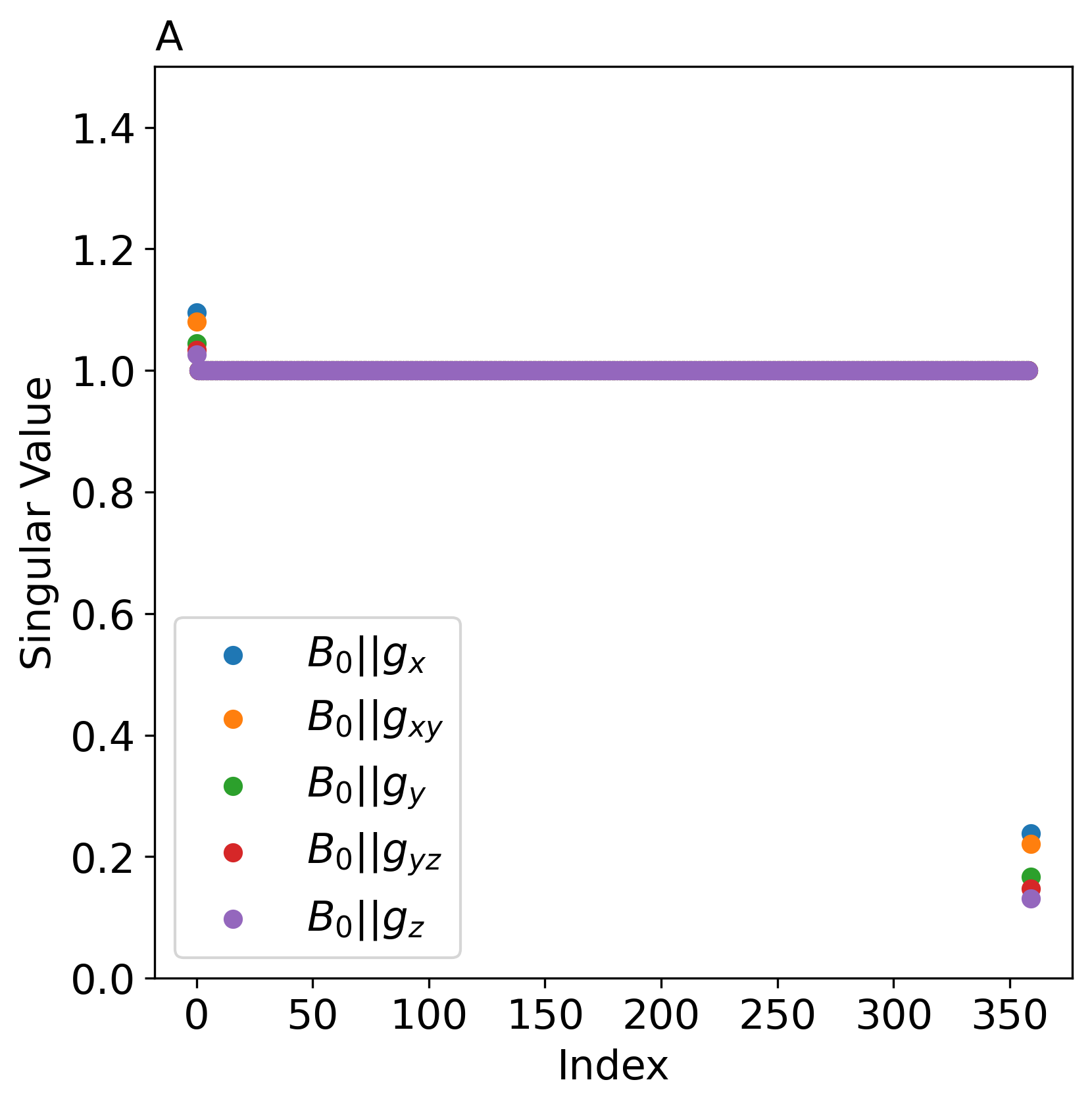}
    \includegraphics[width=0.31\textwidth]{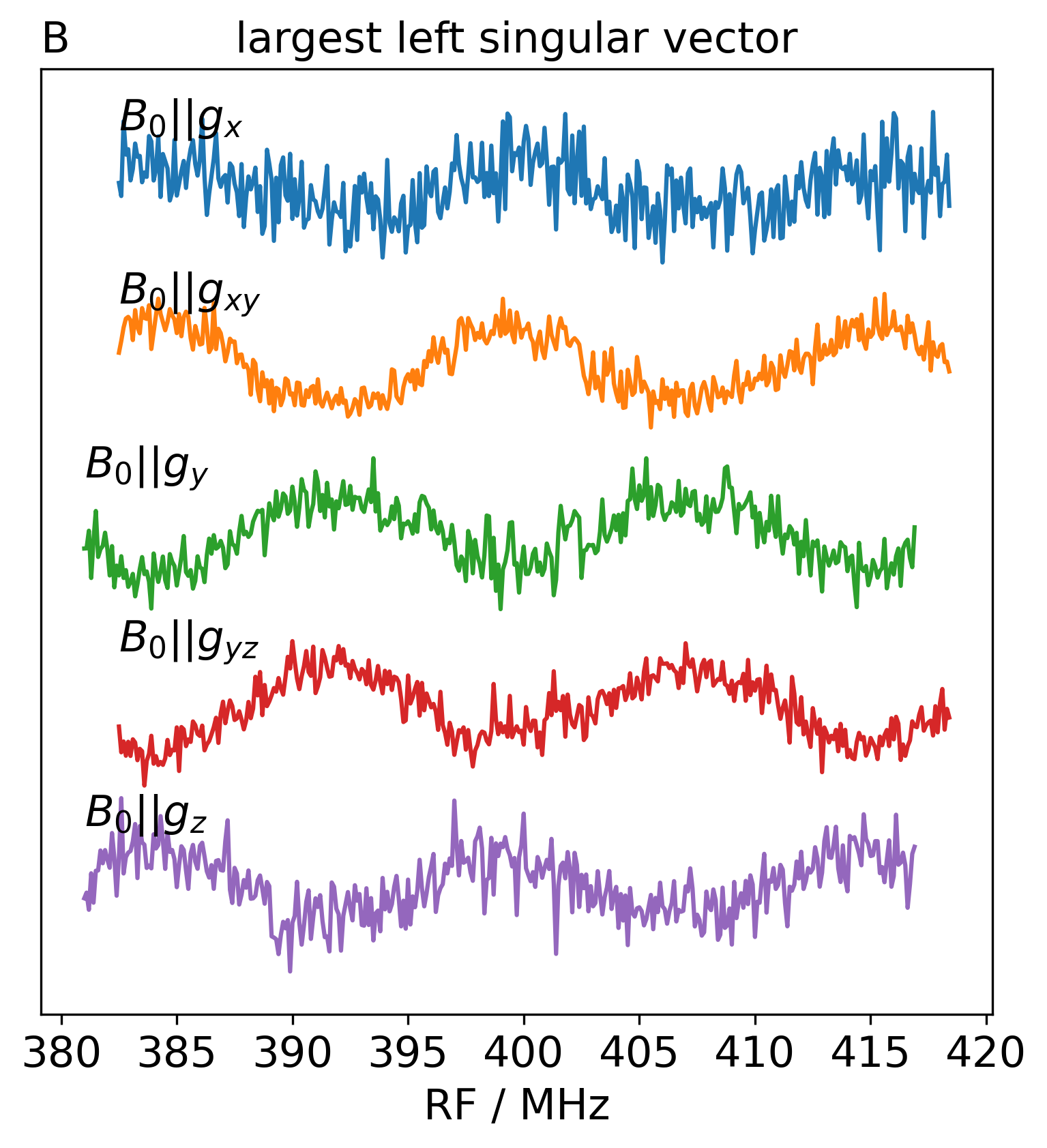}
    \includegraphics[width=0.31\textwidth]{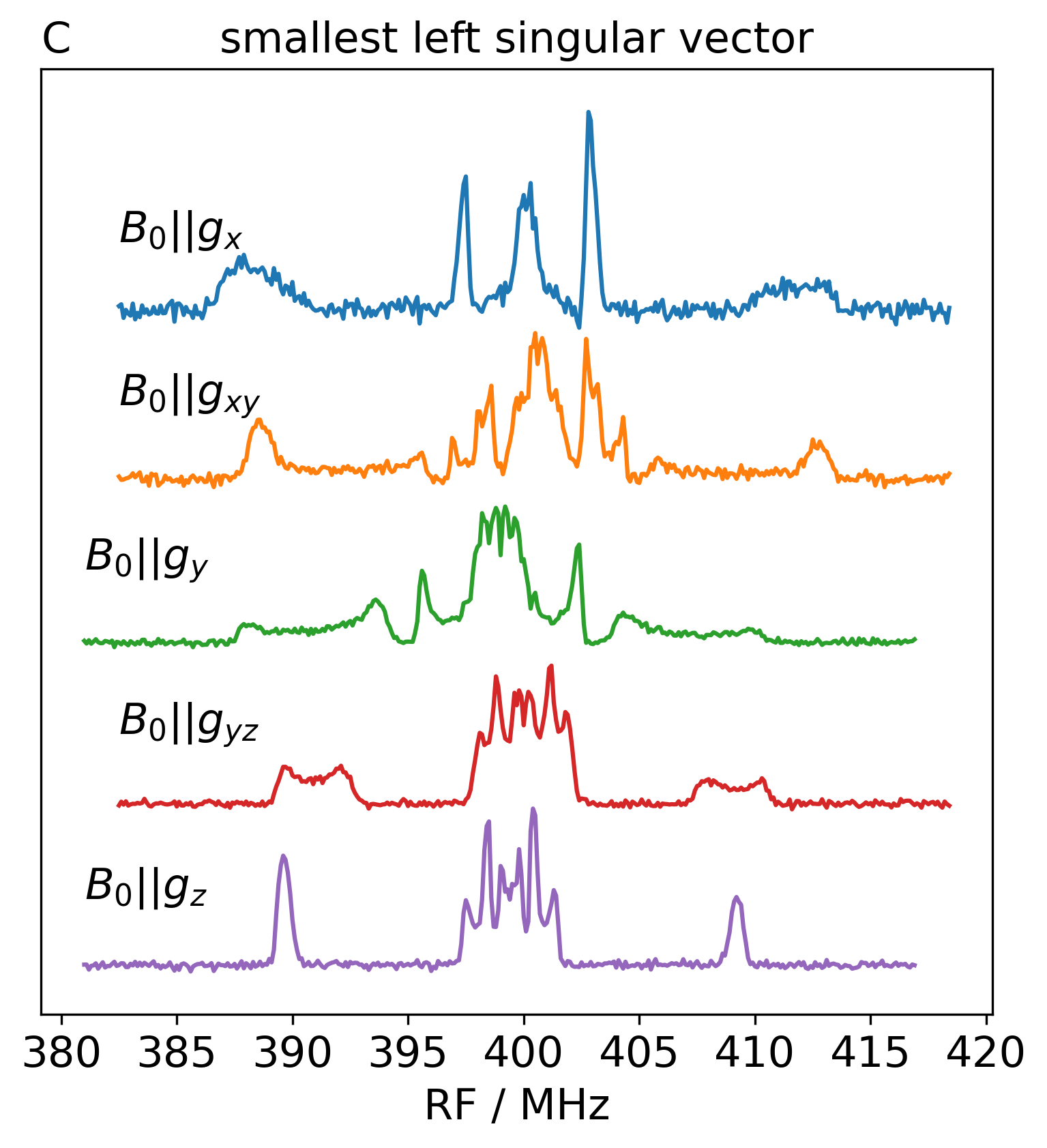}
    \caption{The singular values of $J_x g(0)$ for the different orientations from a chemical sample of the D2-$Y_{122}^{\bullet}$ E.~coli ribonucleotide reductase. Remarkably, almost all singular values are equal to $1$, with one value being slightly larger and one value being markedly smaller (but clearly separated from zero). The variation of the function $g$ stems from two sources. First, changes in $\kappa$ are directly translated into changes in the spectrum, which account for the flat eigenvalue spectrum. Second, the complex rotation by $\lambda$, which depends on $\kappa$, changes the spectrum. The eigenvector to the smallest eigenvalue therefore corresponds very closely to the spectrum itself since we evaluate the Jacobian at this point and thus, when varying $\kappa$ and hence $\lambda$, the change is mostly tangential to that direction. The eigenvector to the largest eigenvalue corresponds closely to the ``imaginary part of the spectrum'' which is projected out, so when varying $\kappa$ the corresponding variation in $\lambda$, which mixes more or less of the wave into the spectrum, compounds the change in this direction, leading to an increased eigenvalue.}
    \label{fig: singular values jacobian}
  \end{figure}
  
  \newpage
  \section{Technical Theorems and Lemmas for Section \ref{sec: Consistency Drift model unknown Sigma} in the main text}
  In this section, we prove technical lemmas for Section \ref{sec: Consistency Drift model unknown Sigma} in the main text. Consequently, we have the same notation, in particular for $\rho$, $\mathfrak{Q}$ and $\mathfrak{P}$.
  \begin{lemma}\label{lem: A2 is not constant}
    Under Definiton \ref{ass: drift model sigma},  it holds for $([\kappa], P) \in \mathfrak{P}$ 
    \begin{align*}
    \int d_P\left(\epsilon,  \hat{\phi}(\kappa,P,\epsilon)\kappa\right)^2\diff\mathbb{P}\left(\epsilon\right)= N\Tr\left(\Sigma^{(0)}P\right) -\Tr\left( \left(\kappa \dia_P \kappa  \right)^{-1} \left(\kappa \dia_{P\Sigma^{(0)}P} \kappa \right)\right)
    \end{align*} 
    where $\kappa \in [\kappa]$
  \end{lemma}
  
  \begin{proof}
    Using Lemma \ref{lem: rho formula} we get 
    \begin{align*}
    \int d_P\left(\epsilon,  \hat{\phi}(\kappa,P,\epsilon)\kappa\right)^2\diff\mathbb{P}\left(\epsilon\right) = \int \spr{\epsilon}{\epsilon}_P  - \spr{\hat{\phi}(\kappa,P, \epsilon) \kappa}{\epsilon}_P\diff\mathbb{P}\left(\epsilon\right)
    \end{align*}
    First, we calculate 
    \begin{align*}
    \int \spr{\epsilon}{\epsilon}_P\diff\mathbb{P}\left(\epsilon\right)&= \sum_{\nu=1}^N\Tr\left(\mathbb{V}\left(\sqrt{P}\vect{\epsilon_\nu}\right)\right)=\sum_{\nu=1}^N\Tr\left(\sqrt{P}\Sigma^{(0)}\sqrt{P}\right)=N\Tr\left(\Sigma^{(0)}P\right).
    \end{align*}
    By using the linearity and the cyclic property of the trace operator we get
    \begin{align*}
    &\int \spr{\hat{\phi}(\kappa,P, \epsilon) \kappa}{\epsilon}_P\diff\mathbb{P}\left(\epsilon\right) 
    = \int \sum_{\nu=1}^N\vect{\epsilon_\nu}P M(\kappa_\nu) \vect{\hat{\phi}(\kappa,P, \epsilon)}
    \diff\mathbb{P}\left(\epsilon\right)\\
    &= \int \left(\kappa\bul_P \epsilon\right)^T \left(\kappa \dia_P \kappa  \right)^{-1} \left(\kappa\bul_P \epsilon\right)
    \diff\mathbb{P}\left(\epsilon\right) 
    =\Tr\left( \left(\kappa \dia_P \kappa  \right)^{-1} \int \left(\kappa\bul_P \epsilon\right) \left(\kappa\bul_P \epsilon\right)^T
    \diff\mathbb{P}\left(\epsilon\right)\right).
    \end{align*}
    Since $\epsilon_1,\dots, \epsilon_N\sim \mathcal{N}(0,\Sigma^{(0)})$ are i.i.id random variables, it holds that 
    \begin{align*}
    &\int \left(\kappa\bul_P \epsilon\right) \left(\kappa\bul_P \epsilon\right)^T
    \diff\mathbb{P}\left(\epsilon\right)
    = \int \left(\sum_{\nu=1}^N M(\kappa_\nu)^T P \vect{\epsilon_\nu} \right) \left(\sum_{\nu=1}^N \vect{\epsilon_\nu}^T P M(\kappa_\nu)\right)
    \diff\mathbb{P}\left(\epsilon\right)\\
    &= \sum_{\nu=1}^N M(\kappa_\nu)^T P \left(\int\vect{\epsilon_\nu}   \vect{\epsilon_\nu}^T \diff\mathbb{P}\left(\epsilon\right) \right)
    P M(\kappa_\nu)= \kappa \dia_{P\Sigma^{(0)}P} \kappa .
    \end{align*}
    Consequently,
    \begin{align*}
    &\int \spr{\hat{\phi}(\kappa,P, \epsilon) \kappa}{\epsilon}_P\diff\mathbb{P}\left(\epsilon\right) =\Tr\left( \left(\kappa \dia_P \kappa  \right)^{-1} \left(\kappa \dia_{P\Sigma^{(0)}P} \kappa \right)\right).
    \end{align*}
    and therefore
    \begin{align*}
    \int d_P\left(\epsilon,  \hat{\phi}(\kappa,P,\epsilon)\kappa\right)^2\diff\mathbb{P}\left(\epsilon\right) = N\Tr\left(\Sigma^{(0)}P\right) -\Tr\left( \left(\kappa \dia_P \kappa  \right)^{-1} \left(\kappa \dia_{P\Sigma^{(0)}P} \kappa \right)\right).
    \end{align*}
  \end{proof}

  \begin{lemma}\label{lem: Partial_derivative P help calc}
    For $([\kappa], P) \in \mathfrak{P}$, $\kappa \in [\kappa]$ and $\Sigma^{(0)}\in \mathrm{SPD}(2)$ we obtain
    \begin{align*}
    (i)\quad &\Tr\left( \left(\kappa \dia_P \kappa  \right)^{-1}\left(\kappa \dia_{\left(\frac{\partial P}{\partial  p_{ij}}\right) \Sigma^{(0)}P}\kappa \right) \right) = \Tr\left( \Sigma^{(0)}P\left( \bar{\kappa} \dia_{\left(\kappa \dia_P \kappa  \right)^{-1}} \bar{\kappa} \right) \left( \frac{\partial P}{\partial  p_{ij}}\right) \right),\\
    (ii)\quad&\Tr\left(\left(\kappa \dia_P \kappa  \right)^{-1}\left(\kappa \dia_{P \Sigma^{(0)}\left(\frac{\partial P}{\partial  p_{ij}}\right)}\kappa \right)  \right) = \Tr\left(\left( \bar{\kappa} \dia_{\left(\kappa \dia_P \kappa  \right)^{-1}} \bar{\kappa} \right)P\Sigma^{(0)}\left(\frac{\partial P}{\partial  p_{ij}}\right) \right),\\
    (iii)\quad&\Tr\left(\left(\kappa \dia_P \kappa  \right)^{-1}\left(\kappa \dia_{\frac{\partial P}{\partial  p_{ij}}} \kappa  \right)\left(\kappa \dia_P \kappa  \right)^{-1} \left(\kappa \dia_{P\Sigma^{(0)}P} \kappa \right) \right) \\
    &= \Tr\left(\left(\bar{\kappa} \dia_{\left(\kappa \dia_P \kappa  \right)^{-1} \left(\kappa \dia_{P\Sigma^{(0)}P} \kappa \right)\left(\kappa \dia_P \kappa  \right)^{-1}} \bar{\kappa} \right)\left(\frac{\partial P}{\partial  p_{ij}}\right) \right).
    \end{align*}
  \end{lemma}
  \begin{proof}
    We start with $(i)$. From the cyclic property of the trace operator we obtain
    \begin{align*}
    &\Tr\left((\left(\kappa \dia_P \kappa  \right)^{-1}\left(\kappa \dia_{\left(\frac{\partial P}{\partial  p_{ij}}\right) \Sigma^{(0)}P}\kappa \right)\right)
    =\Tr\left(\left(\kappa \dia_P \kappa  \right)^{-1}
    \left(\sum_{\nu=1}^N M(\kappa_\nu)^T\left(\frac{\partial P}{\partial  p_{ij}}\right) \Sigma^{(0)}P M(\kappa_\nu)\right)\right)\\
    &=\Tr\left(\Sigma^{(0)}P\left(\sum_{\nu=1}^N M(\kappa_\nu)\left(\kappa \dia_P \kappa  \right)^{-1}
    M(\kappa_\nu)^T\right)\left(\frac{\partial P}{\partial  p_{ij}}\right) \right)\\
    &=\Tr\left( \Sigma^{(0)}P\left( \bar{\kappa} \dia_{\left(\kappa \dia_P \kappa  \right)^{-1}} \bar{\kappa} \right) \left( \frac{\partial P}{\partial  p_{ij}}\right) \right).
    \end{align*}
    Analogously, we obtain for $(ii)$
    \begin{align*}
    &\Tr\left((\left(\kappa \dia_P \kappa  \right)^{-1}\left(\kappa \dia_{P \Sigma^{(0)}\left(\frac{\partial P}{\partial  p_{ij}}\right)}\kappa \right)\right)
    =\Tr\left(\left(\kappa \dia_P \kappa  \right)^{-1}
    \left(\sum_{\nu=1}^N M(\kappa_\nu)^TP \Sigma^{(0)}\left(\frac{\partial P}{\partial  p_{ij}}\right) M(\kappa_\nu)\right)\right)\\
    &=\Tr\left(\left(\sum_{\nu=1}^N M(\kappa_\nu)\left(\kappa \dia_P \kappa  \right)^{-1}
    M(\kappa_\nu)^T\right)\Sigma^{(0)}P\left(\frac{\partial P}{\partial  p_{ij}}\right) \right)\\
    &=\Tr\left(\left(\sum_{\nu=1}^N M(\kappa_\nu)\left(\kappa \dia_P \kappa  \right)^{-1}
    M(\kappa_\nu)^T\right)\Sigma^{(0)}P\left(\frac{\partial P}{\partial  p_{ij}}\right) \right)\\
    &=\Tr\left( \left( \bar{\kappa} \dia_{\left(\kappa \dia_P \kappa  \right)^{-1}} \bar{\kappa} \right) \Sigma^{(0)}P \left( \frac{\partial P}{\partial  p_{ij}}\right) \right).
    \end{align*}
    For $(iii)$ we use the cyclic property of the trace operator
    \begin{align*}
    &\Tr\left(\left(\kappa \dia_P \kappa  \right)^{-1}\left(\kappa \dia_{\frac{\partial P}{\partial  p_{ij}}} \kappa  \right)\left(\kappa \dia_P \kappa  \right)^{-1} \left(\kappa \dia_{P\Sigma^{(0)}P} \kappa \right) \right)\\
    &\Tr\left(\left(\kappa \dia_P \kappa  \right)^{-1} \left(\kappa \dia_{P\Sigma^{(0)}P} \kappa \right)\left(\kappa \dia_P \kappa  \right)^{-1}\left(\kappa \dia_{\frac{\partial P}{\partial  p_{ij}}} \kappa  \right) \right)\\
    &\Tr\left(\left(\kappa \dia_P \kappa  \right)^{-1} \left(\kappa \dia_{P\Sigma^{(0)}P} \kappa \right)\left(\kappa \dia_P \kappa  \right)^{-1}\left(\sum_{\nu=1}^N M(\kappa_\nu)^T\left(\frac{\partial P}{\partial  p_{ij}}\right) M(\kappa_\nu)\right) \right)\\
    &\Tr\left(\sum_{\nu=1}^N \left( M(\kappa_\nu)\left(\kappa \dia_P \kappa  \right)^{-1} \left(\kappa \dia_{P\Sigma^{(0)}P} \kappa \right)\left(\kappa \dia_P \kappa  \right)^{-1}M(\kappa_\nu)^T\right)\left(\frac{\partial P}{\partial  p_{ij}}\right) \right)\\
    &= \Tr\left(\left(\bar{\kappa} \dia_{\left(\kappa \dia_P \kappa  \right)^{-1} \left(\kappa \dia_{P\Sigma^{(0)}P} \kappa \right)\left(\kappa \dia_P \kappa  \right)^{-1}} \bar{\kappa} \right)\left(\frac{\partial P}{\partial  p_{ij}}\right) \right).
    \end{align*}
  \end{proof}
  
  \begin{lemma}\label{lem: Partial_derivative P help 2}
    Let $f:\mathrm{SPD}(2) \rightarrow \mathbb{R}$ be a differentiable function with the property 
    \begin{align*}
    \frac{\partial f(P)}{\partial p_{ij}} = \Tr \left(A  \left(\frac{\partial P}{\partial p_{ij}}\right)\right)
    \end{align*}
    where $P=\begin{pmatrix}
    p_{11}& p_{12}\\
    p_{12}& p_{22}
    \end{pmatrix} \in \mathrm{SPD}(2)$ and $A$ is any symmetric matrix $A = \begin{pmatrix}
    a_{11}& a_{12}\\
    a_{12}& a_{22}
    \end{pmatrix} \in \mathbb{R}^{2\times 2}$.
    Then holds
    \begin{align*}
    \frac{\partial f(P)}{\partial P} = 2A - \mathrm{diag}(A).
    \end{align*}
  \end{lemma}
  
  \begin{proof}
    It holds
    \begin{align*}
    \Tr\left( A \left(\frac{\partial P}{\partial  p_{11}}\right)  \right) = a_{11}, \quad \Tr\left( A \left(\frac{\partial P}{\partial  p_{22}}\right)  \right) = a_{22}, \quad \Tr\left( A \left(\frac{\partial P}{\partial  p_{12}}\right)  \right) = 2a_{12} 
    \end{align*}
    and therefore
    \begin{align*}
    \frac{\partial f(P)}{\partial P} = 2A - \mathrm{diag}(A).
    \end{align*}
  \end{proof}
  
  \begin{lemma}\label{lem: Partial_derivative P}
    For $([\kappa], P) \in \mathfrak{P}$, $\kappa \in [\kappa]$ and $\Sigma^{(0)}\in \mathrm{SPD}(2)$ we obtain
    \begin{align*}
    &\frac{\partial}{\partial P}\left(N\Tr\left(\Sigma^{(0)}P\right) -\Tr\left( \left(\kappa \dia_P \kappa  \right)^{-1} \left(\kappa \dia_{P\Sigma^{(0)}P} \kappa \right)\right)-N\log(\det(P))\right) \\
    &= 2\left(\left( \bar{\kappa} \dia_{\left(\kappa \dia_P \kappa  \right)^{-1}} \bar{\kappa} \right)P\Sigma^{(0)} 
    +\Sigma^{(0)}P\left( \bar{\kappa} \dia_{\left(\kappa \dia_P \kappa  \right)^{-1}} \bar{\kappa} \right)
    - \left(\bar{\kappa} \dia_{\left(\kappa \dia_P \kappa  \right)^{-1} \left(\kappa \dia_{P\Sigma^{(0)}P} \kappa \right)\left(\kappa \dia_P \kappa  \right)^{-1}} \bar{\kappa} \right)\right)\\
    &-\mathrm{diag}\left(\left( \bar{\kappa} \dia_{\left(\kappa \dia_P \kappa  \right)^{-1}} \bar{\kappa} \right)P\Sigma^{(0)} 
    +\Sigma^{(0)}P\left( \bar{\kappa} \dia_{\left(\kappa \dia_P \kappa  \right)^{-1}} \bar{\kappa} \right)
    - \left(\bar{\kappa} \dia_{\left(\kappa \dia_P \kappa  \right)^{-1} \left(\kappa \dia_{P\Sigma^{(0)}P} \kappa \right)\left(\kappa \dia_P \kappa  \right)^{-1}} \bar{\kappa} \right)\right)\\
    &+N\left(2\Sigma^{(0)}-\mathrm{diag}(\Sigma^{(0)})\right)-N\left(2P^{-1}-\mathrm{diag}\left((2P^{-1}\right)\right).
    \end{align*}
  \end{lemma}
  
  \begin{proof}
    Since $P, \Sigma^{(0)}\in \mathrm{SPD}(2)$ we can write $P=\begin{pmatrix} p_{11}& p_{12}\\ p_{12} & p_{22} \end{pmatrix}$ and $\Sigma^{(0)}=\begin{pmatrix} \sigma^{(0)}_{11}& \sigma^{(0)}_{12}\\ \sigma^{(0)}_{12} & \sigma^{(0)}_{22} \end{pmatrix}$.
    For the first term in the sum we get
    \begin{align*}
    \frac{\partial \Tr\left(\Sigma^{(0)}P\right)}{\partial p_{11}}  = \sigma^{(0)}_{11}, \quad \frac{\partial \Tr\left(\Sigma^{(0)}P\right)}{\partial p_{22}}  =  \sigma^{(0)}_{22}, \quad
    \frac{\partial \Tr\left(\Sigma^{(0)}P\right)}{\partial p_{12}} =  2\sigma^{(0)}_{12}
    \end{align*}
    and for the third term
    \begin{align*}
    \frac{\partial \log(\det(P))}{\partial p_{11}}  = \frac{p_{22}}{\det(P)}, \quad 
    \frac{\partial \log(\det(P))}{\partial p_{22}}  =  \frac{p_{11}}{\det(P)}, \quad
    \frac{\partial \log(\det(P))}{\partial p_{12}} =  \frac{-2p_{12}}{\det(P)}.
    \end{align*}
    Thus we get 
    \begin{align}\label{eq: partial P: first and third part}
    &\frac{\partial}{\partial P} \left(N\Tr\left(\Sigma^{(0)}P\right)-N\log(\det(P))\right)\nonumber\\
    &= N\left(2\Sigma^{(0)}-\mathrm{diag}(\Sigma^{(0)})\right)-N\left(2P^{-1}-\mathrm{diag}\left((2P^{-1}\right)\right).
    \end{align}
    For the second term, we calculate the partial derivatives. For this purpose, we first consider the following auxiliary calculations
    \begin{align*}
    0 = \frac{\partial}{\partial  p_{ij}}\Big(\left(\kappa \dia_P \kappa  \right)^{-1}\left(\kappa \dia_P \kappa  \right)\Big)
    =\left(\frac{\partial}{\partial  p_{ij}}\left(\kappa \dia_P \kappa  \right)^{-1} \right)\left(\kappa \dia_P \kappa  \right)
    + \left(\kappa \dia_P \kappa  \right)^{-1}\left(\frac{\partial}{\partial  p_{ij}}\left(\kappa \dia_P \kappa  \right)\right)
    \end{align*}
    and therefore
    \begin{align}\label{eq: partial derivative kappa dia_P kappa inverse}
    \frac{\partial \left(\kappa \dia_P \kappa  \right)^{-1}}{\partial p_{ij}}
    =-\left(\kappa \dia_P \kappa  \right)^{-1}\left(\kappa \dia_{\frac{\partial P}{\partial  p_{ij}}} \kappa  \right)\left(\kappa \dia_P \kappa  \right)^{-1}.
    \end{align}
    We also calculate 
    \begin{align}\label{eq: partial derivative kappa dia_(PSigma0P) kappa}
    \frac{\partial \left(\kappa \dia_{P\Sigma^{(0)}P} \kappa \right)}{\partial p_{ij}}
    = \left(\kappa \dia_{\left(\frac{\partial P}{\partial  p_{ij}}\right) \Sigma^{(0)}P}\kappa \right)
    +\left(\kappa \dia_{P \Sigma^{(0)}\left(\frac{\partial P}{\partial  p_{ij}}\right)}\kappa \right).
    \end{align}
    By utilizing equations (\ref{eq: partial derivative kappa dia_P kappa inverse}) and (\ref{eq: partial derivative kappa dia_(PSigma0P) kappa}) and Lemma \ref{lem: Partial_derivative P help calc}, we can deduce that
    \begin{align}\label{eq: split der}
    \frac{\partial}{\partial  p_{ij}} \Tr&\left( \left(\kappa \dia_P \kappa  \right)^{-1} \left(\kappa \dia_{P\Sigma^{(0)}P} \kappa \right)\right)\nonumber\\
    =\Tr&\Big(\left(\kappa \dia_P \kappa  \right)^{-1}\left(\kappa \dia_{\left(\frac{\partial P}{\partial  p_{ij}}\right) \Sigma^{(0)}P}\kappa \right)
    +\left(\kappa \dia_P \kappa  \right)^{-1}\left(\kappa \dia_{P \Sigma^{(0)}\left(\frac{\partial P}{\partial  p_{ij}}\right)}\kappa \right)\nonumber  
    \\
    &-\left(\kappa \dia_P \kappa  \right)^{-1}\left(\kappa \dia_{\frac{\partial P}{\partial  p_{ij}}} \kappa  \right)\left(\kappa \dia_P \kappa  \right)^{-1} \left(\kappa \dia_{P\Sigma^{(0)}P} \kappa \right)\Big)\nonumber\\
    = \Tr&\Bigg(\Sigma^{(0)}P\left( \bar{\kappa} \dia_{\left(\kappa \dia_P \kappa  \right)^{-1}} \bar{\kappa} \right)\left( \frac{\partial P}{\partial  p_{ij}}\right)
    +\left( \bar{\kappa} \dia_{\left(\kappa \dia_P \kappa  \right)^{-1}} \bar{\kappa} \right)P\Sigma^{(0)}\left(\frac{\partial P}{\partial  p_{ij}}\right) \\
    &-\left(\bar{\kappa} \dia_{\left(\kappa \dia_P \kappa  \right)^{-1} \left(\kappa \dia_{P\Sigma^{(0)}P} \kappa \right)\left(\kappa \dia_P \kappa  \right)^{-1}} \bar{\kappa} \right)\left(\frac{\partial P}{\partial  p_{ij}}\right)\Bigg)\nonumber.
    \end{align}
    We obtain from (\ref{eq: split der}) and Lemma \ref{lem: Partial_derivative P help 2}
    \begin{align*}
    &\frac{\partial}{\partial  P} \Tr\left( \left(\kappa \dia_P \kappa  \right)^{-1} \left(\kappa \dia_{P\Sigma^{(0)}P} \kappa \right)\right) \\
    &= 2\left(\left( \bar{\kappa} \dia_{\left(\kappa \dia_P \kappa  \right)^{-1}} \bar{\kappa} \right)P\Sigma^{(0)} 
    +\Sigma^{(0)}P\left( \bar{\kappa} \dia_{\left(\kappa \dia_P \kappa  \right)^{-1}} \bar{\kappa} \right)
    - \left(\bar{\kappa} \dia_{\left(\kappa \dia_P \kappa  \right)^{-1} \left(\kappa \dia_{P\Sigma^{(0)}P} \kappa \right)\left(\kappa \dia_P \kappa  \right)^{-1}} \bar{\kappa} \right)\right)\\
    &-\mathrm{diag}\left(\left( \bar{\kappa} \dia_{\left(\kappa \dia_P \kappa  \right)^{-1}} \bar{\kappa} \right)P\Sigma^{(0)} 
    +\Sigma^{(0)}P\left( \bar{\kappa} \dia_{\left(\kappa \dia_P \kappa  \right)^{-1}} \bar{\kappa} \right)
    - \left(\bar{\kappa} \dia_{\left(\kappa \dia_P \kappa  \right)^{-1} \left(\kappa \dia_{P\Sigma^{(0)}P} \kappa \right)\left(\kappa \dia_P \kappa  \right)^{-1}} \bar{\kappa} \right)\right).
    \end{align*}
    Using (\ref{eq: partial P: first and third part}) we get the desired result
    \begin{align*}
    &\frac{\partial}{\partial P}\left(N\Tr\left(\Sigma^{(0)}P\right) -\Tr\left( \left(\kappa \dia_P \kappa  \right)^{-1} \left(\kappa \dia_{P\Sigma^{(0)}P} \kappa \right)\right)-N\log(\det(P))\right) \\
    &= 2\left(\left( \bar{\kappa} \dia_{\left(\kappa \dia_P \kappa  \right)^{-1}} \bar{\kappa} \right)P\Sigma^{(0)}  +\Sigma^{(0)}P\left( \bar{\kappa} \dia_{\left(\kappa \dia_P \kappa  \right)^{-1}} \bar{\kappa} \right)
    - \left(\bar{\kappa} \dia_{\left(\kappa \dia_P \kappa  \right)^{-1} \left(\kappa \dia_{P\Sigma^{(0)}P} \kappa \right)\left(\kappa \dia_P \kappa  \right)^{-1}} \bar{\kappa} \right)\right)\\
    &-\mathrm{diag}\left(\left( \bar{\kappa} \dia_{\left(\kappa \dia_P \kappa  \right)^{-1}} \bar{\kappa} \right)P\Sigma^{(0)} 
    +\Sigma^{(0)}P\left( \bar{\kappa} \dia_{\left(\kappa \dia_P \kappa  \right)^{-1}} \bar{\kappa} \right)
    - \left(\bar{\kappa} \dia_{\left(\kappa \dia_P \kappa  \right)^{-1} \left(\kappa \dia_{P\Sigma^{(0)}P} \kappa \right)\left(\kappa \dia_P \kappa  \right)^{-1}} \bar{\kappa} \right)\right)\\
    &+N\left(2\Sigma^{(0)}-\mathrm{diag}(\Sigma^{(0)})\right)-N\left(2P^{-1}-\mathrm{diag}\left((2P^{-1}\right)\right).
    \end{align*}
  \end{proof}
  
  \begin{lemma}\label{lem: insert in kappa0 P0}
    For $([\kappa], P) \in \mathfrak{P}$, $\kappa \in [\kappa]$ and $P^{(0)}, \Sigma^{(0)}\in \mathrm{SPD}(2)$ with $P^{(0)}= \left(\Sigma^{(0)}\right)^{-1}$ we obtain
    \begin{align*}
    &\frac{\partial}{\partial P}\Ffun([\kappa], P)\Big|_{[\kappa]=\left[\kappa^{(0)}\right], P=P^{(0)}}
    = 2\left(\bar{\kappa}^{(0)} \dia_{\left(\kappa^{(0)} \dia_{P^{(0)}} \kappa^{(0)}  \right)^{-1}} \bar{\kappa}^{(0)} \right)  -\mathrm{diag}\left(\bar{\kappa}^{(0)} \dia_{\left(\kappa^{(0)} \dia_{P^{(0)}} \kappa^{(0)}  \right)^{-1}} \bar{\kappa}^{(0)} \right).
    \end{align*}
  \end{lemma}
  
  \begin{proof}
    Analogous to Section \ref{sec: Consistency Drift model unknown Sigma} in the main text, we decompose $\Ffun$ as follows
    \begin{align}\label{eq: partial Ffun partial P}
    &\frac{\partial}{\partial P}\Ffun([\kappa], P)\Big|_{[\kappa]=\left[\kappa^{(0)}\right], P=P^{(0)}}\\
    &= \frac{\partial}{\partial P} \int d_P\left(Y,  \hat{\phi}(\kappa,P,Y)\kappa\right)^2  \diff\mathbb{P}\left(\phi\right)\Big|_{\kappa=\kappa^{(0)}, P=P^{(0)}}\nonumber\\
    &+\frac{\partial}{\partial P} \int d_P\left(\epsilon,  \hat{\phi}(\kappa,P,\epsilon)\kappa\right)^2 \diff\mathbb{P}\left(\epsilon\right) -N\log(\det(P))\Big|_{\kappa=\kappa^{(0)}, P=P^{(0)}}.\nonumber
    \end{align}
    From Theorem \ref{theo: Frechet mean value is unique} in the main text, it follows that all $\kappa \in [\kappa^{(0)}]$ minimize the expression 
    \begin{align*}
    \int d_P\left(Y,  \hat{\phi}(\kappa,P,Y)\kappa\right)^2  \diff\mathbb{P}\left(\phi\right)
    \end{align*} 
    for any $P\in \mathrm{SPD}(2)$.
    Consequently,
    \begin{align*}
    \frac{\partial}{\partial P} \int d_P\left(Y,  \hat{\phi}(\kappa,P,Y)\kappa\right)^2 \diff\mathbb{P}\left(\phi\right)\Big|_{\kappa=\kappa^{(0)}, P=P^{(0)}} = 0.
    \end{align*}
    We utilize Lemma \ref{lem: A2 is not constant} and Lemma \ref{lem: Partial_derivative P} for the second part of (\ref{eq: partial Ffun partial P})
    \begin{align*}
    &\frac{\partial}{\partial P} \int d_P\left(\epsilon,  \hat{\phi}(\kappa,P,\epsilon)\kappa\right)^2 \diff\mathbb{P}\left(\epsilon\right) -N\log(\det(P))\Big|_{\kappa=\kappa^{(0)}, P=P^{(0)}}\\
    &= 2\left(\bar{\kappa}^{(0)} \dia_{\left(\kappa^{(0)} \dia_{P^{(0)}} \kappa^{(0)}  \right)^{-1}} \bar{\kappa}^{(0)} \right)  -\mathrm{diag}\left(\bar{\kappa}^{(0)} \dia_{\left(\kappa^{(0)} \dia_{P^{(0)}} \kappa^{(0)}  \right)^{-1}} \bar{\kappa}^{(0)} \right).
    \end{align*}
  \end{proof}

  \FloatBarrier
  \newpage
  \section{Heteroscedastic Drift Model}\label{sec: heteroscedastic appendix}
  \subsection{Goodness of Fit and Standard Deviations}\label{ssec: heteroscedastic GOF}
  \begin{table}[ht!]
    \centering
    \caption{Results of Kolmogorov–Smirnov tests for Gaussianity applied to the real and imaginary parts of the residuals $\hat{\epsilon}_{b,\nu}=Y_{b,\nu}-\hat{\psi}_b - \hat{\phi}_b\hat{\kappa}_\nu$, pooled over $b$ and $\nu$, obtained from the homoscedastic drift model applied to the \SI{94}{\GHz} data.}
    \begin{tabular}{c|c c}
      \hline
      Orientation & Real & Imaginary \\
      \hline
      $x$ &\num{1.13e-5} & \num{4.19e-6}\\
      $xy$& \num{2.08e-8} &\num{0.929}\\
      $y$& \num{2.71e-4} &\num{0.262}\\
      $yz$ & \num{2.54e-10} &\num{5.83e-20}\\
      $z$& \num{0.173} &\num{3.53e-2}\\
    \end{tabular}
    \label{tab:hom_p_values_94}
  \end{table}
  
  \begin{table}[ht!]
    \centering
    \caption{The standard deviation of the flat regions shown in Panel B of Figure \ref{fig:Signal to noise ratio heterovs average}, computed for both the averaging model and the heteroscedastic drift model.}
    \begin{tabular}{c|c|c}
      \hline
      Orientation & heteroscedastic drift model & averaging model\\
      \hline
      x & \num{4.2e-03} & \num{9.2e-03}\\
      xy & \num{2.6e-03} & \num{3.1e-03}\\
      y & \num{2.9e-03} & \num{4.4e-03}\\
      yz & \num{3.9e-03} & \num{9.0e-03}\\
      z & \num{9.3e-03} & \num{1.2e-02}
    \end{tabular}
    
    \label{tab:SNR hetero, avg}
  \end{table}
  
  \begin{figure}[ht!]
    \centering
    \includegraphics[width=\textwidth]{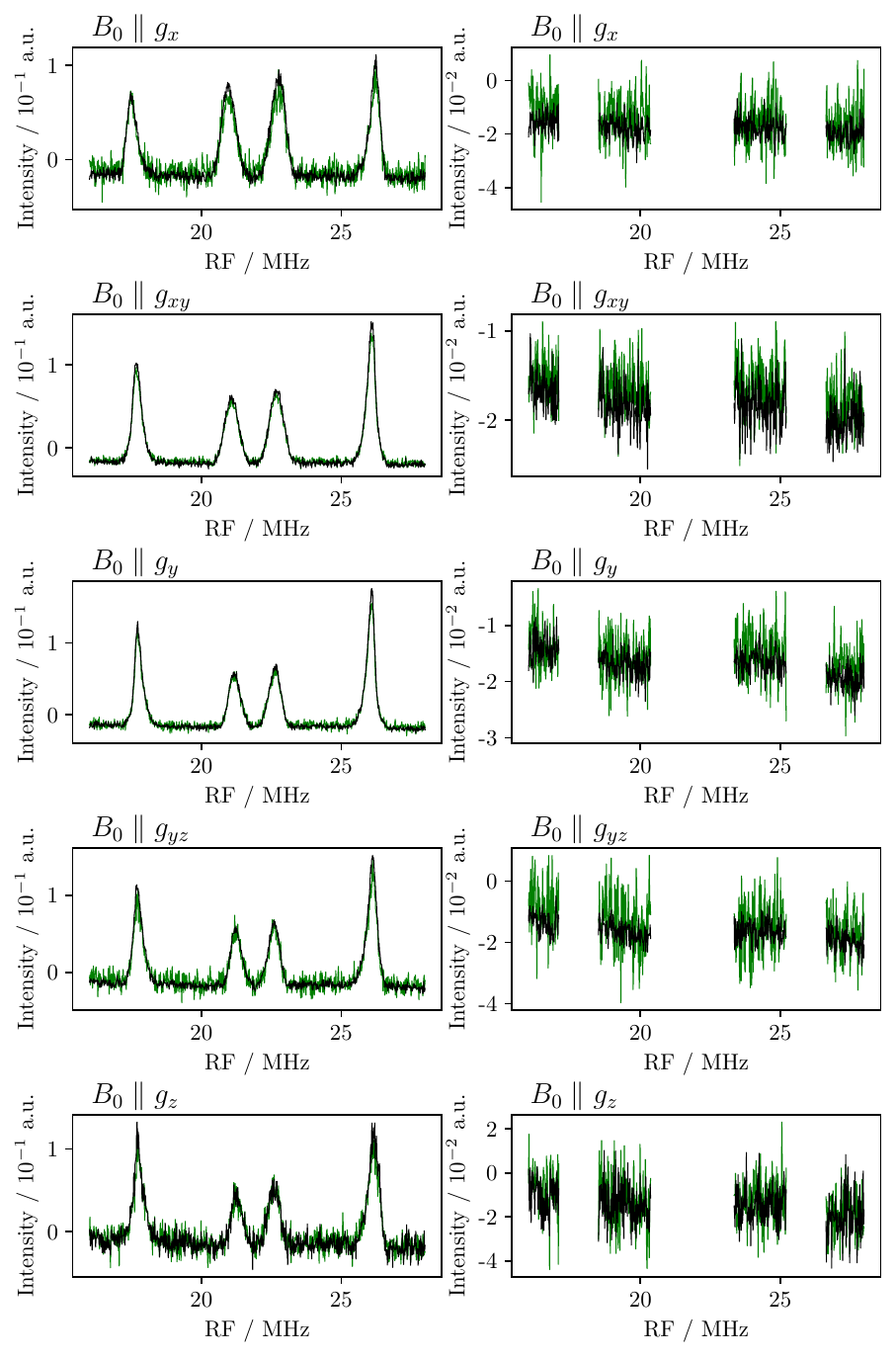}
    \caption{Comparison of the averaging model (green) and the heteroscedastic drift model (black).
    }
    \label{fig:Signal to noise ratio heterovs average}
  \end{figure}

  \begin{table}[ht!]
    \centering
    \caption{$p$-values from applying the heteroscedastic drift model to the \SI{94}{\GHz} data.}
    \begin{tabular}{c|c c}
      \hline
      orientation & $p_\real$ & $p_\imaginary$ \\
      \hline
      $x$ &$0.827$ &$0.321$\\
      $xy$& $1.11\times 10^{-3}$ &$0.984$\\
      $y$& $0.0294$ &$0.587$\\
      $yz$ &$0.269$ &$0.253$\\
      $z$& $0.755$ &$0.889$\\
    \end{tabular}
    
    \label{tab:heteroscedastic p values}
  \end{table}

\FloatBarrier
  
  \subsection{Boundary Maxima in the Heteroscedastic Drift Model}\label{ssec: heteroscedastic boundary maximum}
  The heteroscedastic drift model exhibits boundary maxima as $\Sigma_0$ tends to a rank-deficient matrix. A detailed example exhibiting these is given in \autoref{lem: boundary max}. The iterative \autoref{alg:het} fitting the above model did not find these boundary maxima when initialized from the homoscedastic drift model on the tested datasets. From the log likelihood values resulting from these fits, we looked at the upper bound for the minimal eigenvalue of $\Sigma_0$ for which these log likelihood values can be obtained by the parameter sequence constructed in \autoref{lem: boundary max}. These are reported in \autoref{tab:cut_of_real_data} and compared with the minimal eigenvalues of the estimated $\hat{\Sigma_0}$. From the differences, which are about 200 orders of magnitude, we concluded that the algorithm did indeed not find the boundary global maxima but found some local MLE. In practice, we did not actually need to restrict the parameter space for $\Sigma_0$ to impose lower bounds on its eigenvalues, even though this would reasonably represent minimum receiver noise.
  
  \begin{lemma}
    \label{lem: boundary max}
    The heteroscedastic drift model has boundary maxima.
  \end{lemma}
  \begin{proof}
    Let $Y_{b,\nu}\in\mathbb{C}$ be arbitrary. WLOG $Y\neq 0, \sum_{\nu = 0}^N Y_{1,\nu}\neq 0$. And choose $\psi_1 = \frac{1}{N+1} \sum_{\nu = 0}^N Y_{1,\nu}$. Then we can chose $\kappa_\nu$ such that the residuals $R_{1,\nu}=0$ for all $\nu = 0,\dots ,N$ by just using the averaging model estimator applied to the first batch
    \begin{align*}
    \kappa_{\nu}
    &= \frac{Y_{1,\nu}-\psi_1}{\sqrt{\sum_{\nu'=0}^N\abs{Y_{1,\nu'}-\psi_1}^2}}\\
    c
    &= 0\\
    \phi_1
    &= \sqrt{\sum_{\nu'=0}^N\abs{Y_{1,\nu'}-\psi_1}^2}\\
    \Rightarrow R_{1,\nu}
    &= Y_{1,\nu}-\psi_1-\phi_1\kappa_{\nu}
    = Y_{1,\nu} -  \frac{1}{N+1}\sum_{\nu'=0}^N Y_{1,\nu'} - \qty( Y_{1,\nu} -  \frac{1}{N+1}\sum_{\nu'=0}^N Y_{1,\nu'} )
    = 0
    \end{align*}
    We then choose a sequence $\Sigma_0^{(k)},\Tilde{\sigma}$ such that the likelihood diverges to $+\infty$ as $k\rightarrow \infty$. For notational convenience, we express all matrices in the basis $\qty{\vect{\frac{\psi_1}{\abs{\psi_1}}},\vect{i\frac{\psi_1}{\abs{\psi_1}}}}$.
    \begin{align*}
    \Sigma_0^{(k)}
    &=  \frac1k \vect{\psi_1}\vect{\psi_1}^T+ \vect{i\psi_1}\vect{i\psi_1}^T
    = \abs{\psi_1}^2\begin{pmatrix}
    \frac1k & 0 \\
    0 & 1 
    \end{pmatrix}\\
    \Tilde{\sigma} 
    &= 1\\
    \Rightarrow\Sigma_1^{(k)}
    &= \Sigma_0^{(k)} +\Tilde{\sigma}\vect{i\psi_1}\vect{i\psi_1}^T
    =  \abs{\psi_1}^2\qty(\begin{pmatrix}
    \frac1k & 0\\
    0 & 1
    \end{pmatrix} + 
    \begin{pmatrix}
    0 & 0\\
    0 & 1
    \end{pmatrix})
    = \abs{\psi_1}^2\begin{pmatrix}
    \frac1k & 0\\
    0 & 2
    \end{pmatrix}
    \end{align*}
    We first focus on the log likelihood contribution $ \ell_{Y_{1,:}}^{{(k)}}$ associated with batch $b=1$ and consider the remaining contributions later.
    \begin{align*}
    \ell_{y_{1,:}}^{{(k)}}
    &= -\frac12\left(\sum_{\nu = 0}^{N}\qty(\vect{R_{1,\nu}}^T P^{(k)}_1\vect{R_{1,\nu}}) + ({N+1})\qty(\log(\det(\Sigma_1^{(k)}))+\log((2\pi)^2))\right)\\
    &= -\frac {N+1}2\qty(
    \log({\det(\Sigma_1^{(k)})})+\log((2\pi)^2))
    = \frac {N+1}2
    \log(\frac{1}{\det(\Sigma_1^{(k)})})-({N+1})\log(2\pi)\\
    &= \frac {N+1}2
    \log(\frac{k}{2\abs{\psi_1}^4})-({N+1})\log(2\pi)\\
    \Rightarrow \lim_{k\to\infty} \ell_{Y_{1,:}}^{{(k)}}
    &=  + \infty
    \end{align*}
    We now choose $\psi_b^{(k)}$ for $b\neq1$ such that $\Sigma_b^{(k)}$ is constant in k. So let $b$ not equal to 1
    \begin{align*}
    \psi_b^{(k)}
    &= -i\sqrt{1-\frac1k}\psi_1\\
    \vect{i\psi_b^{(k)}}\vect{i\psi_b^{(k)}}^T
    &=   \qty(1-\frac1k) \vect{\psi_1}\vect{\psi_1}^T
    =\abs{\psi_1}^2\begin{pmatrix}
    1-\frac{1}{k} & 0 \\
    0 & 0 
    \end{pmatrix}\\
    \Rightarrow \Sigma^{(k)}_b
    &= \Sigma_0^{(k)} +\Tilde{\sigma}\vect{i\psi_b^{(k)}}\vect{i\psi_b^{(k)}}^T
    = \abs{\psi_1}^2 \qty( \begin{pmatrix}
    \frac1k & 0 \\
    0 & 1 
    \end{pmatrix} +
    \begin{pmatrix}
    1-\frac{1}{k} & 0 \\
    0 & 0
    \end{pmatrix})\\
    &= \abs{\psi_1}^2\begin{pmatrix}
    1 & 0 \\
    0 & 1 
    \end{pmatrix}\\
    \Rightarrow P_b^{(k)}
    &= \abs{\psi_1}^{-2} \begin{pmatrix}
    1 & 0 \\
    0 & 1 
    \end{pmatrix}
    \end{align*}
    But then the only dependency on $k$ for $b\neq 1$ in the likelihood is in the residuals. (We choose $\phi_b=-i\phi_1$ for convenience so $\phi_b\kappa_{\nu} =- i\qty(Y_{1,\nu} - \frac{1}{{N+1}}\sum_{\nu'=0}^N Y_{1,\nu'}) = -i\qty(Y_{1,\nu}-\psi_1)$)
    \begin{align*}
    \Rightarrow R_{b,\nu}^{(k)}
    &= Y_{b,\nu}- \psi_b^{(k)}-\phi_b\kappa_\nu
    =  Y_{b,\nu} + i \sqrt{1-\frac1n}\psi_1 + i\qty(Y_{1,\nu} - \psi_1)\\
    \Rightarrow \lim_{k\to\infty} R_{b,\nu}^{(k)}
    &=  Y_{b,\nu} + i Y_{1,\nu}
    \eqdef R_{b,\nu}\\
    \Rightarrow \ell_{Y_{b,:}}^{(k)}
    &= -\frac12\left(\sum_{\nu = 0}^{N}\qty(\vect{R_{b,\nu}^{(k)}}^T P^{(k)}_b\vect{R_{b,\nu}^{(k)}}) + ({N+1})\qty(\log(\det(\bm\Sigma_b^{(k)}))+\log((2\pi)^2))\right)\\
    &= -\frac12\left(\sum_{\nu = 0}^{N}\frac{\abs{R_{b,\nu}^{(k)}}^2}{\abs{\psi_1}^2} +({N+1})\qty(\log(\abs{\psi_1}^4)+\log((2\pi)^2))\right)\\
    \Rightarrow \lim_{k\to\infty} \ell_{y_{b,:}}^{(k)}
    &= -\frac12\left(\sum_{\nu = 0}^{N}\frac{\abs{R_{b,\nu}}^2}{\abs{\psi_1}^2} +({N+1})\qty(\log(\abs{\psi_1}^4)+\log((2\pi)^2))\right)
    \end{align*}
    So the log likelihood of the residuals for $b\neq 1$ does not diverge to $-\infty$ but instead converges to a finite value. But by design the log likelihood of the first batch diverges like $\log(k)$. Therefore,
    \begin{align*}
    &\lim_{k\to\infty}\ell_Y^{(k)}
    = \lim_{k\to\infty} \qty( \ell_{Y_{1,:}}^{(k)} + \sum_{b=2}^B \ell_{Y_{b,:}}^{(k)})\\
    &= \lim_{k\to\infty} \left( \frac {N+1}2
    \log(\frac{k}{2\abs{\psi_1}^4})-(N+1)\log(2\pi) \vphantom{\sum_{\nu = 0}^{N}\frac{\abs{R_{b,\nu}^{(k)}}^2}{\abs{\psi_1}^2}}\right.\\
    &\left.\vspace{2cm}+ \sum_{b= 2}^B\qty(-\frac12\left(\sum_{\nu = 0}^{N}\frac{\abs{R_{b,\nu}^{(k)}}^2}{\abs{\psi_1}^2} +(N+1)\qty(\log(\abs{\psi_1}^4)+\log((2\pi)^2))\right))\right)\numberthis\label{eqn:convergence}\\
    &= + \infty  -\frac12\left(\qty(\sum_{b= 2}^B\sum_{\nu = 0}^{N}\frac{\abs{R_{b,\nu}}^2}{\abs{\psi_1}^2} )+(N+1)(B-1)\log(\abs{\psi_1}^4)+(N+1)B\log((2\pi)^2)\right) 				 = +\infty
    \end{align*}
  \end{proof}

  \begin{table}[]
    \centering
    \caption{Applying the example of a parameter sequence with divergent log likelihood from \autoref{lem: boundary max} to the datasets, we can compare the algorithmic fit with these parameters by focusing on the sequence index $k^*$ for which the log likelihood of the fit is first reached by the sequence.}
    \begin{tabular}{c|c|c|c|c}
      \hline
      orientation & k* & smallest eigenvalue $\Sigma_0^{(k*)}$  & smallest eigenvalue of $\widehat{\Sigma_0}$ & $\widehat{\Tilde{\sigma}}$\\\hline
      x & \num{e271} & \num{e-262} & \num{e3} & \num{1.73e-4}\\
      xy & \num{e339} & \num{e-330} & \num{e3}& \num{1.69e-4} \\
      y & \num{e468} & \num{e-460} & \num{e3}& \num{1.72e-4}\\
      yz & \num{e316} & \num{e-307} & \num{e-2}& \num{1.51e-4} \\
      z & \num{e467} & \num{e-459} & \num{e4}& \num{1.79e-4}\\
    \end{tabular}
    \label{tab:cut_of_real_data}
  \end{table}
  
  \subsection{Phase Noise Truncation}\label{ssec: heteroscedastic truncation}
  Looking at the mean and variance of the wrapped Gaussian
  \begin{align*}
  \mathbb{E}\qty[\Tilde{\psi}_{b,\nu}]
  =&\psi_b\exp(-\frac{\tilde{\sigma}^2}{2})
  = \psi_b\qty(1-\frac{\tilde{\sigma}^2}{2})+\mathcal{O}(\tilde{\sigma}^4)\\
  \Var{\Tilde{\psi}_{b,\nu}}
  \defeq&\Covmat{\vect{\Tilde{\psi}_{b,\nu}}}{\vect{\Tilde{\psi}_{b,\nu}}}
  = M(\psi_b)\begin{pmatrix}\frac{1+e^{-2\tilde{\sigma}^2}-2e^{-\tilde{\sigma}^2}}{2} &0\\
  0& \frac{1-e^{-2\tilde{\sigma}^2}}{2}\\
  \end{pmatrix} M(\psi_b)^T\\
  =& M(\psi_b)\begin{pmatrix}0 +\mathcal{O}(\tilde{\sigma}^4)&0\\
  0& \tilde{\sigma}^2+\mathcal{O}(\tilde{\sigma}^4)\\
  \end{pmatrix} M(\psi_b)^T.
  \end{align*}
  we see that the expansion of the mean to higher than linear order is not consistent with the mean of Definition \ref{def: hetero} 
  due to the correction in $\tilde{\sigma}^2$ which comes from the quadratic term $-\frac{\psi_b\tilde{\sigma}^2}{2}\varphi_{b,\nu}^2$. Replacing $\psi_b$ by $\breve{\psi}_b= \psi_b\frac{2-\tilde{\sigma}^2}{2}$ in Definition \ref{def: hetero} 
  on the other hand, leads to a different noise scale parameter $\sigma^2 = \left(\frac{2}{2-\tilde{\sigma}^2}\right)^2\tilde{\sigma}^2$ as
  \begin{align*}
  \vect{\breve{\psi_b}}\sigma^2\vect{\breve{\psi_b}}^T = \vect{\breve{\psi_b}}\qty(\frac{2}{2-\tilde{\sigma}^2})^2\tilde{\sigma}^2\vect{\breve{\psi_b}}^T \eqdef \vect{\psi_b}\tilde{\sigma}^2\vect{\psi_b}^T.
  \end{align*}
  This second parametrization was used in the heteroscedastic drift model. As $\sigma^2 = \tilde{\sigma}^2 +\mathcal{O}(\tilde{\sigma}^4)$, the validity of an expansion to linear order in $\sigma^2$ is equivalent to one in $\tilde{\sigma}^2$. 
  
  The next term in the expansion of $\tilde{\psi}_{b,\nu}=\psi_b\exp{i\tilde{\sigma}\varphi_{b,\nu}}$ not modeled is the quadratic term. The variance contribution of this term is
  \begin{align*}
  \Var{-\vect{\psi_b}\frac{\tilde{\sigma}^2}{2} \varphi_{b,\nu}^2}
  = \vect{\frac{\psi_b}{\abs{\psi_b}}}\frac{\abs{\psi_b}^2\tilde{\sigma}^4}{2}\vect{\frac{\psi_b}{\abs{\psi_b}}}^T.
  \end{align*}
  Given our data, when calculated based on the MLE estimators for $\psi_b$ and $\tilde{\sigma}^2$, this is dominated by the marginal variance of $\hat{\Sigma}_0$ in the subspace spanned by $\vect{\hat{\psi}_b}$ given by $\norm{\frac{\vect{\hat{\psi}_b}}{\abs{\hat{\psi_b}}}}_{\hat{\Sigma_0}}$ justifying the truncation. Even when minimizing this comparison over the batch parameter independently, the marginal variance is still larger by 2 orders of magnitude as reported in \autoref{tab:comparison Sigma_0 sigma_tilde}. Based on this, explicit modelling of the quadratic term was deemed unnecessary.
  
  \begin{table}[ht!]
    \centering
    \caption{Comparison of the contribution of $\Sigma_0$ in the direction $\psi_b$ (minimized over the batches) with the maximal contribution of $\frac{\sigma^4}{2}\abs{\psi_b}^2$ (maximized over the batches) in the heteroscedastic drift model. Noise contributions from the quadratic order term in the wrapped Gaussian expansion for the phase noise are at least two orders of magnitude smaller than those of $\Sigma_0$.}
    \begin{tabular}{c|c c}
      \hline
      orientation & $\min\limits_{b\in B}\vect{\frac{\psi_b}{\abs{\psi_b}}}^T\Sigma_0\vect{\frac{\psi_b}{\abs{\psi_b}}}$ & $\max\limits_{b\in B}\frac{\abs{\psi_b}^2\sigma^4}{2}$ \\
      \hline
      x & \num{8.0e+3} & \num{2.6e+1}\\
      xy & \num{3.9e+3} & \num{2.3e+1}\\
      y & \num{2.8e+3} & \num{1.5e+1}\\
      yz & \num{2.4e+3} & \num{8.4e+0}\\
      z & \num{1.7e+3} & \num{8.1e+0}
    \end{tabular}
    \label{tab:comparison Sigma_0 sigma_tilde}
  \end{table}
  
  \subsection{Algorithm}
  We included the update step $\Tilde{\psi} = \psi - \Delta_c\phi$, $\Tilde{c} = c + \Delta_c$ for a numerically optimized value of $\Delta_c$ in the optimizer in order to improve convergence properties. It does not change the residuals but only the covariance matrix. Without it, the log likelihood improvements stagnate. 
  Adding this update from the beginning led to unstable trajectories of $\hat{c}$ over the iterations, so the algorithm we used starts this additional update after the 25th iteration. 
  
  The initialization of $\Sigma_0$ and $\Tilde{\sigma}$ is done by regressing the matrices $\vect{i\hat{\psi_b}_{hom}}\vect{i\hat{\psi_b}_{hom}}^T$, which are obtained from the residuals arising from fitting the homoscedastic drift model, onto the sample covariance matrix of the homoscedastic drift model. The intercept is taken as an initial value for $\Sigma_0$ and the slope initializes $\Tilde{\sigma}$.
  
  The full algorithm is given in \autoref{alg:het}.
  
  \begin{algorithm}
    \caption{Heteroscedastic drift model MLE}\label{alg:het}
    \begin{algorithmic}
      \State \Load $\bm y$
      \State $\widehat{\bm \psi}^{(0)},\widehat{\bm\phi}^{(0)},\widehat{\bm{\breve{\kappa}}}^{(0)},\widehat{\bm\Sigma}_{hom}\gets \autoref{alg:hom}(\bm y)$
      \State $R\gets \vect{y-\hat{\psi}^{(0)}(\bm1_{N+1})^T-\hat{\phi}^{(0)}(\hat{\breve{\kappa}}^{(0)})^T}$
      \State $\hat\Psi,\hat S\gets \vect{i\hat{\psi}^{(0)}}\vect{i\hat{\psi}^{(0)}}^T, \frac{1}{N+1}\sum_{\nu =0}^NR_{:,\nu}R_{:,\nu}^T$
      \State $\widehat{\Tilde{\sigma}}^{(0)},\widehat{\bm{\Sigma_0}}^{(0)}\gets LinReg\qty(\hat{\Psi},\hat{S})$
      \State $k \gets 0$
      \While{$k \leq maxiter = 200$}
      \State $\ell^{(k)} \gets \ell\qty(\bm{y},\widehat{\bm{\psi}}^{(k)},\widehat{\bm{\phi}}^{(k)},\widehat{\bm{\breve{\kappa}}}^{(k)},\widehat{\Tilde{\sigma}}^{(k)},\widehat{\bm{\Sigma_0}}^{(k)})$
      \If{$k>0$}
      \If{$\ell^{(k)}-\ell^{(k-1)} < min\_delta\_loglik= 10^{-4}$}
      \State \Break
      \EndIf
      \EndIf
      \State $\widehat{\Tilde{\sigma}}^{(k+1)},\widehat{\bm{\Sigma_0}}^{(k+1)} \gets  L-BFGS-B\qty(x_0=(\widehat{\Tilde{\sigma}}^{(k)},\widehat{\bm{\Sigma_0}}^{(k)}),func = \ell^{\widehat{\bm \psi}^{(k)},  \widehat{\bm\phi}^{(k)},\widehat{\bm{\breve{\kappa}}}^{(k)}}_{\bm y},jac =D\ell^{\widehat{\bm \psi}^{(k)},\widehat{\bm\phi}^{(k)},\widehat{\bm{\breve{\kappa}}}^{(k)}}_{\bm y})$
      \State $\widehat{\bm{\phi}}^{(k+1)} \gets	\widehat{\bm{\phi}}\qty(\bm{y},\widehat{\bm{\psi}}^{(k)},\widehat{\bm{\breve{\kappa}}}^{(k)},\widehat{\Tilde{\sigma}}^{(k+1)},\widehat{\bm{\Sigma_0}}^{(k+1)})$
      
      \State $\widehat{\bm{\breve{\kappa}}}^{(k+1)} \gets  \widehat{\bm{\breve{\kappa}}}\qty(\bm{y},\widehat{\bm{\psi}}^{(k)}\widehat{\bm{\phi}}^{(k+1)},\widehat{\Tilde{\sigma}}^{(k+1)},\widehat{\bm{\Sigma_0}}^{(k+1)})$
      \State $\hat{r}\gets \norm{\hat{\breve{\bm\kappa}}^{(k+1)}-\frac1{N+1}\sum\limits_{\nu = 0}^{N}\widehat{\breve{\kappa}}_{\nu}^{(k+1)}}$
      \State $\widehat{\bm\phi}^{(k+1)}, \widehat{\bm{\breve{\kappa}}}^{(k+1)}\gets \widehat{r}\widehat{\bm{\phi}}^{(k+1)},\frac{\widehat{\bm{\breve{\kappa}}}^{(k+1)}}{\widehat{r}}$
      \State $\widehat{\bm{\psi}}^{(k+1)} \gets  Nelder-Mead\qty(x_0 = \widehat{\bm{\psi}}^{(k)},func = \ell^{\widehat{\bm \phi}^{(k+1)},\widehat{\bm{\breve{\kappa}}}^{(k+1)},\widehat{\Tilde{\sigma}}^{(k+1)},\widehat{\bm{\Sigma_0}}^{(k+1)}}_{\bm y})$
      \If{$k\geq start\_c\_opt= 25$}
      \State $\Delta_{c}\gets Nelder-Mead\qty(x_0 =  0, func = \ell^{\widehat{\bm\theta_{c}}}_{\bm y})$
      \State $\widehat{\bm{\psi}}^{(k+1)},\widehat{\bm{\breve{\kappa}}}^{(k+1)}\gets\widehat{\bm{\psi}}^{(k+1)}-\widehat{\bm{\phi}}^{(k+1)}\Delta_{c},\widehat{\bm{\breve{\kappa}}}^{(k+1)}+\Delta_{c}\bm1_{N+1}$
      \EndIf
      
      \State $k\gets k+1$
      \EndWhile
      \State$\hat{c}, \widehat{\bm{\kappa}}\gets \frac1{N+1}\sum\limits_{\nu = 0}^{N}\widehat{\breve{\kappa}}_{\nu}^{(k)},\qty(\widehat{\bm{\breve{\kappa}}}^{(k)}-\hat{c}\bm 1_{N+1})$
      \State \Return $\widehat{\bm\psi},\widehat{\bm\phi}^{(k)},\widehat{\bm\kappa},\hat{c},\widehat{\Tilde{\sigma}},\widehat{\bm\Sigma_0}^{(k)}$
    \end{algorithmic}
  \end{algorithm}

\end{appendices}

\end{document}